\documentclass[11pt]{article}

\usepackage[a4paper,margin=1.15in]{geometry}
\usepackage{amsmath,amssymb,amsthm,mathtools}
\usepackage[numbers,sort&compress]{natbib}
\usepackage{microtype}
\usepackage[hidelinks]{hyperref}

\newtheorem{theorem}{Theorem}[section]
\newtheorem{proposition}[theorem]{Proposition}
\newtheorem{lemma}[theorem]{Lemma}
\newtheorem{corollary}[theorem]{Corollary}
\newtheorem{remark}[theorem]{Remark}
\theoremstyle{definition}
\newtheorem{definition}[theorem]{Definition}

\newcommand{\dd}{\,\mathrm d}
\newcommand{\R}{\mathbb R}
\newcommand{\T}{\mathbb T}

\title{Projection--Lift Equivalence and Dissipation-Tight Compactness\\
for Continuum-State FENE--Markov Fluids}
\author{Sai Peng\\
School of Mathematics and Computational Science, Xiangtan University\\
Xiangtan, Hunan, China\\
\texttt{pscfd@xtu.edu.cn}}
\date{}

\hypersetup{
  pdftitle={Projection--Lift Equivalence and Dissipation-Tight Compactness for Continuum-State FENE--Markov Fluids},
  pdfauthor={Sai Peng},
  pdfsubject={Projection-lift equivalence, direct state-resolved compactness, and global weak solutions for continuum-state FENE--Markov systems},
  pdfkeywords={FENE dumbbell, zero diffusion, direct compactness, full drag, relative entropy, continuum-state Markov reaction}
}

\begin{document}
\maketitle

\begin{abstract}
We consider incompressible FENE dumbbells coupled to a reversible Markov
operator on a compact continuum of internal states.  At zero centre-of-mass
diffusion we prove that state averaging is an exact factor map and that a
Lagrangian state lift is its unique inverse on the natural energy-solution
classes.  Thus existence, multiplicity, and uniqueness of the state-resolved
system are precisely those of its scalar FENE projection; the internal-state
dynamics creates no additional large-data nonuniqueness.  The lift is driven
by a trace-free matrix fibre evolution and permits nonlinear local activities,
including rates with linear dependence on the singular Kramers stress.

We also prove a sequential form of this structure.  For state-resolved
regularizations with relative-entropy-prepared initial fibres and no viscous
dissipation defect, the complete densities converge strongly without
reconstructing them after the scalar limit.  Consequently the full drag,
state-dependent activity, singular stress, and non-atomic Jeffreys production
all pass to the limit.  The argument combines stability of regular Lagrangian
flows with a relative-entropy estimate for simultaneously varying matrix
drifts and jump rates.  A stationary oscillation shows that the preparation
cannot follow from the natural entropy bounds alone.  With positive
centre-of-mass diffusion we independently construct global large-data weak
solutions and identify the same nonlinear terms.  An explicit infinite-rank
kernel proves that these results are not finite-species reductions.
\end{abstract}

\paragraph{Keywords.}
FENE dumbbells; direct compactness; zero diffusion; full drag; relative
entropy; continuum-state reactions.

\paragraph{MSC 2020.}
35K65; 35Q84; 60J27; 76A10.

\section{Introduction}

The FENE dumbbell system couples an incompressible fluid to a probability
density on a bounded configuration domain.  Its decisive analytic obstruction
is the singular spring force at finite extension: weak convergence of the
density does not identify the Kramers stress.  This obstruction persists at
large data and without centre-of-mass diffusion.  The scalar theory was built
through local and perturbative results
\citep{jourdain2004existence,lin2007micromacro,masmoudi2008wellposedness},
the corotational construction of Lions and Masmoudi
\citep{lions2007micromacro}, the positive-diffusion theory of Barrett and
S\"uli \citep{barrett2011existence}, and Masmoudi's defect-propagation method
for the nondiffusive full-drag system \citep{masmoudi2013global}.  The FENE
potential originates in Warner's kinetic model \citep{warner1972kinetic} and
the bead--spring constitutive theory of Bird, Dotson, and Johnson
\citep{bird1980polymer}.

This paper asks whether the scalar large-data theory survives an infinite
internal-state resolution.  We replace the scalar density by a density over a
compact non-atomic state space and couple the states through a reversible
Markov operator whose activity may depend on local moments.  Finite-state
reactive polymer systems were studied in
\citep{wang2021twospecies,liu2022global}, and reversible jump processes carry
an entropy-gradient structure \citep{maas2011gradient,mielke2014relation}.
The continuum-state problem is different in two respects: the rate depends on
a density that is only weakly compact, and the Jeffreys production lives on a
non-atomic fibre.  Treating each label by a scalar existence theorem does not
identify either term.

Our central result is a structural equivalence rather than a new conditional
choice of scalar base solution.  At zero centre-of-mass diffusion, state
averaging sends every continuum-state energy solution to a scalar FENE energy
solution.  Conversely, regular-Lagrangian pullback and a trace-free matrix
fibre evolution assign to every scalar solution exactly one state-resolved
lift with the prescribed conditional initial law.  These maps are inverse.
Hence the state-resolved solution space and the scalar solution space are in
bijection; in particular, internal-state resolution preserves existence and
multiplicity and creates no additional source of nonuniqueness.  This is
Theorem~\ref{thm:projection-lift-equivalence} and Corollary
\ref{cor:base-free-existence-multiplicity}.

The fibre evolution is the main analytic input.  Along almost every
DiPerna--Lions trajectory \citep{diperna1989ordinary}, full drag becomes
\[
 B(t,a)=\nabla_xu(t,X(t,a))\in L^2(0,T),
 \qquad \operatorname{tr}B=0.
\]
The matrix is neither bounded nor skew.  Proposition
\ref{prop:full-drag-fibre-evolution} constructs the corresponding
configuration--state evolution, proves positivity and exact \(L^1\)
contraction, and controls the entropy--stretching balance by a FENE Hardy
estimate.  A Bielecki--Volterra argument then closes nonlinear local activity
on the full time interval.  State-blind drivers need only be integrable along
the scalar projection, which allows linear dependence on the complete
singular Kramers stress; state-sensitive moments require Lipschitz control only
on their accessible compact range.

The equivalence has a direct sequential counterpart.  Theorem
\ref{thm:direct-state-resolved-compactness} starts with the state-resolved
regularizations themselves.  If the initial fibres are prepared in relative
entropy and the viscous dissipation has no defect, stability of regular
Lagrangian flows makes the pulled-back deformation histories converge.  A
relative-entropy estimate for two fibres with simultaneously varying matrix
drifts and jump rates then yields strong convergence of the complete
state-resolved densities.  It identifies full drag, nonlinear activity,
singular stress, and Jeffreys production without projecting first and lifting
afterwards.  Corollary~\ref{cor:viscous-defect-criterion} characterizes the
gradient assumption by equality of viscous dissipation.  The stationary
oscillation of Proposition
\ref{prop:v38-local-moment-oscillation-obstruction} shows that the natural
entropy bounds alone cannot imply the required preparation.

For positive centre-of-mass diffusion, Theorem
\ref{thm:v35-large-data-coupled-existence} constructs global large-data weak
solutions on bounded Lipschitz domains in dimensions two and three.  Spatial
Fisher information compactifies the finitely many activity moments, while the
full state density remains weak.  Strong compactness of its state average,
combined with a weighted Hardy estimate, identifies the singular stress; lower
semicontinuity on a fixed reference measure passes the non-atomic reaction
production.  The construction closes the velocity, state density, and
self-consistent moments in one approximation.  Proposition
\ref{prop:infinite-rank-state-kernel} supplies an admissible Markov generator
with infinitely many non-zero modes, so this theorem is not a disguised
finite-species system.

The three regimes are summarized in Theorem
\ref{thm:master-closure-classification}: positive diffusion gives compactness
of local moments; zero diffusion gives an exact constructive factorization;
and dissipation-tight zero-diffusion sequences give direct strong compactness.
Section~2 states the model and the classification.  Section~3 proves the
positive-diffusion theorem.  Section~4 develops the matrix fibre theory,
projection--lift equivalence, and direct compactness.  The appendix records
the sharper corotational consequences and the compactness obstruction.

We make no claim of three-dimensional large-data uniqueness for the scalar
Navier--Stokes--FENE equations, and we do not obtain direct state-resolved
compactness from entropy bounds alone.  What is proved is exact: within the
stated energy classes, uniqueness of the continuum-state system is equivalent
to uniqueness of its scalar projection; and within prepared, dissipation-tight
regularizations, the state-resolved sequence converges without reconstruction.

\section{Continuum-state model}
\label{sec:continuum-state}

Let \(\Omega\) denote either a bounded Lipschitz domain or the torus specified
in each theorem.  For the Warner potential,
\begin{equation}
 D=B_{\sqrt b}(0),\qquad
 U(q)=-\frac b2\log\left(1-\frac{|q|^2}{b}\right),\qquad
 M(q)=Z_b^{-1}e^{-U(q)} .
 \label{eq:fene-potential}
\end{equation}
More general finite-extension potentials are introduced in
Section~\ref{sec:v35-large-data-coupled}.

\begin{definition}[Admissible finite-extension Maxwellian]
\label{def:admissible-Maxwellian}
Let \(D\subset\R^d\) be bounded and \(C^2\), with \(0\in D\), and put
\(s_D(q)=\operatorname{dist}(q,\partial D)\).  A probability density
\(M=Z^{-1}e^{-U}\) is called admissible if \(M>0\) in \(D\),
\(U\in C^2_{\rm loc}(D)\), and there are \(s_0>0\), \(\gamma>1\), and
positive constants \(c_j\) such that, on \(\{s_D<s_0\}\),
\begin{equation}
 c_1s_D^\gamma\le M\le c_2s_D^\gamma,
 \qquad c_3\le s_D|\nabla_qU|\le c_4.
 \label{eq:admissible-Maxwellian-collar}
\end{equation}
We also require \(\nabla_qU\otimes q\in L^2_{\rm loc}(D)\) and
\(\nabla_qU\) bounded on every set compactly contained in \(D\).
The constants in estimates for this class may depend on
\(D,M,s_0,\gamma\),
but not on the solution.
\end{definition}

The class is not radial.  It contains the Warner Maxwellian, the
Barrett--S\"uli radial class used below, and smooth anisotropic barriers on
bounded \(C^2\) configuration domains.  For example, if \(A\) is positive
definite, then
\[
 D_A=\{q:q\cdot Aq<b\},\qquad
 U_A(q)=-\frac b2\log\left(1-\frac{q\cdot Aq}{b}\right),\qquad b>2,
\]
is admissible.  Indeed, a linear change of variables reduces its boundary
estimates to the Warner case.  The point of Definition
\ref{def:admissible-Maxwellian} is that the full-drag fibre analysis uses only
the collar geometry and the weighted Hardy estimate proved in Lemma
\ref{lem:FENE-Hardy-tail}, not radial symmetry.

Let \((Y,d_Y)\) be a compact metric space and \(\pi\) a Borel probability
measure with full support.  Assume
\begin{equation}
 k\in L^\infty(Y\times Y;\pi\otimes\pi),\qquad
 k(y,y')=k(y',y)\ge0 .
 \label{eq:continuum-kernel}
\end{equation}
A finite-dimensional parametrization or differentiable structure on \(Y\) is
never used; in particular, compact infinite-dimensional state spaces are
included.
A spectral gap,
\begin{equation}
 \|v-\langle v\rangle_\pi\|_{L^2_\pi}^2
 \le\lambda_Y^{-1}\mathcal E_k(v),\qquad
 \mathcal E_k(v)=\frac12\int_{Y^2}k(y,y')|v'-v|^2\dd\pi\dd\pi',
 \label{eq:continuum-spectral-gap}
\end{equation}
may be imposed for quantitative state mixing but is not needed for the
existence results below.

Fix \(N_\eta<\infty\) and
\begin{equation}
 c_a\in L^\infty(Y),\qquad
 \psi_a\in W^{1,\infty}(D),\qquad 1\le a\le N_\eta .
 \label{eq:continuum-observables}
\end{equation}
For \(f=Mh\), define
\begin{equation}
 \eta_a[h](t,x)=\int_{D\times Y}c_a(y)\psi_a(q)
 (h(t,x,q,y)-1)M(q)\dd q\dd\pi(y).
 \label{eq:continuum-moments}
\end{equation}
The symmetric activity is measurable in \((q,y,y')\), continuous in
\(z\in\R^{N_\eta}\), and satisfies
\begin{align}
 &\mathfrak a(q,z,y,y')=\mathfrak a(q,z,y',y),\qquad
 0<a_*\le\mathfrak a\le a^*,
 \label{eq:continuum-activity-bounds}\\
 &|\mathfrak a(q,z,y,y')-\mathfrak a(q,\widetilde z,y,y')|
 \le A|z-\widetilde z|.
 \label{eq:continuum-activity-lipschitz}
\end{align}
Set
\begin{align}
 K_h(t,x,q,y,y')&=k(y,y')\mathfrak a(q,\eta[h](t,x),y,y'),
 \label{eq:continuum-effective-kernel}\\
 (\mathcal R_hh)(y)&=\int_YK_h(y,y')[h(y')-h(y)]\dd\pi(y').
 \label{eq:continuum-reaction-operator}
\end{align}
The associated Jeffreys production is
\begin{equation}
 \mathcal D_Y(Mh)=\frac12\int_{\Omega\times D\times Y^2}
 MK_h(h'-h)(\log h'-\log h)\dd\pi\dd\pi'\dd q\dd x .
 \label{eq:continuum-reaction-entropy}
\end{equation}

\begin{lemma}[Continuum detailed balance]
\label{lem:continuum-reaction-thermodynamics}
For every non-negative \(h\),
\(\int_Y\mathcal R_hh\dd\pi=0\).  For positive \(h\),
\begin{align}
 -\int_Y(\mathcal R_hh)\log h\dd\pi
 &=\frac12\int_{Y^2}K_h(h'-h)(\log h'-\log h)\dd\pi\dd\pi',
 \label{eq:continuum-entropy-identity}\\
 -\int_Y(\mathcal R_hv)v\dd\pi
 &=\frac12\int_{Y^2}K_h|v'-v|^2\dd\pi\dd\pi'
 \ge a_*\lambda_Y\|v-\langle v\rangle_\pi\|_2^2
 \label{eq:continuum-quadratic-coercivity}
\end{align}
whenever the gap \eqref{eq:continuum-spectral-gap} is assumed.  Moreover,
\begin{equation}
 \mathcal D_Y(Mh)\ge4a_*\int_{\Omega\times D}
 M\mathcal E_k(\sqrt h)\dd q\dd x .
 \label{eq:continuum-square-root-coercivity}
\end{equation}
\end{lemma}

\begin{proof}
Interchanging \(y\) and \(y'\) and averaging the two oriented expressions
gives mass conservation and the factor \(1/2\).  The quadratic estimate follows
from \eqref{eq:continuum-activity-bounds} and
\eqref{eq:continuum-spectral-gap}.  Finally,
\((r-s)(\log r-\log s)\ge4|\sqrt r-\sqrt s|^2\).
\end{proof}

\begin{proposition}[An admissible reaction with infinitely many state modes]
\label{prop:infinite-rank-state-kernel}
Let \(Y=\T^1\), let \(\pi\) be normalized Lebesgue measure, and set
\begin{equation}
 k_\infty(y,y')
 =1+\sum_{m=1}^\infty 2^{-m}\cos\bigl(2\pi m(y-y')\bigr).
 \label{eq:infinite-rank-state-kernel}
\end{equation}
Then \(k_\infty\) is continuous, symmetric, non-negative and bounded by two,
so it satisfies \eqref{eq:continuum-kernel}.  The associated constant-activity
generator
\[
 (\mathcal L_\infty v)(y)
 =\int_Y k_\infty(y,y')[v(y')-v(y)]\dd\pi(y')
\]
has infinitely many distinct non-zero eigenvalues.  Consequently it has
infinite rank on \(L^2_\pi(Y)\) and is not linearly conjugate to a Markov
generator on finitely many internal states.
\end{proposition}

\begin{proof}
Uniform convergence of the series gives continuity and symmetry.  Since
\(\sum_{m\ge1}2^{-m}=1\), one has \(0\le k_\infty\le2\).  Moreover
\(\int_Yk_\infty(y,y')\dd\pi(y')=1\).  For
\(e_n(y)=e^{2\pi iny}\), \(n\ne0\), the convolution identity gives
\[
 \int_Y k_\infty(y,y')e_n(y')\dd\pi(y')
 =2^{-|n|-1}e_n(y).
\]
Hence
\[
 \mathcal L_\infty e_n
 =-(1-2^{-|n|-1})e_n,
 \qquad n\in\mathbb Z\setminus\{0\}.
\]
The eigenvalues indexed by \(|n|\ge1\) are non-zero and pairwise distinct.
A generator on \(N<\infty\) states acts on an \(N\)-dimensional space and
therefore has finite rank and only finitely many eigenvalues.  Similarity
preserves rank and spectrum, which proves the last assertion.
\end{proof}

\begin{lemma}[Two compactness tools]
\label{lem:compactness-tools}
Let \(Q_T=(0,T)\times\Omega\times D\), where \(\Omega\) is bounded
Lipschitz.  Let \(h_n\) be uniformly integrable in
\(L^1_M(Q_T\times Y)\), with a modulus that is also tight in boundary collars
of \(D\).  Assume either that \(h_n\ge0\) and
\[
 \int_{Q_T\times Y}M\bigl(|\nabla_x\sqrt{h_n}|^2
 +|\nabla_q\sqrt{h_n}|^2\bigr)\dd\pi\le C,
\]
or that
\[
 \int_{Q_T\times Y}M\bigl(|\nabla_xh_n|+|\nabla_qh_n|\bigr)\dd\pi\le C.
\]
Suppose also that \(\partial_t(M\bar h_n)\) is bounded in
\(L^1(0,T;(W^{1,\infty}(\Omega\times D))')\), where
\(\bar h_n=\int_Yh_n\dd\pi\).  Then \(M\bar h_n\) is relatively compact
in \(L^1(Q_T)\).

Let, in addition, \(v_n\) be uniformly bounded in \(L^r(Q_T^x)\) for some
\(r>1\),
where \(Q_T^x=(0,T)\times\Omega\), and assume
\[
 \|\nabla_xv_n\|_{L^1(Q_T^x)}
 +\|\partial_tv_n\|_{L^1(0,T;(W^{1,\infty}(\Omega))')}\le C.
\]
Then \(v_n\) is relatively compact in \(L^p(Q_T^x)\) for every
\(1\le p<r\).
\end{lemma}

\begin{proof}
In the first alternative, Cauchy--Schwarz in \(Y\) gives
\(|\nabla\sqrt{\bar h_n}|^2\le\int_Y|\nabla\sqrt{h_n}|^2\dd\pi\).
On each Lipschitz set \(D_\varepsilon\Subset D\), the sublevels of
\[
 g\mapsto \|\sqrt g\|_{H^1(\Omega\times D_\varepsilon)}^2
       +\|g\|_{L^1(\Omega\times D_\varepsilon)}
\]
are compact in \(L^1\): apply Rellich to the square roots and use
\(\|g-\widetilde g\|_1\le
\|\sqrt g-\sqrt{\widetilde g}\|_2
\|\sqrt g+\sqrt{\widetilde g}\|_2\).  In the second alternative, local
equivalence of \(M\dd x\dd q\) and Lebesgue measure reduces the claim to
\(W^{1,1}(\Omega\times D_\varepsilon)\Subset L^1\).  Compactness of either
family of sublevels and
the injective embedding into \((W^{1,\infty})'\) give, for every
\(\epsilon>0\),
\[
 \|g-\widetilde g\|_1\le
 \epsilon\{1+\mathcal A(g)+\mathcal A(\widetilde g)\}
 +C_{\epsilon,\varepsilon}
 \|g-\widetilde g\|_{(W^{1,\infty})'} .
\]
Applying this inequality to time translates and integrating in time proves
local \(L^1\) compactness by the Kolmogorov--Riesz criterion; this is the
compact-sublevel form of the Aubin--Lions--Simon argument
\citep{simon1987compact}.  The assumed boundary-collar tightness removes the
configuration cutoff.  In the entropy application below it follows from the
Fenchel inequality
\[
 \lambda h_+\mathbf1_E\le F(h_+)+e^{\lambda\mathbf1_E}-1
\]
gives, after integration against \(M\dd q\dd\pi\dd x\),
\[
 \lambda\int_{\Omega\times E\times Y}M[h_n]_+
 \le C+(e^\lambda-1)|\Omega|\int_E M.
\]
Choose \(E=D\setminus D_\varepsilon\), then first choose \(\lambda\) and
then \(\varepsilon\).  The negative parts are handled by their assumed
uniform-integrability modulus.  This makes the boundary strip uniformly small
and globalizes the compactness.

For the second assertion, bounded subsets of \(BV(\Omega)\) are compact in
\(L^1(\Omega)\), and the
same compact-sublevel argument gives
\[
 \|v-\widetilde v\|_1\le
 \epsilon\bigl(1+\|v\|_{BV}+\|\widetilde v\|_{BV}\bigr)
 +C_\epsilon\|v-\widetilde v\|_{(W^{1,\infty})'}.
\]
Spatial translations are controlled by the \(BV\) seminorm and time
translations by this inequality and the derivative bound.  Kolmogorov--Riesz
therefore gives strong \(L^1\) compactness.  Interpolation with the uniform
\(L^r\) bound gives every \(L^p\), \(1\le p<r\).
\end{proof}

\begin{lemma}[Finite-extension Hardy bound and stress tails]
\label{lem:FENE-Hardy-tail}
Let \((D,M)\) be an admissible finite-extension Maxwellian in the sense of
Definition~\ref{def:admissible-Maxwellian}, and write
\(W=\nabla_qU\otimes q\).  For every non-negative \(g\) with
\(\sqrt g\in H^1_M(D)\),
\begin{equation}
 \int_D M|W|^2g\dd q
 \le C_U\int_D M\bigl(g+|\nabla_q\sqrt g|^2\bigr)\dd q.
 \label{eq:FENE-Hardy-pointwise}
\end{equation}
Consequently, if \(\rho_g(t,x)=\int_{D\times Y}Mg\dd q\dd\pi\le R\), then
\begin{align}
 \left\|\int_{D\times Y}MWg\dd q\dd\pi\right\|_{L^2_{t,x}}^2
 &\le C_UR\int M(g+|\nabla_q\sqrt g|^2),
 \label{eq:FENE-stress-L2}\\
 \int M|W|g\mathbf1_{\{|W|>A\}}
 &\le \frac{C_U}{A}\int M(g+|\nabla_q\sqrt g|^2).
 \label{eq:FENE-stress-tail}
\end{align}
The same estimates hold with \(g\) replaced by any \(0\le\beta(g)\le g\)
in the stress integrals.
\end{lemma}

\begin{proof}
Only a collar of \(\partial D\) is singular.  The normal-coordinate map
\((p,s)\mapsto p-sn(p)\) is a \(C^1\) diffeomorphism on a sufficiently thin
collar and has Jacobian bounded above and below.  A finite tangential partition
of unity and \eqref{eq:admissible-Maxwellian-collar} reduce
\eqref{eq:FENE-Hardy-pointwise}, up to bounded Jacobians and interior terms,
to the one-dimensional weighted Hardy estimate
\[
\int_0^{s_0}s^{\gamma-2}|z|^2\dd s
\le C_\gamma\int_0^{2s_0}s^\gamma|z'|^2\dd s
+C_\gamma\int_{s_0}^{2s_0}|z|^2\dd s,
\]
which holds exactly for \(\gamma>1\).  Here is the complete absorption step.
Choose \(\chi\in C_c^1([0,2s_0))\), with \(\chi=1\) on \([0,s_0]\) and
\(|\chi'|\le2/s_0\), and put \(w=\chi z\).  Initially for smooth \(z\),
integration by parts gives
\[
 (\gamma-1)\int_0^{2s_0}s^{\gamma-2}|w|^2\dd s
 =-2\int_0^{2s_0}s^{\gamma-1}ww'\dd s.
\]
The boundary term at \(2s_0\) vanishes by the cutoff and the term at zero
vanishes because \(\gamma>1\).  Weighted Cauchy--Schwarz therefore yields
\[
 \int_0^{2s_0}s^{\gamma-2}|w|^2\dd s
 \le \frac4{(\gamma-1)^2}
 \int_0^{2s_0}s^\gamma|w'|^2\dd s.
\]
Expanding \(w'=\chi z'+\chi'z\) proves the displayed one-dimensional
estimate, with a constant depending only on \(\gamma\) and \(s_0\).
Weighted density extends it to the energy class.  Since \(D\) is bounded and
\(|\nabla U|\le c_4/s_D\) on the collar, integration over the tangential
coordinates, followed by an ordinary interior estimate, proves
\eqref{eq:FENE-Hardy-pointwise}.
For \eqref{eq:FENE-stress-L2}, apply Cauchy--Schwarz in \((q,y)\) at each
\((t,x)\):
\[
 \left|\int MWg\right|^2
 \le \rho_g\int M|W|^2g
 \le C_UR\int M(g+|\nabla_q\sqrt g|^2).
\]
Integration proves \eqref{eq:FENE-stress-L2}.  Finally
\(|W|\mathbf1_{\{|W|>A\}}\le A^{-1}|W|^2\), which proves
\eqref{eq:FENE-stress-tail}.  The cutoff assertion follows from
\(\beta(g)\le g\) directly in the stress integrals.
\end{proof}

\begin{lemma}[Closed drift--stress work pairing]
\label{lem:closed-drift-stress-pairing}
Let \((D,M)\) be admissible, let \(B\in\R^{d\times d}\) satisfy
\(\operatorname{tr}B=0\), and let \(g\ge0\) have finite mass with
\(\sqrt g\in H^1_M(D\times Y)\).  Then the distribution
\(M^{-1}\nabla_q\cdot(BqMg)\) can be paired with the entropy variable by
closure from compactly supported configuration tests, and
\begin{equation}
 \int_{D\times Y}MBq\cdot\nabla_qg
 =B:\int_{D\times Y}M(\nabla_qU\otimes q)g.
 \label{eq:closed-drift-stress-pairing}
\end{equation}
Moreover,
\begin{equation}
 \left|B:\int M(\nabla_qU\otimes q)g\right|
 \le \eta\int M|\nabla_q\sqrt g|^2
 +C_{D,M,\eta}|B|^2\int Mg+C_{D,M}|B|\int Mg
 \label{eq:closed-drift-stress-bound}
\end{equation}
for every \(\eta>0\).
\end{lemma}

\begin{proof}
Take the cutoffs \(\chi_\delta\) used below.  On the support of
\(\nabla\chi_\delta\), Definition~\ref{def:admissible-Maxwellian} gives
\[
 \frac{|q|}{\delta}\le \frac{2|q|}{s_D}
 \le \frac{2}{c_3}\,|q|\,|\nabla_qU|,
\]
so the cutoff error is controlled by
\(C|B|\int_{\{s_D<2\delta\}}M|W|g\), where
\(W=\nabla_qU\otimes q\).  Lemma~\ref{lem:FENE-Hardy-tail} makes
\(M|W|g\) integrable; absolute continuity sends this error to zero.
On the support of \(\chi_\delta\), ordinary integration by parts and
\(\nabla_q M=-M\nabla_q U\) give
\[
 \int M\chi_\delta Bq\cdot\nabla g
 =B:\int M\chi_\delta(\nabla_q U\otimes q)g+o(1),
\]
because \(\operatorname{tr}B=0\).  Passing to the limit proves
\eqref{eq:closed-drift-stress-pairing}.  The Hardy estimate followed by
Cauchy--Schwarz and Young gives
\eqref{eq:closed-drift-stress-bound}.  Weighted density of smooth square
roots proves that the identity is independent of the approximating sequence,
which is the asserted closure.
\end{proof}

\begin{lemma}[Fixed-measure Jeffreys liminf]
\label{lem:fixed-measure-Jeffreys-liminf}
Let \(\mu\) be a finite measure, let \(a_n,a\) be measurable with
\(0<a_*\le a_n,a\le a^*\), and suppose \(a_n\to a\) in measure.  If
\((r_n,s_n)\rightharpoonup(r,s)\) weakly in \(L^1(\mu;\R^2)\), then
\[
 \int a(r-s)(\log r-\log s)\dd\mu
 \le\liminf_n\int a_n(r_n-s_n)(\log r_n-\log s_n)\dd\mu,
\]
with the non-negative lower-semicontinuous extension at zero.
\end{lemma}

\begin{proof}
The Jeffreys integrand \(J(r,s)=(r-s)(\log r-\log s)\) is jointly convex
and lower semicontinuous on \([0,\infty)^2\).  Indeed, in the positive
quadrant,
\[
 D^2J(r,s)=(r+s)
 \begin{pmatrix}r^{-2}&-(rs)^{-1}\\-(rs)^{-1}&s^{-2}\end{pmatrix},
\]
which is positive semidefinite and has determinant zero.  Along a subsequence realizing the
liminf, take \(a_n\to a\) almost everywhere.  Egorov's theorem gives sets
\(E_j\uparrow\operatorname{supp}\mu\), up to a null set, on which the
convergence is uniform.  On \(E_j\), for all sufficiently large \(n\),
\(a_n\ge(1-\epsilon)a\); weak lower semicontinuity of the convex integral
on the fixed set \(E_j\) yields the desired bound there, up to
\(1-\epsilon\).  This is also a direct instance of Ioffe's
fixed-coefficient integral lower-semicontinuity theorem
\citep{ioffe1977lower}.  Let \(n\to\infty\), then \(j\to\infty\), and finally
\(\epsilon\downarrow0\), using non-negativity and monotone convergence.
\end{proof}

\begin{theorem}[Diffusion--drag--locality closure classification]
\label{thm:master-closure-classification}
Let the internal state space, reversible reference measure, and symmetric jump
 kernel satisfy \eqref{eq:continuum-kernel}.  The following seven statements
hold.
\begin{enumerate}
\item[\textup{(D)}] If the centre-of-mass diffusivity is positive, then on
  every bounded Lipschitz domain in dimensions two and three the fully coupled
  FENE--Navier--Stokes system admits a global large-data energy--entropy weak
  solution for every bounded uniformly positive activity depending
  Lipschitz-continuously on finitely many bounded local moments and every datum
  satisfying \eqref{eq:v35-large-data-initial-class}.  The nonlinear
  activity, singular Kramers stress, and Jeffreys production are all
  identified in the limit.  The class contains the infinite-rank
  continuous-state generator of Proposition
  \ref{prop:infinite-rank-state-kernel}, so it is not a reformulation of
  scalar FENE existence through a linear finite-species state map.
\item[\textup{(F)}] If the centre-of-mass diffusivity is zero, then for full
  drag on \(\T^d\), or on a bounded \(C^2\) domain, \(d=2,3\), the
  continuum-state system admits a global large-data weak solution for every
  bounded Lipschitz local finite-moment activity.  Uniform positivity is not
  required: \(0\le\mathfrak a\le a^*\) is allowed.  No scalar trajectory is
  prescribed in the data.  The proof factors the nonempty solution set through
  a trace-free matrix fibre evolution over its scalar projections; its
  nonlinear activity, Kramers stress, and Jeffreys production are identified.
  Projection and lift are inverse, so the state-resolved and scalar solution
  sets have equal multiplicity; over every projection the state lift and
  self-consistent moment field are unique.  The data satisfy
  \eqref{eq:full-drag-state-data-a}--\eqref{eq:full-drag-scalar-log2-data}.
\item[\textup{(S)}] In the full-drag zero-diffusion class, the activity may
  additionally depend on state-blind observables that are merely integrable
  along the intrinsic scalar projection.  In particular, it may depend on the
  singular Kramers stress and grow linearly in that stress.  The resulting reaction
  rate belongs to \(L^1_{t,x}(L^\infty_{q,y,y'})\) but need not be bounded.
\item[\textup{(C)}] Let a sequence of zero-diffusion, state-resolved
  full-drag regularizations have unit number density, uniformly positive
  bounded Lipschitz local activities, well-prepared state data, and no
  viscous-gradient defect.  Then the complete state densities converge
  strongly, uniformly in time in fibrewise total variation.  The local
  activities, full-drag product, singular stress, and Jeffreys production are
  identified directly along the sequence.  No scalar projection is used to
  reconstruct the limiting state density.
\item[\textup{(L)}] If the centre-of-mass diffusivity is zero, then on
  \(\T^2\), for corotational drag and Warner parameter \(b>2\), a global weak
  solution exists for every bounded uniformly positive activity depending
  Lipschitz-continuously on finitely many bounded local moments generated by
  arbitrary observables in \(L^\infty(D\times Y)\), provided the datum satisfies
  \eqref{eq:v37-corotational-scalar-class}--
  \eqref{eq:v37-Masmoudi-initial-class}.  For each scalar corotational base
  solution, the continuum-state lift and its self-consistent local moment field
  are unique in the renormalized class with bounded number density.
\item[\textup{(G)}] If the centre-of-mass diffusivity is zero, then on
  \(\T^2\), for corotational drag and Warner parameter \(b>2\), a global weak
  solution exists for every bounded uniformly positive activity depending on
  finitely many globally averaged moments generated by bounded observables,
  provided the datum satisfies
  \eqref{eq:v37-corotational-scalar-class}--
  \eqref{eq:v37-Masmoudi-initial-class}.  For each scalar corotational base
  solution, the continuum-state lift and its self-consistent moment path are
  unique.  The path is absolutely continuous under
  \eqref{eq:v38-generator-assumption-a}--
  \eqref{eq:v38-generator-assumption-b}.
\item[\textup{(O)}] The natural zero-diffusion mass, entropy, number-density,
  configurational-Fisher, and time-derivative bounds do not sequentially close
  a general bounded Lipschitz activity of local moments.  There are stationary
  unit-number-density solutions satisfying all these bounds whose local
  state-sensitive moments
  converge weakly but whose activities converge to a value different from the
  activity evaluated at the weak limit.
\end{enumerate}
Thus positive spatial diffusion closes general bounded local finite-moment
feedback by compactness.  At zero diffusion, the full-drag Lagrangian
construction closes bounded state-sensitive feedback directly and also
admits unbounded, state-blind stress drivers.  It also closes global
finite-dimensional
feedback.  Statement \textup{(C)} gives sequential closure when the viscous
dissipation has no defect.  Statement \textup{(O)} says that the same local
conclusion cannot be obtained from the displayed weak bounds alone; it is a
sequential compactness obstruction, not a non-existence theorem.
\end{theorem}

\begin{proof}
Statement \textup{(D)} is Theorem
\ref{thm:v35-large-data-coupled-existence}, including the identifications in
Steps 3, 4, and 6 of its proof.  Statement \textup{(F)} is
Theorem \ref{thm:full-drag-local-closure}, based on the trace-free fibre
 evolution in Proposition \ref{prop:full-drag-fibre-evolution} and the
 full-drag lift in Theorem \ref{thm:full-drag-state-fibre}.  Statement
 \textup{(S)} is Theorem \ref{thm:integrable-driver-activity} and Corollary
 \ref{cor:unbounded-stress-feedback}.  Statement \textup{(C)} is Theorem
 \ref{thm:direct-state-resolved-compactness} and Corollary
 \ref{cor:viscous-defect-criterion}.  Statement
\textup{(L)} is Theorem
\ref{thm:local-moment-Lagrangian-closure}, based on the label-dependent lift in
Proposition \ref{prop:label-dependent-fibre-lift}.  Statement \textup{(G)} is Theorem
\ref{thm:v38-self-consistent-global-moment-existence}, based on the
Lagrangian state-fibre well-posedness theorem
\ref{thm:lagrangian-state-fibre}.  Statement \textup{(O)} is the stationary
oscillation construction in Proposition
\ref{prop:v38-local-moment-oscillation-obstruction}.
\end{proof}

\section{Large-data coupled continuum-state weak solutions}
\label{sec:v35-large-data-coupled}

We extend the Barrett--S\"uli large-data FENE construction
\citep{barrett2011existence} to a continuum reversible state fibre and bounded
finite-moment activities.  The additional limit passages identify the
nonlinear activity and the continuum Jeffreys production; the algebraic input
is the detailed-balance identity of Lemma
\ref{lem:continuum-reaction-thermodynamics}.

Throughout this section, \(\Omega\subset\R^d\), \(d\in\{2,3\}\), is a
bounded open Lipschitz domain, \(\delta,\nu,{\rm Wi}>0\), and
\((Y,\pi)\), \(k\), \(c_a\), \(\psi_a\), and \(\mathfrak a\) satisfy
\eqref{eq:continuum-kernel}, \eqref{eq:continuum-observables}, and
\eqref{eq:continuum-activity-bounds}--
\eqref{eq:continuum-activity-lipschitz}.  The spectral gap
\eqref{eq:continuum-spectral-gap} is allowed but is not needed for existence.
The lower activity bound is used in the fixed-measure Jeffreys liminf.  We use
the finite-extension potential class of
Barrett and S\"uli: \(D\subset\R^d\) is a ball, the radial potential
\(U(q)=\mathcal U(\frac12|q|^2)\), with \(\mathcal U\in C^2\),
\(\mathcal U(0)=0\), is increasing and unbounded at
\(\partial D\), and, for some \(\gamma>1\) and positive constants \(c_j\),
\begin{align}
 c_1\operatorname{dist}(q,\partial D)^\gamma
 &\le M(q)\le
 c_2\operatorname{dist}(q,\partial D)^\gamma,
 \label{eq:v36-BS-potential-a}\\
 c_3&\le \operatorname{dist}(q,\partial D)
 \mathcal U'(\tfrac12|q|^2)\le c_4.
 \label{eq:v36-BS-potential-b}
\end{align}
The Warner FENE potential \eqref{eq:fene-potential} is the principal example:
then \(D=B_{\sqrt b}(0)\), \(\gamma=b/2\), and
\eqref{eq:v36-BS-potential-a}--\eqref{eq:v36-BS-potential-b} hold exactly
when \(b>2\).  No Hookean potential is included in this bounded-configuration
class.  Because \(|q|\) is bounded away from zero in a sufficiently thin
boundary collar, these radial assumptions imply Definition
\ref{def:admissible-Maxwellian}.

Let
\begin{equation}
 H_\sigma(\Omega):=\overline{C_{c,\sigma}^\infty(\Omega;\R^d)}^{L^2},
 \qquad
 V_\sigma(\Omega):=H^1_0(\Omega;\R^d)\cap H_\sigma(\Omega).
 \label{eq:v36-Leray-spaces}
\end{equation}
Thus the velocity has the no-slip trace.  The Fokker--Planck boundary
conditions are understood in the natural weak sense as
\begin{equation}
 \delta\nabla_xh\cdot n_\Omega=0\quad\hbox{on }\partial\Omega\times D,
 \qquad
 \left((\nabla_xu)qMh-\frac1{2{\rm Wi}}M\nabla_qh\right)\cdot n_D=0
 \quad\hbox{on }\Omega\times\partial D.
 \label{eq:v36-no-flux-boundaries}
\end{equation}
Put
\begin{align}
 \rho[h](t,x)&:=\int_{D\times Y}Mh\dd q\dd\pi,
 \label{eq:v35-number-density}\\
 \tau_Y[h](t,x)&:=\int_{D\times Y}\nabla_qU\otimes q\,Mh
 \dd q\dd\pi.
 \label{eq:v35-large-data-stress}
\end{align}
The isotropic part of the Kramers stress is omitted because its divergence is
absorbed into the pressure.  We use the convention
\((\nabla_xu)_{ij}=\partial_j u_i\) and the Frobenius product
\(A:B=\operatorname{tr}(A^TB)\).  Thus
\((\nabla_xu)q\cdot\nabla_qU=(\nabla_qU\otimes q):\nabla_xu\).
For radial potentials this tensor equals the more customary
\(q\otimes\nabla_qU\); fixing the order is essential for the non-radial class
of Definition~\ref{def:admissible-Maxwellian}.

\begin{definition}[Large-data coupled energy--entropy weak solution]
\label{def:v35-large-data-weak-solution}
Let \(T>0\).  A pair \((u,h)\) is a coupled weak solution on \([0,T]\) if
\(u\in L^\infty(0,T;H_\sigma(\Omega))\cap
L^2(0,T;V_\sigma(\Omega))\),
\(u\in C_w([0,T];H_\sigma(\Omega))\), and
\(h\ge0\) satisfies
\begin{align}
 &Mh\in L^\infty(0,T;L^1(\Omega\times D\times Y)),
 \qquad Mh\log h\in L^\infty(0,T;L^1),
 \label{eq:v35-weak-density-class}\\
 &\sqrt h\in L^2(0,T;H_x^1(L^2_{M\otimes\pi})),
 \qquad
 \nabla_q\sqrt h\in L^2((0,T)\times\Omega\times D\times Y;M),
 \label{eq:v35-weak-fisher-class}\\
 &\mathcal D_Y(Mh)\in L^1(0,T),
 \qquad \tau_Y[h]\in L^2((0,T)\times\Omega).
 \label{eq:v35-weak-reaction-stress-class}
\end{align}
For every divergence-free
\(w\in C^1([0,T];C_{c,\sigma}^\infty(\Omega;\R^d))\),
\begin{align}
 &(u(t),w(t))-(u^{\rm in},w(0))
 -\int_0^t\!\int_{\Omega}u\cdot\partial_rw
 +(u\otimes u):\nabla_xw\dd x\dd r\notag\\
 &\qquad=-\nu\int_0^t\!\int_{\Omega}\nabla_xu:\nabla_xw\dd x\dd r
 -\int_0^t\!\int_{\Omega}\tau_Y[h]:\nabla_xw\dd x\dd r.
 \label{eq:v35-weak-momentum}
\end{align}
For every \(\pi\)-measurable test \(\varphi(t,x,q,y)\) for which
\(\varphi,\partial_t\varphi,\nabla_x\varphi,\nabla_q\varphi\) are bounded and
continuous in \((t,x,q)\) for almost every \(y\),
\begin{align}
 &\int Mh(t)\varphi(t)-\int Mh^{\rm in}\varphi(0)
 -\int_0^t\!\int Mh
 (\partial_r\varphi+u\cdot\nabla_x\varphi
 +(\nabla_xu)q\cdot\nabla_q\varphi)\notag\\
 &\quad=-\delta\int_0^t\!\int M\nabla_xh\cdot\nabla_x\varphi
 -\frac1{2{\rm Wi}}\int_0^t\!\int M\nabla_qh\cdot\nabla_q\varphi\notag\\
 &\qquad+\frac12\int_0^t\!\int_{\Omega\times D\times Y^2}
 MK_h(h'-h)(\varphi-\varphi')
 \dd\pi\dd\pi'\dd q\dd x\dd r.
 \label{eq:v35-weak-kinetic}
\end{align}
Here and below unmarked kinetic integrals include \(\dd\pi(y)\dd q\dd x\),
and derivatives of \(h\) in \eqref{eq:v35-weak-kinetic} are understood as
the \(L^1\) fluxes
\(\nabla h=2\sqrt h\nabla\sqrt h\).  Allowing tests with non-zero boundary
values encodes both conditions in \eqref{eq:v36-no-flux-boundaries}; at
\(\partial D\) no separate trace is assigned to the singular drift or the
diffusion.  We also require
\(Mh\in C_w([0,T];L^1(\Omega\times D\times Y))\).  The initial traces in both
identities are part of the definition.
\end{definition}

Define the total free energy
\begin{equation}
 \mathcal F_Y[u,h](t):=\frac12\|u(t)\|_2^2
 +\int_{\Omega\times D\times Y}M(h\log h-h+1).
 \label{eq:v35-total-free-energy}
\end{equation}

\begin{proposition}[One-step closure of the entropy-cutoff scheme]
\label{prop:one-step-entropy-cutoff}
Fix \(0<\delta_0<1<L\), a time step \(\Delta t>0\), and a solenoidal
Galerkin space \(X_N\).  Fix also a finite Borel partition
\(\mathcal P_m\) of \(Y\), and let \(\mathcal Y_m\subset L^2_\pi(Y)\) be the
space of functions constant on its cells.  Suppose \(u^-\in X_N\),
\(h^-\in L^2_M(\Omega\times D;\mathcal Y_m)\), and
\(F^L_{\delta_0}(h^-)\in L^1_M\).  There is a triple

\[
 u^+\in X_N,\qquad
 h^+\in H^1_M(\Omega\times D;\mathcal Y_m),\qquad
 z^+=\eta[h^+]\in H^1(\Omega;\R^{N_\eta}),
\]

solving the implicit momentum--kinetic step with cutoff
\(\beta^L_{\delta_0}\).  The full stress and reaction kernel are
\[
 \tau[h^+]=\int_{D\times Y}\nabla_qU\otimes q\,Mh^+,
 \qquad K_{z^+}=k\mathfrak a(q,z^+,y,y').
\]
More precisely, for
\(w\in X_N\) and bounded energy tests
\(\phi\in H^1_M(\Omega\times D;\mathcal Y_m)\),
\begin{align}
 &\frac{(u^+-u^-,w)}{\Delta t}+b(u^-,u^+,w)
 +\nu(\nabla u^+,\nabla w)+(\tau[h^+],\nabla w)=0,
 \label{eq:one-step-momentum}\\
 &\frac1{\Delta t}\int M(h^+-h^-)\phi
 -\int Mh^+u^+\cdot\nabla_x\phi
 -\int M\beta^L_{\delta_0}(h^+)(\nabla_xu^+)q\cdot\nabla_q\phi
 \notag\\
 &\quad+\delta\int M\nabla_xh^+\cdot\nabla_x\phi
 +\frac1{2{\rm Wi}}\int M\nabla_qh^+\cdot\nabla_q\phi
 \notag\\
 &\quad+\frac12\int_{\Omega\times D\times Y^2}
 MK_{z^+}(h^{+\prime}-h^+)(\phi'-\phi)=0.
 \label{eq:one-step-kinetic}
\end{align}
Here \(b(v,w,w)=0\) is any consistent skew form for convection.  The solution
satisfies
\begin{align}
 &\frac12\|u^+\|_2^2+\int MF^L_{\delta_0}(h^+)
 +\frac12\|u^+-u^-\|_2^2+\Delta t\,\nu\|\nabla u^+\|_2^2
 \notag\\
 &\quad+\Delta t\,\delta\int M
 \frac{|\nabla_xh^+|^2}{\beta^L_{\delta_0}(h^+)}
 +\frac{\Delta t}{2{\rm Wi}}\int M
 \frac{|\nabla_qh^+|^2}{\beta^L_{\delta_0}(h^+)}
 +\Delta t\,\mathcal D^L_{\delta_0,Y}(h^+)
 \notag\\
 &\le \frac12\|u^-\|_2^2+\int MF^L_{\delta_0}(h^-).
 \label{eq:one-step-entropy}
\end{align}
Writing \([r]_-:=\max\{-r,0\}\), the same solution satisfies
\begin{align}
 &\frac12\|[h^+]_-\|_{L^2_M}^2
 +\Delta t\,\delta\|\nabla_x[h^+]_-\|_{L^2_M}^2
 +\frac{\Delta t}{4{\rm Wi}}
 \|\nabla_q[h^+]_-\|_{L^2_M}^2
 \notag\\
 &\qquad\le\frac12\|[h^-]_-\|_{L^2_M}^2
 +C_b{\rm Wi}\,\delta_0^2\Delta t\|\nabla_xu^+\|_{L^2_x}^2.
 \label{eq:one-step-negative-part}
\end{align}
Here
\begin{equation}
 \mathcal D^L_{\delta_0,Y}(h):=
 \frac12\int_{\Omega\times D\times Y^2}MK_h(h'-h)
 \bigl[(F^L_{\delta_0})'(h')-(F^L_{\delta_0})'(h)\bigr].
 \label{eq:regularized-reaction-production}
\end{equation}
The constants in this estimate are independent of \(m\) and \(N\), and the step
preserves total mass.  If
\(0\le\rho^-:=\int_{D\times Y}Mh^-\in L^\infty(\Omega)\), then
\begin{equation}
 \frac{\rho^+-\rho^-}{\Delta t}+u^+\cdot\nabla_x\rho^+
 =\delta\Delta_x\rho^+,
 \qquad
 0\le\rho^+\le\|\rho^-\|_\infty.
 \label{eq:one-step-density-maximum}
\end{equation}
\end{proposition}

\begin{proof}
Put \(F=F^L_{\delta_0}\), \(\beta=\beta^L_{\delta_0}\).  The entropy
variable \(w=F'(h)\) is globally invertible because

\[
 \delta_0\le\beta(h)=\frac1{F''(h)}\le L.
\]

Choose nested finite-dimensional spaces
\[
 E_R\subset H^1_M(\Omega\times D;\mathcal Y_m)\cap L^\infty,
\]
containing constants and dense in
\(H^1_M(\Omega\times D;\mathcal Y_m)\).  They are tensor products of smooth
\((x,q)\)-functions with the finitely many cell indicators of
\(\mathcal P_m\).  On \(X_N\times E_R\) use the unknowns
\((u_R,w_R)\), put \(h_R=(F')^{-1}(w_R)\), and evaluate the rate at the
actual moment \(\eta[h_R]\).  All maps are continuous in finite dimension.
Let \(\mathcal A_R\) be the residual of
\eqref{eq:one-step-momentum}--\eqref{eq:one-step-kinetic}, represented through
the Euclidean inner product on \(X_N\times E_R\), and let \(\mathcal J_R\) be
the corresponding Riesz isomorphism.  We use the explicit degree homotopy
\begin{equation}
 (1-\lambda)\mathcal J_R(u_R,w_R)
 +\lambda\mathcal A_R(u_R,w_R)=0,
 \qquad0\le\lambda\le1.
 \label{eq:one-step-degree-homotopy}
\end{equation}
The stress and drag terms, as well as the fluid and kinetic transport terms,
are multiplied by the same \(\lambda\); hence their cancellation is retained
along the homotopy.

Test a homotopy solution by \(u_R\) and \(w_R=F'(h_R)\).  The time terms
satisfy
\[
 (u_R-u^-,u_R)\ge\tfrac12
 (\|u_R\|_2^2-\|u^-\|_2^2+\|u_R-u^-\|_2^2)
\]
and
\[
 (h_R-h^-)F'(h_R)\ge F(h_R)-F(h^-).
\]
Spatial transport and the skew convection form vanish.  Since
\(\beta(h_R)\nabla_qF'(h_R)=\nabla_qh_R\), configurational integration by
parts changes the drag contribution into
\(-\tau[h_R]:\nabla u_R\), cancelling the momentum stress work.  Pair
symmetry gives \(\mathcal D^L_{\delta_0,Y}(h_R)\ge0\).  We obtain
\eqref{eq:one-step-entropy}, with the additional non-negative term
\((1-\lambda)(\|u_R\|_{X_N}^2+\|w_R\|_{E_R}^2)\), uniformly in the homotopy
parameter and in \(R\).  The entropy controls the norm of the entropy variable
at the endpoint \(\lambda=1\).  Indeed,
\begin{equation}
 F(h)\ge\frac{|h-1|^2}{2L},\qquad
 |F'(h)|\le\delta_0^{-1}|h-1|,\qquad
 |\nabla F'(h)|^2\le\delta_0^{-1}
 \frac{|\nabla h|^2}{\beta(h)}.
 \label{eq:one-step-entropy-variable-control}
\end{equation}
For \(\lambda\le1/2\) the Riesz term supplies the same bound directly, while
for \(\lambda\ge1/2\) the entropy estimate and
\eqref{eq:one-step-entropy-variable-control} do so.  Thus no homotopy solution
lies on the boundary of one fixed ball in \(X_N\times E_R\).  At
\(\lambda=0\) the map is \(\mathcal J_R\), whose Brouwer degree is one.
Homotopy invariance therefore gives a zero at \(\lambda=1\), namely a coupled
solution \((u_R,h_R)\).

The lower cutoff does not preserve pointwise positivity at fixed
\(\delta_0\), so a quantitative negative-part estimate is needed.  We do not
test the finite-dimensional equation by \(-[h_R]_-\), which need not belong to
\(E_R\).  Instead we first remove the Galerkin dimension using only the entropy
estimate, and perform the Stampacchia test in the limiting one-step equation
below.

At fixed \((m,L,\delta_0,\Delta t)\), the entropy estimate is coercive in the
full weighted \(H^1\) norm:
\[
 \int M|h_R|^2\le C_L\left(1+\int MF(h_R)\right),\qquad
 \int M|\nabla h_R|^2
 \le L\int M\frac{|\nabla h_R|^2}{\beta(h_R)}.
\]
Because \(\mathcal Y_m\) is finite dimensional, the weighted compact embedding
for \((x,q)\) under \eqref{eq:v36-BS-potential-a}--
\eqref{eq:v36-BS-potential-b} gives, after extraction,
\begin{equation}
 h_R\longrightarrow h^+
 \quad\text{strongly in }L^2_M(\Omega\times D\times Y).
 \label{eq:one-step-strong-density-limit}
\end{equation}
The velocity converges strongly in the fixed space \(X_N\).  Since \(F'\) and
\(\beta\) are globally Lipschitz at fixed cutoffs,
\(F'(h_R)\to F'(h^+)\) and \(\beta(h_R)\to\beta(h^+)\) strongly in
\(L^2_M\).  The diffusion gradients converge weakly.  The moment component is
bounded in \(H^1(\Omega)^{N_\eta}\); indeed, weighted Cauchy--Schwarz gives,
pointwise in \(x\),
\begin{equation}
 |\nabla_x\eta[h_R](x)|
 \le C_\eta
 \left(\int_{D\times Y}M\beta(h_R)\right)^{1/2}
 \left(\int_{D\times Y}M
 \frac{|\nabla_xh_R|^2}{\beta(h_R)}\right)^{1/2}.
 \label{eq:one-step-moment-H1}
\end{equation}
Since \(\beta(h_R)\le L\) and \(M\dd q\dd\pi\) is a probability measure,
\eqref{eq:one-step-entropy} yields the claimed \(H^1_x\) bound.  Rellich
compactness gives
\[
 \eta[h_R]\longrightarrow z^+
 \quad\text{strongly in }L^2(\Omega;\R^{N_\eta})
 \quad\text{and almost everywhere}.
\]
The bounded Lipschitz activity therefore converges strongly in every finite
\(L^p\).  The reaction term passes by strong convergence of \(h_R\) and the
bounded coefficient; the drag passes by strong convergence of \(\beta(h_R)\).
Moreover, \(\nabla_qU\otimes q\in L^2_M(D)\) because \(\gamma>1\), so the full
stress passes in \(L^1_x\) by Cauchy--Schwarz.  Weak lower semicontinuity
retains the entropy inequality.  Passing in the
Galerkin identities against dense tests proves
\eqref{eq:one-step-momentum}--\eqref{eq:one-step-entropy}.

The limiting kinetic equation is valid for every bounded
\(H^1_M(\Omega\times D;\mathcal Y_m)\) test.  Approximate
\(-[h^+]_-\) by bounded Lipschitz truncations in this space.  The chain rule
for \(r\mapsto\frac12[r]_-^2\) gives
\[
 (h^+-h^-)(-[h^+]_-)
 \ge\frac12\bigl([h^+]_-^2-[h^-]_-^2\bigr).
\]
Spatial transport vanishes, both diffusion terms are non-negative, and pair
symmetry makes the reaction term non-negative because
\(r\mapsto\min\{r,0\}\) is increasing.  On \(\{h^+<0\}\),
\(\beta(h^+)=\delta_0\), and Young's inequality gives
\begin{align*}
 \delta_0\left|\int M(\nabla_xu^+)q\cdot\nabla_q[h^+]_-\right|
 &\le\frac1{4{\rm Wi}}\int M|\nabla_q[h^+]_-|^2
 +C_b{\rm Wi}\,\delta_0^2\|\nabla_xu^+\|_2^2.
\end{align*}
This proves \eqref{eq:one-step-negative-part} without an inadmissible
finite-dimensional test.

The constant test preserves total mass.  Tests independent of \((q,y)\) give
\eqref{eq:one-step-density-maximum}; testing that scalar equation by
\([\rho^+]_-\) and \((\rho^+-\|\rho^-\|_\infty)^+\) proves its two-sided
maximum principle.
\end{proof}

\begin{lemma}[Joint cutoff removal without compactness in the state variable]
\label{lem:joint-cutoff-removal}
Let \(\delta_j\downarrow0\), \(L_j\uparrow\infty\), and set
\[
 P_j(s):=\max\{\delta_j,\min\{s,L_j\}\},\qquad s\in\R.
\]
Suppose that \(h_j\) is bounded in entropy on its positive set and
\begin{equation}
 \|[h_j]_-\|_{L^2_M}\le C\delta_j.
 \label{eq:joint-cutoff-negative-assumption}
\end{equation}
Then
\begin{equation}
 \|P_j(h_j)-h_j\|_{L^1_M}
 \le C\delta_j+C(\log L_j)^{-1}\longrightarrow0.
 \label{eq:joint-cutoff-L1-distance}
\end{equation}
If \(h_j\rightharpoonup h\) weakly in \(L^1_M\), then \(h\ge0\),
\(P_j(h_j)\rightharpoonup h\), and
\begin{align}
 \int M F(h)&\le\liminf_{j\to\infty}
 \int M F^{L_j}_{\delta_j}(h_j),
 \label{eq:joint-cutoff-entropy-liminf}\\
 4\int M|\nabla\sqrt h|^2&\le\liminf_{j\to\infty}
 \int M\frac{|\nabla h_j|^2}{P_j(h_j)}.
 \label{eq:joint-cutoff-Fisher-liminf}
\end{align}
The gradient in \eqref{eq:joint-cutoff-Fisher-liminf} may be either
\(\nabla_x\) or \(\nabla_q\).  Moreover, with
\(\Phi(r,s):=(r-s)(\log r-\log s)\), extended lower
semicontinuously to \([0,\infty)^2\), one has pointwise
\begin{equation}
 (r-s)\bigl[(F^L_\delta)'(r)-(F^L_\delta)'(s)\bigr]
 \ge \Phi(P_{\delta,L}(r),P_{\delta,L}(s)).
 \label{eq:regularized-Jeffreys-projection-bound}
\end{equation}
Consequently, on any fixed finite measure space \((Z,\mu)\), if
\(a_j\to a\) strongly in measure and
\(0<a_*\le a_j,a\le a^*\), then
\begin{align}
 &\frac12\int_Z a\,\Phi(h',h)\dd\mu
 \notag\\
 &\qquad\le\liminf_{j\to\infty}\frac12\int_Z a_j
 (h_j'-h_j)\bigl[(F^{L_j}_{\delta_j})'(h_j')
 -(F^{L_j}_{\delta_j})'(h_j)\bigr]\dd\mu.
 \label{eq:joint-cutoff-Jeffreys-liminf}
\end{align}
Here \(\mu\) may include fixed non-negative weights, such as the possibly
degenerate jump kernel \(M(q)k(y,y')\).
\end{lemma}

\begin{proof}
On \(\{0\le h_j\le L_j\}\), the difference in
\eqref{eq:joint-cutoff-L1-distance} is at most \(\delta_j\).  On the
negative set it is bounded by \(\delta_j+[h_j]_-\), whereas on
\(\{h_j>L_j\}\) it equals \((h_j-L_j)_+\).  Since
\(F(s)\ge \frac12s\log L_j\) for \(s\ge L_j\) once \(L_j\) is large,
the entropy bound and Cauchy--Schwarz prove
\eqref{eq:joint-cutoff-L1-distance}.  They also imply \(h\ge0\).

The convex regularization satisfies
\(F^L_\delta(s)\ge F(P_{\delta,L}(s))\).  The integral functional
generated by \(F\) is weakly lower semicontinuous on \(L^1\), which gives
\eqref{eq:joint-cutoff-entropy-liminf}.  Next,
\[
 |\nabla\sqrt{P_j(h_j)}|^2
 =\frac{\mathbf1_{\{\delta_j<h_j<L_j\}}}{4h_j}|\nabla h_j|^2
 \le\frac14\frac{|\nabla h_j|^2}{P_j(h_j)}.
\]
The \(L^1\) convergence in \eqref{eq:joint-cutoff-L1-distance} implies
\(\sqrt{P_j(h_j)}\to\sqrt h\) in \(L^2_M\).  Weak compactness in the
corresponding weighted Sobolev space and lower semicontinuity prove
\eqref{eq:joint-cutoff-Fisher-liminf}.

It remains to prove the reaction comparison.  Assume \(r\ge s\).  The
projection \(P=P_{\delta,L}\) is nondecreasing and one-Lipschitz, hence
\(r-s\ge P(r)-P(s)\ge0\).  Since
\[
 (F^L_\delta)'(r)-(F^L_\delta)'(s)
 =\int_s^r\frac{\dd z}{P(z)}
 \ge\int_{P(s)}^{P(r)}\frac{\dd z}{z}
 =\log P(r)-\log P(s),
\]
multiplication proves \eqref{eq:regularized-Jeffreys-projection-bound};
the case \(r<s\) follows by symmetry.  Finally, \(\Phi\) is nonnegative,
convex, and lower semicontinuous on \([0,\infty)^2\).  Thus its integral
is weakly lower semicontinuous under
\((P_j(h_j),P_j(h_j'))\rightharpoonup(h,h')\).  Lemma
\ref{lem:fixed-measure-Jeffreys-liminf}, applied with coefficient \(a_j\),
gives the corresponding weighted conclusion.  Combining it with
\eqref{eq:regularized-Jeffreys-projection-bound} proves
\eqref{eq:joint-cutoff-Jeffreys-liminf}.
\end{proof}

\begin{theorem}[Large-data coupled continuum-state existence]
\label{thm:v35-large-data-coupled-existence}
Assume \(u^{\rm in}\in H_\sigma(\Omega)\) and \(h^{\rm in}\ge0\) satisfy
\begin{equation}
 \mathcal F_Y[u^{\rm in},h^{\rm in}]<\infty,
 \qquad
 \rho[h^{\rm in}]\in L^\infty(\Omega).
 \label{eq:v35-large-data-initial-class}
\end{equation}
Then for every \(T>0\) there is a weak solution in the sense of Definition
\ref{def:v35-large-data-weak-solution}.  It conserves total polymer mass,
preserves non-negativity, and for almost every \(t\in(0,T)\) satisfies
\begin{align}
 &\mathcal F_Y[u,h](t)
 +\nu\int_0^t\|\nabla_xu\|_2^2\dd r
 +4\delta\int_0^t\!\int M|\nabla_x\sqrt h|^2\notag\\
 &\quad+\frac2{{\rm Wi}}\int_0^t\!\int M|\nabla_q\sqrt h|^2
 +\int_0^t\mathcal D_Y(Mh)\dd r
 \le \mathcal F_Y[u^{\rm in},h^{\rm in}].
 \label{eq:v35-large-data-energy-entropy}
\end{align}
Moreover, \(\rho[h]\) is the weak solution of
\(\partial_t\rho+u\cdot\nabla_x\rho=\delta\Delta_x\rho\) and
\begin{equation}
 \|\rho[h](t)\|_{L^\infty_x}
 \le\|\rho[h^{\rm in}]\|_{L^\infty_x}
 \quad\text{for almost every }t.
 \label{eq:v35-number-density-maximum}
\end{equation}
For both \(d=2\) and \(d=3\) this is a genuine weak solution with the
identified FENE stress; there is no stress-defect measure or momentum
subsolution.  No uniqueness is asserted in this energy--entropy class.
\end{theorem}

\begin{proof}
We organize the argument around the parts inherited from the classical FENE
scheme and the two state-fibre passages that are new.

\paragraph{1. Entropy-compatible approximation.}
Put \(F(s)=s(\log s-1)+1\) for \(s\ge0\).  For
\(0<\delta_0<1<L\), define the convex regularization
\(F^L_{\delta_0}\) on the whole real line by the normalization
\(F^L_{\delta_0}(1)=(F^L_{\delta_0})'(1)=0\) and
\begin{equation}
 (F^L_{\delta_0})''(s)=\frac1{\beta^L_{\delta_0}(s)},
 \qquad
 \beta^L_{\delta_0}(s):=\max\{\delta_0,\min\{s,L\}\},\qquad s\in\R.
 \label{eq:entropy-compatible-cutoff}
\end{equation}
Thus \(F^L_{\delta_0}\to F\) locally uniformly on \([0,\infty)\) as
\(\delta_0\downarrow0\) and \(L\uparrow\infty\), and, crucially,
\begin{equation}
 \beta^L_{\delta_0}(h)\nabla_q(F^L_{\delta_0})'(h)=\nabla_qh.
 \label{eq:entropy-cutoff-chain-rule}
\end{equation}
This is the entropy-compatible cutoff used in the classical FENE construction
\citep{barrett2011existence}; an ordinary logarithmic test against a separately
truncated stress would not give the cancellation below.  Fix
\(s>d/2+1\) and let
\(V_{\sigma,s}\) be the closure of
\(C_{c,\sigma}^\infty(\Omega;\R^d)\) in \(H^s(\Omega;\R^d)\).  The compact
embedding \(V_{\sigma,s}\Subset H_\sigma\) defines an \(L^2\)-orthonormal
spectral basis \((w_k)\) for the associated positive compact solution
operator.  Set \(X_N=\operatorname{span}\{w_1,\ldots,w_N\}\).  Its
\(L^2\)-orthogonal projection \(P_N\) is a contraction in
\(V_{\sigma,s}\), and the union of the \(X_N\) is dense in \(V_\sigma\).
Work with the fixed Maxwellian \(M\) and nested configuration spaces
dense in \(H^1_M(D)\); the weighted density theorem under
\eqref{eq:v36-BS-potential-a}--\eqref{eq:v36-BS-potential-b} supplies them.
Choose nested finite Borel partitions \(\mathcal P_m\) whose union generates
\(\mathcal B(Y)\), and denote conditional expectation in \(y\) by \(\Pi_m\).
Use the bounded non-negative initial ratios
\[
 h^{\rm in}_{\ell,m}=\Pi_m\min\{h^{\rm in},\ell\}.
\]
They satisfy \(\rho[h^{\rm in}_{\ell,m}]\le\rho[h^{\rm in}]\), converge in
weighted \(L^1\) as \(m,\ell\to\infty\), and have convergent entropies by
conditional Jensen and martingale convergence.  They lie in the weighted
\(L^2\) space needed at the first implicit step.  Their masses converge to the
prescribed mass; exact mass conservation at each approximation level therefore
passes to the limit.

At time step \(\Delta t\), solve the implicit kinetic problem with velocity
in \(X_N\), replace the configurational drag by
\(\nabla_q\cdot((\nabla_xu_{L,N})qM
\beta^L_{\delta_0}(h_{L,N}))\), and use in the momentum equation the full
regularized Kramers stress
\begin{equation}
 \tau_{L,N}:=\int_{D\times Y}\nabla_qU\otimes q\,
 Mh_{L,N}\dd q\dd\pi.
 \label{eq:v35-approximate-stress}
\end{equation}
The reaction is kept implicit and untruncated:
\(M\mathcal R_{\eta[h_{L,N}]}h_{L,N}\).  Proposition
\ref{prop:one-step-entropy-cutoff}, on the fixed state space
\(\mathcal Y_m\), solves this coupled step without assuming
uniqueness of an auxiliary frozen-rate problem.  Iterate it for
\(n=1,\ldots,\lceil T/\Delta t\rceil\), using the terminal value of one step as
the datum of the next, and sum \eqref{eq:one-step-entropy}.  The intermediate
\((x,q)\)-Galerkin dimension has already been removed inside the proposition;
only \((m,\delta_0,L,N,\Delta t,\ell)\) remain.  This closes the moment
\emph{fields}, rather than an unstated finite vector.  No compact embedding in
\(Y\) is used: the state operator enters as a bounded accretive form, while
compactness is required only for the finitely many moment fields and the fluid
Galerkin variables.

For completeness, the boundary approximation does not require a smooth
boundary.  A bounded Lipschitz domain has an \(H^1\)-extension operator and
the compact embedding \(H^1(\Omega)\Subset L^2(\Omega)\).  The solenoidal
test space is dense in \(V_\sigma\), while the scalar conforming spaces are
dense in \(H^1(\Omega)\); hence the zero velocity trace is built into
\(X_N\), and the spatial Neumann condition is the natural boundary condition
of the scalar variational problem.  The weighted density theorem used by
Barrett--S\"uli under
\eqref{eq:v36-BS-potential-a}--\eqref{eq:v36-BS-potential-b} supplies smooth
configuration approximants.  Integrating the combined configurational drift
and diffusion by parts, rather than either term separately, gives the second
condition in \eqref{eq:v36-no-flux-boundaries}.  A \(C^{1,1}\) domain may
equivalently be treated with a Stokes eigenbasis, but no \(C^{1,1}\) estimate
is used in the compactness argument.

Test the momentum equation by \(u_{L,N}\) and the kinetic equation by
\((F^L_{\delta_0})'(h_{L,N})\).  By
\eqref{eq:entropy-cutoff-chain-rule}, incompressibility, radiality of \(U\),
and configurational integration by parts, the drag contribution is
\begin{equation}
 \int M\beta^L_{\delta_0}(h_{L,N})(\nabla_xu_{L,N})q\cdot
 \nabla_q(F^L_{\delta_0})'(h_{L,N})
 =\int_\Omega\tau_{L,N}:\nabla_xu_{L,N},
 \label{eq:entropy-compatible-drag-cancellation}
\end{equation}
with the sign opposite to the stress work in the momentum equation.  The
reaction contributes \(\mathcal D^{L}_{\delta_0,Y}\) from
\eqref{eq:regularized-reaction-production}.  It is non-negative by symmetry
and monotonicity.  At this finite level it is not yet the Jeffreys
production.  We do not remove the lower cutoff at fixed \(L\): that route
would require strong compactness of the entire density in the continuum
variable \(y\), which the model neither assumes nor provides.  Instead, sum
\eqref{eq:one-step-negative-part} over the time steps.  Since the initial
truncation is non-negative and the velocity part of the discrete energy
inequality is uniform, one obtains
\begin{align}
 &\max_{0\le m\le T/\Delta t}\|[h^m_{\delta_0,L}]_-\|_{L^2_M}^2
 +\sum_m\Delta t\left(
 \delta\|\nabla_x[h^m_{\delta_0,L}]_-\|_{L^2_M}^2
 +\frac1{2{\rm Wi}}\|\nabla_q[h^m_{\delta_0,L}]_-\|_{L^2_M}^2\right)
 \notag\\
 &\hspace{5cm}\le C_T\delta_0^2.
 \label{eq:global-negative-part-vanishing}
\end{align}
Thus positivity is recovered quantitatively in the joint limit
\(\delta_0\downarrow0\), \(L\uparrow\infty\), rather than asserted at an
intermediate signed approximation.

The finite activity moments are compact uniformly in both cutoffs.  Indeed,
\[
 \int_{D\times Y}M\beta^L_{\delta_0}(h)
 \le \rho[h]+\int_{D\times Y}M[h]_-+\delta_0.
\]
After integration in \(x\), the density maximum principle,
\eqref{eq:global-negative-part-vanishing}, and
\eqref{eq:one-step-moment-H1} give a cutoff-independent \(BV_x\) bound; no
pointwise \(H^1_x\) bound is asserted.  The time-translation estimate proved
in Step 3 below is uniform for the same reason.  The compactness criterion of
Lemma \ref{lem:compactness-tools} therefore gives
\begin{equation}
 \eta[h_{\delta_0,L}]\longrightarrow\eta[h]
 \quad\text{strongly in }L^1((0,T)\times\Omega;\R^{N_\eta})
 \quad\text{and almost everywhere}.
 \label{eq:joint-cutoff-moment-identification}
\end{equation}
The activity converges strongly in every finite \(L^p\), and its product with
the weakly convergent density is identified.  Lemma
\ref{lem:joint-cutoff-removal} simultaneously supplies the actual free-energy
and Fisher-information lower bounds and, through
\eqref{eq:joint-cutoff-Jeffreys-liminf}, the actual Jeffreys production.  This
joint argument is the point at which positivity, the nonlinear reaction rate,
and continuum-state dissipation are closed without any compact embedding in
\(Y\).
Testing by a function
independent of \((q,y)\) gives the scalar advection--diffusion equation for
\(\rho_{L,N}\); its maximum principle proves
\eqref{eq:v35-number-density-maximum} uniformly.  The only coupled rate is
the cutoff remainder \(O(\Delta t L)\).  Choose explicitly
\begin{equation}
 \begin{gathered}
 L_j\uparrow\infty,\qquad \ell_j\uparrow\infty,\qquad
 N_j\uparrow\infty,\qquad m_j\uparrow\infty,
 \qquad \Delta t_j\downarrow0,\qquad \delta_{0,j}\downarrow0,\\
 \Delta t_jL_j\le j^{-1},\qquad
 (\log L_j)^{-1/2}+\delta_{0,j}^{1/2}\le j^{-1}.
 \end{gathered}
 \label{eq:diffusive-public-diagonal}
\end{equation}
Enumerate countable determining families of fluid and kinetic tests.  After the
uniform estimates in Steps 2--6 have been obtained, choose the \(j\)-th tuple
so that, for the first \(j\) tests \((w_m,\varphi_m)\),
\begin{equation}
 |\mathfrak R^u_j(w_m)|+|\mathfrak R^h_j(\varphi_m)|\le j^{-1}
 \qquad(1\le m\le j),
 \label{eq:diffusive-test-residuals}
\end{equation}
where the residuals are the differences between the time-interpolated
Galerkin identities and \eqref{eq:v35-weak-momentum}--
\eqref{eq:v35-weak-kinetic}, including the replacement of
\(h_{L,N}\) by \(\beta^{L}_{\delta_0}(h_{L,N})\) in the drag.  The latter
contribution is bounded by
\(C_{\varphi_m}[(\log L_j)^{-1/2}+\delta_{0,j}^{1/2}]\): this follows from the
entropy tail, \eqref{eq:global-negative-part-vanishing}, the number-density
maximum principle, and Cauchy--Schwarz against
\(\nabla u_{L,N}\), exactly as quantified in
\eqref{eq:drag-cutoff-removal} below.  Galerkin projection, state-partition,
and initial-truncation errors are made at most \(j^{-1}\) by increasing
\(N_j\), \(m_j\), and \(\ell_j\) after
\((L_j,\Delta t_j,\delta_{0,j})\) have been fixed.  The
lower regularization remains on this same explicit diagonal; by
\eqref{eq:global-negative-part-vanishing} its negative part is
\(O(\delta_{0,j})\), and by \eqref{eq:joint-cutoff-L1-distance} its drag
coefficient differs from the density by
\(O(\delta_{0,j}+(\log L_j)^{-1})\) in \(L^1_M\).  This diagonal
simultaneously removes time, Galerkin, state-partition, both entropy cutoffs,
and initial-truncation
parameters and makes every discrete residual vanish.  The estimates are uniform in all parameters removed
later, so the construction is made on the whole interval \([0,T]\); no local
restart or continuation criterion is assumed.  From this point on we write
\((u_j,h_j)\) for the resulting diagonal interpolants; occurrences of
\((u_{L,N},h_{L,N})\) in parameter-independent estimates refer to this same
sequence.

\paragraph{2. Uniform integrability and time compactness.}
The de la Vall\'ee--Poussin criterion applied to
\(h\mapsto h\log h-h+1\) makes the positive parts of
\(Mh_{\delta_0,L,N}\) uniformly integrable on the full
\((t,x,q,y)\) cylinder, while
\eqref{eq:global-negative-part-vanishing} makes the negative parts converge
strongly to zero.  Thus the signed approximations are uniformly integrable
and every weak limit is non-negative.  Weighted Cauchy--Schwarz gives, for
\(P_j=\beta^{L_j}_{\delta_{0,j}}\),
\begin{equation}
 \int M|\nabla h_j|
 \le\left(\int MP_j(h_j)\right)^{1/2}
 \left(\int M\frac{|\nabla h_j|^2}{P_j(h_j)}\right)^{1/2}\le C_T,
 \label{eq:v35-L1-flux-bound}
\end{equation}
for either \(\nabla=\nabla_x\) or \(\nabla_q\).  Testing the equation against bounded Lipschitz functions
gives a uniform time-translation bound in the dual of
\(W^{1,\infty}_{x,q}(C(Y))\).  Consequently the state average
\(\bar h_{L,N}:=\int_Yh_{L,N}\dd\pi\) is strongly compact in
\begin{equation}
 M\bar h_{L,N}\longrightarrow M\bar h
 \quad\text{in }L^1((0,T)\times\Omega\times D).
 \label{eq:v35-state-average-strong}
\end{equation}
Indeed, reaction cancels after state integration, so
\eqref{eq:v35-L1-flux-bound} supplies the derivative hypothesis in the first
part of Lemma \ref{lem:compactness-tools}; its entropy-tightness argument
removes the FENE boundary collar.  The full state-resolved sequence is only
weakly compact in \(L^1\):
\begin{equation}
 Mh_{L,N}\rightharpoonup Mh
 \quad\text{in }L^1((0,T)\times\Omega\times D\times Y).
 \label{eq:v35-state-density-weak}
\end{equation}
This distinction is essential on a non-atomic state space.

\paragraph{3. Strong finite moments and nonlinear activity.}
Write
\(\Phi_{a,j}(q,y)=(\Pi_{m_j}c_a)(y)\psi_a(q)\).  Since \(h_j\) is
\(\mathcal P_{m_j}\)-measurable,
\(\eta_a[h_j]=\int Mh_j\Phi_{a,j}\), while
\(\|\Phi_{a,j}\|_{W^{1,\infty}_q(L^\infty_y)}\le C_a\).  For
\(\zeta\in W^{1,\infty}(\Omega)\), testing the regularized kinetic equation
by the admissible function \(\Phi_{a,j}\zeta\) gives, for the standard affine
time interpolant,
\begin{align}
 \langle\partial_t\eta_a,\zeta\rangle
 &=\int M h_j\,u_j\cdot\nabla_x\zeta\,\Phi_{a,j}
 +\int MP_j(h_j)(\nabla_xu_j)q\cdot\nabla_q\Phi_{a,j}\,\zeta
 \notag\\
 &\quad-\delta\int M\nabla_xh_j\cdot\nabla_x\zeta\,\Phi_{a,j}
 -\frac1{2{\rm Wi}}\int M\nabla_qh_j\cdot\nabla_q\Phi_{a,j}\,\zeta
 \notag\\
 &\quad-\frac12\int_{\Omega\times D\times Y^2}
 MK_{h_j}(h_j'-h_j)(\Phi_{a,j}'-\Phi_{a,j})\zeta
 +\mathfrak r_{\Delta t}(\zeta),
 \label{eq:finite-moment-distributional-evolution}
\end{align}
where the time-interpolation residual tends to zero in
\(L^1(0,T;(W^{1,\infty})')\) along
\eqref{eq:diffusive-public-diagonal}.  The first two terms are bounded by
\begin{align*}
 C_a\bigl[&R_\rho\|u_j(t)\|_{L^1_x}
 +\|u_j(t)\|_{L^2_x}\|n_j(t)\|_{L^2_x}\\
 &+R_\rho\|\nabla_xu_j(t)\|_{L^1_x}
 +\|\nabla_xu_j(t)\|_{L^2_x}
 \|n_j(t)+\delta_{0,j}\|_{L^2_x}\bigr]
 \|\zeta\|_{W^{1,\infty}},
\end{align*}
where \(n_j(x):=\int_{D\times Y}M[h_j]_-(x,q,y)\dd q\dd\pi\) and
\(\|n_j\|_{L^2_{t,x}}\le C\delta_{0,j}\).  The two diffusion fluxes satisfy
\begin{equation}
 \begin{aligned}
 &\int M|\nabla_xh_j|+\int M|\nabla_qh_j|\\
 &\quad\le\left(\int MP_j(h_j)\right)^{1/2}
 \left[\left(\int M\frac{|\nabla_xh_j|^2}{P_j(h_j)}\right)^{1/2}
 +\left(\int M\frac{|\nabla_qh_j|^2}{P_j(h_j)}\right)^{1/2}\right].
 \end{aligned}
 \label{eq:finite-moment-flux-bound}
\end{equation}
and the reaction term is at most
\(2\|k\|_\infty a^*\|\Phi_{a,j}\|_\infty
\int M|h_j|\,\|\zeta\|_\infty\).
All right-hand sides are integrable in time uniformly along the diagonal.
Moreover,
\begin{equation}
 \int|\nabla_x\eta_a[h_j]|
 \le\|\Phi_{a,j}\|_\infty
 \left(\int MP_j(h_j)\right)^{1/2}
 \left(\int M\frac{|\nabla_xh_j|^2}{P_j(h_j)}\right)^{1/2}
 \le C_a.
 \label{eq:finite-moment-BV-bound}
\end{equation}
Consequently
\begin{equation}
 \|\partial_t\eta_a[h_{L,N}]\|_{L^1(0,T;(W^{1,\infty}(\Omega))')}
 +\|\nabla_x\eta_a[h_{L,N}]\|_{L^1((0,T)\times\Omega)}\le C_a.
 \label{eq:v36-finite-moment-compactness-bound}
\end{equation}
The estimate
\(|\eta_a[h_j]|\le
\|c_a\|_\infty\|\psi_a\|_\infty(\rho[h_j]+2n_j)\) and the maximum principle
give a uniform \(L^2_{t,x}\) bound.  The second part of Lemma
\ref{lem:compactness-tools}, applied to
\eqref{eq:v36-finite-moment-compactness-bound}, first gives strong \(L^1\)
convergence.  Since the limit satisfies
\(|\eta_a[h]|\le C_aR_\rho\), the envelope above yields
\[
 \|\eta_a[h_j]-\eta_a[h]\|_2^2
 \le C_aR_\rho\|\eta_a[h_j]-\eta_a[h]\|_1
 +C_a\|n_j\|_2\|\eta_a[h_j]-\eta_a[h]\|_2.
\]
Hence
\begin{equation}
 \eta_a[h_j]\longrightarrow\eta_a[h]
 \quad\text{strongly in }L^p((0,T)\times\Omega)
 \quad\text{for every }1\le p\le2
 \quad(1\le a\le N_\eta).
 \label{eq:v35-activity-moment-strong}
\end{equation}
After extraction the convergence is almost everywhere.  The Lipschitz and
boundedness assumptions on \(\mathfrak a\) imply convergence in measure and,
by dominated convergence, in every finite \(L^p\) of the effective kernels.
More precisely, the number-density maximum principle and
\eqref{eq:global-negative-part-vanishing} give the quantitative weak--strong product
estimate
\[
 \kappa_j(t,x):=\mathop{\rm ess\,sup}_{q,y,y'}
 |K_{h_j}(t,x,q,y,y')-K_h(t,x,q,y,y')|
 \le \|k\|_\infty A\sum_a|\eta_a[h_j]-\eta_a[h]|,
\]
and hence
\begin{equation}
 \int M|h_j||K_{h_j}-K_h|
 \le
 R_\rho\|\kappa_j\|_{L^1_{t,x}}
 +2\|n_j\|_{L^2_{t,x}}\|\kappa_j\|_{L^2_{t,x}}
 \longrightarrow0.
 \label{eq:v35-activity-weighted-identification}
\end{equation}
Combining this with \eqref{eq:v35-state-density-weak} identifies the reaction
term in \eqref{eq:v35-weak-kinetic}.  The coefficient converges strongly and
the state density weakly; no strong convergence in \(y\) is asserted.

\paragraph{4. Stress compactness and identification.}
Conditions \eqref{eq:v36-BS-potential-a}--
\eqref{eq:v36-BS-potential-b} give
\(M|\nabla_qU|^2\asymp\operatorname{dist}(q,\partial D)^{\gamma-2}\),
which is integrable exactly when \(\gamma>1\); for the Warner potential this is
\(b>2\).  Set \(\widetilde h_j:=[h_j]_++\delta_{0,j}\).  Since
\begin{equation}
 |\nabla_q\sqrt{\widetilde h_j}|^2
 =\frac{\mathbf1_{\{h_j>0\}}|\nabla_qh_j|^2}
 {4([h_j]_++\delta_{0,j})}
 \le\frac14\frac{|\nabla_qh_j|^2}{P_j(h_j)},
 \label{eq:signed-cutoff-square-root-control}
\end{equation}
Lemma \ref{lem:FENE-Hardy-tail}, integrated also over \(Y\), controls the
stress generated by \(\widetilde h_j\).  The remaining pieces satisfy
\begin{equation}
 \left\|\int M(\nabla_qU\otimes q)[h_j]_-\dd q\dd\pi\right\|_{L^2_{t,x}}
 \le \|\nabla_qU\otimes q\|_{L^2_M(D)}\|[h_j]_-\|_{L^2_M}
 \le C\delta_{0,j},
 \label{eq:negative-stress-vanishing}
\end{equation}
and the artificial constant \(\delta_{0,j}\) contributes \(O(\delta_{0,j})\)
because \(\nabla_qU\otimes q\in L^1_M(D)\).  Consequently the stresses are
bounded in \(L^1((0,T)\times\Omega)\):
\begin{equation}
 \|\tau_j\|_{L^1_{t,x}}
 \le C_U\left[1+\left(\int_0^T\!\int M
 \frac{|\nabla_qh_j|^2}{P_j(h_j)}\right)^{1/2}\right]\le C_T.
 \label{eq:v35-stress-Hardy-bound}
\end{equation}
No equiintegrability in \((t,x)\) is needed at the approximation level.  To
identify the stress distribution, truncate
\(W=\nabla_qU\otimes q\) by \(W_A=W\mathbf1_{\{|W|\le A\}}\).
Equation \eqref{eq:v35-state-average-strong} gives strong \(L^1_{t,x}\)
convergence of \(\int_DMW_A\bar h_j\dd q\).  Estimate
\eqref{eq:FENE-stress-tail}, applied to \(\widetilde h_j\), together with
\eqref{eq:negative-stress-vanishing}, gives a modulus
\(\omega_U(A)\downarrow0\) such that
\begin{equation}
 \limsup_{j\to\infty}\left\|\int_{D\times Y} M(W-W_A)
 h_j\dd q\dd\pi\right\|_{L^1_{t,x}}
 \le\omega_U(A).
 \label{eq:v35-stress-uniform-tail}
\end{equation}
The limiting Fisher bound and
\eqref{eq:v35-number-density-maximum} give the same tail estimate for \(h\).
If \(\Xi\in L^\infty((0,T)\times\Omega)\), then
\begin{align}
 \left|\int (\tau_j-\tau_Y[h]):\Xi\right|
 &\le \left|\int (\tau^A_j-\tau^A_Y[h]):\Xi\right|
 +2\|\Xi\|_\infty\omega_U(A),
 \label{eq:stress-two-limit-identification}
\end{align}
where \(\tau^A\) denotes insertion of \(W_A\).  The first term tends to zero
at fixed \(A\) by \eqref{eq:v35-state-average-strong}; letting \(A\to\infty\)
identifies the distributional limit as \(\tau_Y[h]\).  After the limit has
been taken, \(h\ge0\), \(\rho[h]\le R_\rho\), and
\(\sqrt h\in L^2_tH^1_{M,q}\) by
\eqref{eq:joint-cutoff-Fisher-liminf}.  Applying
\eqref{eq:FENE-stress-L2} to \(h\) now gives
\begin{equation}
 \tau_Y[h]\in L^2((0,T)\times\Omega),
 \label{eq:v35-stress-weak-compactness}
\end{equation}
with strong convergence only for bounded stress truncations.  No strong
stress convergence is assumed or claimed.

\paragraph{5. Fluid and deformation limits.}
The energy bound gives weak convergence of \(u_j\) in
\(L^2(0,T;V_\sigma)\).  In three dimensions interpolation gives
\begin{equation}
 \|u_j\|_{L^{8/3}(0,T;L^4(\Omega))}\le C,
 \qquad
 \|(u_j\cdot\nabla)u_j\|_{L^{4/3}(0,T;V_\sigma')}
 \le C;
 \label{eq:v36-3d-convection-bound}
\end{equation}
in two dimensions the same bound follows from Ladyzhenskaya's inequality
(with a better time exponent).  For \(\phi\in V_{\sigma,s}\), self-adjointness
of \(P_{N_j}\) in \(L^2\) and \(\partial_tu_j\in X_{N_j}\) give
\[
 \langle\partial_tu_j,\phi\rangle
 =\langle\partial_tu_j,P_{N_j}\phi\rangle,
 \qquad
 \|P_{N_j}\phi\|_{W^{1,\infty}}
 \le C\|P_{N_j}\phi\|_{V_{\sigma,s}}
 \le C\|\phi\|_{V_{\sigma,s}}.
\]
The projected momentum equation, the convection estimate above, and
\eqref{eq:v35-stress-Hardy-bound} therefore bound
\(\partial_tu_j\) in \(L^1(0,T;V_{\sigma,s}')\), uniformly in both
dimensions.  The compact embedding \(V_\sigma\Subset H_\sigma\), the
continuous embedding \(H_\sigma\hookrightarrow V_{\sigma,s}'\), and Simon's
time-compactness theorem give
\begin{equation}
 u_j\longrightarrow u
 \quad\text{strongly in }L^2(0,T;L^2(\Omega)),
 \label{eq:v36-Leray-strong}
\end{equation}
which identifies \(u\otimes u\) in \(L^1\).

For the deformation term, a fixed test defines the bounded moment
\begin{equation}
 G_j(t,x):=\int_{D\times Y}Mh_j
 q\cdot\nabla_q\varphi(t,x,q,y)\dd q\dd\pi.
 \label{eq:v36-deformation-moment}
\end{equation}
To justify compactness for this non-autonomous observable, first take
\(\varphi\) smooth with bounded derivatives through order two in \((x,q)\),
put \(\varphi_j=\Pi_{m_j}\varphi\), and set
\(\Psi_j=q\cdot\nabla_q\varphi_j\).  Since \(h_j\) is
\(\mathcal P_{m_j}\)-measurable, conditional expectation gives
\(G_j=\int Mh_j\Psi_j\) exactly.  Moreover, all required \((t,x,q)\)
derivative bounds of \(\Psi_j\) are bounded by those of \(\varphi\).
Repeating
\eqref{eq:finite-moment-distributional-evolution}, now accounting also for
\(\partial_t\Psi\) and \(\nabla_x\Psi\), gives
\begin{equation}
 \|\partial_tG_j\|_{L^1(0,T;(W^{1,\infty}(\Omega))')}
 +\|\nabla_xG_j\|_{L^1((0,T)\times\Omega)}
 \le C_\varphi.
 \label{eq:deformation-moment-compactness-bound}
\end{equation}
Indeed, the terms containing derivatives of \(\Psi_j\) are bounded directly by
\(R_\rho\), while transport, diffusion, deformation, and reaction are bounded
exactly as in \eqref{eq:finite-moment-flux-bound}--
\eqref{eq:finite-moment-BV-bound}.  Lemma \ref{lem:compactness-tools} gives
strong \(L^1_{t,x}\) convergence along a subsequence.  Its limit is
\(G(t,x)=\int Mh\Psi\): this follows by testing
\eqref{eq:v35-state-density-weak} against scalar functions of \((t,x)\) times
\(\Psi\).  A countable dense family in the smooth \((t,x,q)\) variables,
tensored with simple functions from a countable Borel algebra generating
\(Y\), permits one common subsequence.  Bounded simple-function approximation
in \(y\), followed by the monotone-class theorem and dominated convergence for
each finite flux measure, extends the conclusion to every test in Definition
\ref{def:v35-large-data-weak-solution}.  Since \(D\) and
\(q\cdot\nabla_q\varphi\) are bounded,
\(|G_j|\le C_\varphi(R_\rho+2n_j)\), whereas
\(|G|\le C_\varphi R_\rho\).  Therefore
\[
 \|G_j-G\|_2^2\le C_\varphi R_\rho\|G_j-G\|_1
 +C_\varphi\|n_j\|_2\|G_j-G\|_2,
\]
and the strong \(L^1\) convergence upgrades to strong \(L^2_{t,x}\).
Pairing it with the weak
\(L^2_{t,x}\) convergence of \(\nabla_xu_j\) closes the kinetic
deformation term after removal of the cutoff.  Indeed, with
\(P_j=\beta^{L_j}_{\delta_{0,j}}\) and
\[
 d_j(t,x):=\int_{D\times Y}M
 \left|h_j-P_j(h_j)\right|\dd q\dd\pi,
\]
the entropy bound, \eqref{eq:global-negative-part-vanishing}, and
\(\rho[h_j]\le R_\rho\) give
\begin{equation}
 \int d_j\le C\left(\frac1{\log L_j}+\delta_{0,j}\right),
 \qquad
 \|d_j\|_{L^2_{t,x}}^2
 \le C R_\rho\|d_j\|_{L^1_{t,x}}
 +C\|[h_j]_-\|_{L^2_M}^2\longrightarrow0.
 \label{eq:drag-cutoff-removal}
\end{equation}
To obtain the second estimate, split into \(h_j>L_j\),
\(0\le h_j\le L_j\), and \(h_j<0\), and use weighted
Cauchy--Schwarz in \((q,y)\) on the last set.  Consequently the moment formed
with \(P_j(h_j)\) has the same strong \(L^2\) limit as \(G_j\).  Thus neither three-dimensional convection nor drag
requires strong state-resolved density or strong velocity-gradient
convergence.

\paragraph{6. Jeffreys lower semicontinuity.}
On the fixed measure
\(M(q)k(y,y')\dd r\dd x\dd q\dd\pi\dd\pi'\), the two coordinate lifts of
\eqref{eq:v35-state-density-weak} converge weakly in \(L^1\): for a bounded
test on \(Y^2\), integration in the unused state coordinate produces a
bounded test on \(Y\).  The activities converge in measure by
\eqref{eq:v35-activity-moment-strong} and stay between \(a_*\) and \(a^*\).
The projection comparison
\eqref{eq:regularized-Jeffreys-projection-bound} and Lemma
\ref{lem:joint-cutoff-removal} therefore give
\begin{equation}
 \int_0^t\mathcal D_Y(Mh)\dd r
 \le\liminf_{j\to\infty}\int_0^t
 \mathcal D^{L_j}_{\delta_{0,j},Y}(h_j)\dd r.
 \label{eq:v35-Jeffreys-lsc}
\end{equation}
The Fisher and viscous dissipation integrands are convex, so weak lower
semicontinuity applies to them.  This proves
\eqref{eq:v35-large-data-energy-entropy}.

\paragraph{7. Initial trace and mass.}
The bounds \eqref{eq:v36-3d-convection-bound} and the momentum equation give
\(u\in C_w([0,T];H_\sigma)\).  For the kinetic density, the uniform
time-translation estimate gives continuity against a countable determining
class of bounded tests that are smooth in \((x,q)\) and simple on a countable
Borel algebra generating \(Y\).  Lusin approximation in \((x,q,y)\) and the
uniform-integrability modulus extend the continuity to every \(L^\infty\)
test.  Uniform integrability from
the entropy bound upgrades this to
\(Mh\in C_w([0,T];L^1(\Omega\times D\times Y))\).  Passing in the
time-integrated weak forms with strongly convergent approximate initial data
identifies
\(u(0)=u^{\rm in}\) and \(Mh(0)=Mh^{\rm in}\).  Taking
\(\varphi\equiv1\) in \eqref{eq:v35-weak-kinetic} annihilates both diffusion
fluxes, transport, and the antisymmetric reaction pair, giving
\begin{equation}
 \int_{\Omega\times D\times Y}Mh(t)
 =\int_{\Omega\times D\times Y}Mh^{\rm in}
 \quad\text{for every }t\in[0,T].
 \label{eq:v35-exact-mass}
\end{equation}
Weak closure of the positive cone gives \(h\ge0\).  The scalar test already
used in Step 1 passes to the limit and proves
\eqref{eq:v35-number-density-maximum}.  The construction is global because
all continuation bounds depend only on the initial energy, entropy, number
density, and fixed coefficients, not on a smallness threshold.
\end{proof}

\begin{corollary}[Strict infinite-state positive-diffusion extension]
\label{cor:strict-positive-diffusion-extension}
In Theorem~\ref{thm:v35-large-data-coupled-existence}, take
\(Y=\T^1\) and \(k=k_\infty\) from Proposition
\ref{prop:infinite-rank-state-kernel}.  Then every admissible large datum and
every activity satisfying
\eqref{eq:continuum-activity-bounds}--
\eqref{eq:continuum-activity-lipschitz} generate a global energy--entropy weak
solution.  Even for constant activity, the state reaction in this solution
cannot be reduced by a linear state identification to a scalar FENE equation
or to finitely many reacting FENE species.
\end{corollary}

\begin{proof}
Proposition~\ref{prop:infinite-rank-state-kernel} verifies all kernel
hypotheses, so existence follows from Theorem
\ref{thm:v35-large-data-coupled-existence}.  A scalar model has no nontrivial
state generator.  A model with \(N\) species has a generator of rank at most
\(N\), whereas \(\mathcal L_\infty\) has infinite rank.  Therefore no exact
linear state identification reduces the displayed continuum-state reaction
to either class.
\end{proof}

\begin{remark}[What is inherited from scalar FENE compactness]
\label{rem:strict-comparison-scalar-FENE}
When \(Y\) is a singleton, the reaction and Jeffreys terms disappear and the
positive-diffusion construction reduces to the scalar compactness mechanism
used in the classical theory \citep{barrett2011existence}.  Three passages in
the present proof are not consequences of that scalar theorem: approximation
of the non-atomic state fibre by finite sigma-algebras, strong compactness of
the selected state-sensitive moments while \(h_j\) remains weak in \(Y\), and
joint lower semicontinuity of the two-copy Jeffreys integrand after activity
identification and entropy-cutoff removal.  Proposition
\ref{prop:infinite-rank-state-kernel} shows that these passages are needed for
an admissible system that cannot be reduced by a linear finite-state change of
variables.  This is the precise sense in which Theorem
\ref{thm:v35-large-data-coupled-existence} extends, rather than restates, the
scalar weak-existence theory.
\end{remark}

\begin{remark}[Exact boundary of the activity class]
\label{rem:v35-activity-boundary}
The proof closes for prescribed bounded uniformly positive symmetric
activities and, more
generally, for the bounded finite-moment class above because
\eqref{eq:v35-activity-moment-strong} identifies the coefficient.  It does
 not by itself cover pointwise density-dependent rates, unbounded
 state-sensitive moment observables, or
activities that are merely weakly continuous on the state-resolved density.
In those classes, \eqref{eq:v35-state-density-weak} does not identify the
coefficient, and the product in
\eqref{eq:v35-activity-weighted-identification} may develop a defect.  A
solution-dependent reference measure would create an additional obstruction:
the relative entropy and detailed-balance measure would vary with the
approximation, so \eqref{eq:v35-Jeffreys-lsc} would no longer follow from the
fixed-measure convex argument.  These are compactness obstacles, not evidence
of non-existence.  No uniqueness statement is made for any of these
large-data classes.
\end{remark}

\section{Zero centre-of-mass diffusion: matrix fibres and projection--lift equivalence}
\label{sec:v37-zero-diffusion}

We now set the centre-of-mass diffusion equal to zero.  This is not the
formal substitution \(\delta=0\) in Theorem
\ref{thm:v35-large-data-coupled-existence}: all spatial Fisher compactness
used there disappears.  The replacement is the propagation-of-compactness
argument of Masmoudi \citep{masmoudi2013global}, applied only to the closed
state average.  The continuum state density is then recovered from a linear
Markov lift, so no compact embedding or defect liminf in \(Y\) is needed.
The theorem is stated for corotational drag in two dimensions and a prescribed
kernel.  The distinction from general drag is made explicit in Remark
\ref{rem:v37-general-drag-obstruction}.

Let \(\Omega=\T^2\), \(D=B_{\sqrt b}(0)\), and let \(U,M\) be given by
\eqref{eq:fene-potential}, with \(b>2\).  Put
\begin{equation}
 \mathsf W(u):=\frac12(\nabla_xu-(\nabla_xu)^T).
 \label{eq:v37-corotational-gradient}
\end{equation}
Let \((Y,\pi)\) be the compact probability space of Section
\ref{sec:continuum-state}, and fix
\begin{equation}
 K_0\in L^\infty(D\times Y\times Y),\qquad
 K_0(q,y,y')=K_0(q,y',y)\ge0.
 \label{eq:v37-fixed-kernel}
\end{equation}
Thus \(K_0\) may vary with configuration and state, but not with the
solution.  Define
\begin{equation}
 (\mathcal R_0h)(q,y):=\int_YK_0(q,y,y')
 [h(q,y')-h(q,y)]\dd\pi(y').
 \label{eq:v37-fixed-reaction}
\end{equation}
The zero-diffusion kinetic equation is
\begin{equation}
 \partial_t(Mh)+u\cdot\nabla_x(Mh)
 =-\nabla_q\cdot(\mathsf W(u)qMh)
 +\frac1{2{\rm Wi}}\nabla_q\cdot(M\nabla_qh)+M\mathcal R_0h.
 \label{eq:v37-zero-diffusion-kinetic}
\end{equation}
Because \(U\) is radial,
\(\mathsf W(u)q\cdot\nabla_qU=0\) and
\(\nabla_q\cdot(\mathsf W(u)q)=0\).  These identities, which fail with
\(\mathsf W(u)\) replaced by the full \(\nabla_xu\), are the precise
corotational simplification used below.

For \(\rho^{\rm in}(x):=\int_{D\times Y}Mh^{\rm in}\dd q\dd\pi\), set
\begin{equation}
 \mathcal L_2(h^{\rm in}):=
 \int_\Omega\left[\int_{D\times Y}Mh^{\rm in}
 \left|\log\frac{h^{\rm in}}{\rho^{\rm in}}\right|^2
 \dd q\dd\pi\right]^{1/2}\dd x,
 \label{eq:v37-log2-initial}
\end{equation}
with the integrand set to zero on \(\{\rho^{\rm in}=0\}\).  This is the
finite-measure form of Masmoudi's additional logarithmic-square hypothesis,
now integrated over the state fibre.  It is not a consequence of finite
relative entropy and is therefore stated separately.

\begin{definition}[Zero-diffusion corotational weak solution]
\label{def:v37-zero-diffusion-weak}
Let \(\mathcal K\in L^\infty((0,T)\times\T^2\times D\times Y^2)\) be
measurable, symmetric in \((y,y')\), and non-negative.  A pair \((u,h)\) on
\([0,T]\) is a zero-diffusion corotational weak solution with reaction
coefficient \(\mathcal K\) if
\begin{align}
 u&\in L^\infty(0,T;L^2_\sigma(\T^2))
 \cap L^2(0,T;H^1_\sigma(\T^2))
 \cap C_w([0,T];L^2_\sigma),
 \label{eq:v37-u-class}\\
 Mh&\in L^\infty(0,T;L^1(\Omega\times D\times Y)),\qquad
 Mh\log h\in L^\infty(0,T;L^1),
 \label{eq:v37-h-class}\\
 \nabla_q\sqrt h&\in L^2((0,T)\times\Omega\times D\times Y;M),
 \qquad \tau_Y[h]\in L^2((0,T)\times\Omega),
 \label{eq:v37-fisher-stress-class}
\end{align}
and \(Mh\in C_w([0,T];L^1)\).  It satisfies the momentum identity
\eqref{eq:v35-weak-momentum} and the kinetic identity below for every test
with the measurability and bounded derivative properties stated in Definition
\ref{def:v35-large-data-weak-solution}:
\begin{align}
 &\int Mh(t)\varphi(t)-\int Mh^{\rm in}\varphi(0)
 -\int_0^t\!\int Mh(\partial_r\varphi+u\cdot\nabla_x\varphi
 +\mathsf W(u)q\cdot\nabla_q\varphi)\notag\\
 &\quad=-\frac1{2{\rm Wi}}\int_0^t\!\int
 M\nabla_qh\cdot\nabla_q\varphi
 +\frac12\int_0^t\!\int_{\Omega\times D\times Y^2}
  M\mathcal K(h'-h)(\varphi-\varphi')
 \dd\pi\dd\pi'\dd q\dd x\dd r.
 \label{eq:v37-zero-diffusion-weak-form}
\end{align}
The initial traces in both identities are part of the definition.
\end{definition}

\begin{lemma}[Reaction renormalization identity]
\label{lem:v37-reaction-defect-sign}
Let \(\beta\in C^2([0,\infty))\) be convex with bounded \(\beta'\) and
\(\beta''\ge0\).  For every non-negative \(h\),
\begin{align}
 \int_Y\beta'(h)\mathcal R_0h\dd\pi
 &=-\frac12\int_{Y^2}K_0(h'-h)
 [\beta'(h')-\beta'(h)]\dd\pi\dd\pi'\le0,
 \label{eq:v37-reaction-renormalized-sign}\\
 \int_Y\mathcal R_0h\dd\pi&=0.
 \label{eq:v37-reaction-average-zero}
\end{align}
For the entropy renormalization, approximation by bounded derivatives and
Lemma \ref{lem:fixed-measure-Jeffreys-liminf} show that the corresponding
Jeffreys dissipation is lower semicontinuous under weak \(L^1\) convergence.
No sign is asserted here for the defect associated with an arbitrary convex
\(\beta\): such a sign would additionally require joint convexity (or another
liminf theorem) for
\((r-s)(\beta'(r)-\beta'(s))\).
\end{lemma}

\begin{proof}
Exchange \(y\) and \(y'\), average the two oriented expressions, and use the
symmetry and non-negativity of \(K_0\).  The first identity follows, and its
right-hand side is non-positive by monotonicity of \(\beta'\).  The same
symmetrization with the constant test gives
\eqref{eq:v37-reaction-average-zero}.  Smooth truncation and monotone
convergence give the entropy identity, whose integrand is the jointly convex
Jeffreys integrand treated in Lemma
\ref{lem:fixed-measure-Jeffreys-liminf}.
\end{proof}

\begin{lemma}[Weighted fibre resolvent]
\label{lem:weighted-fibre-resolvent}
Let \(B\) be a skew-symmetric matrix, let
\(K\in L^\infty(D\times Y^2)\) be symmetric and non-negative, and put
\[
 \mathcal H=L^2(D\times Y;M\dd q\dd\pi),\qquad
 \mathcal V=\{v\in\mathcal H:\nabla_qv\in\mathcal H\}.
\]
For \(\alpha>0\), define on \(\mathcal V\)
\begin{align}
 \mathfrak b_{\alpha,B,K}(v,\phi)
 &:=\alpha\int_{D\times Y}Mv\phi
 +\frac1{2{\rm Wi}}\int_{D\times Y}M\nabla_qv\cdot\nabla_q\phi
 -\int_{D\times Y}MvBq\cdot\nabla_q\phi \notag\\
 &\quad+\frac12\int_{D\times Y^2}MK(v'-v)(\phi'-\phi).
 \label{eq:fibre-resolvent-form}
\end{align}
For every \(f\in\mathcal H\), there is a unique \(v\in\mathcal V\) such
that
\begin{equation}
 \mathfrak b_{\alpha,B,K}(v,\phi)=\alpha\int M f\phi
 \qquad(\phi\in\mathcal V).
 \label{eq:fibre-resolvent-equation}
\end{equation}
The resolvent is positivity preserving and mass preserving.  If \(f\ge0\),
then
\begin{align}
 &\alpha\int M(v\log v-v+1)
 +\frac2{{\rm Wi}}\int M|\nabla_q\sqrt v|^2 \notag\\
 &\quad+\frac12\int_{D\times Y^2}MK(v'-v)(\log v'-\log v)
 \le \alpha\int M(f\log f-f+1).
 \label{eq:fibre-resolvent-entropy}
\end{align}
For two data and the same coefficients, the corresponding solutions satisfy
\begin{equation}
 \|v_1-v_2\|_{L^1_M(D\times Y)}
 \le \|f_1-f_2\|_{L^1_M(D\times Y)}.
 \label{eq:fibre-resolvent-L1}
\end{equation}
\end{lemma}

\begin{proof}
The radiality of \(M\) and skew-symmetry of \(B\) give
\(\operatorname{div}_q(MBq)=0\), while the natural weighted boundary flux
vanishes.  Hence
\begin{equation}
 \mathfrak b_{\alpha,B,K}(v,v)
 =\alpha\|v\|_{\mathcal H}^2
 +\frac1{2{\rm Wi}}\|\nabla_qv\|_{\mathcal H}^2
 +\frac12\int_{D\times Y^2}MK|v'-v|^2.
 \label{eq:fibre-resolvent-coercivity}
\end{equation}
The form is bounded on \(\mathcal V\), with a bound depending only on
\(\alpha,{\rm Wi},\sqrt b|B|\), and \(\|K\|_\infty\).  Lax--Milgram proves
existence and uniqueness.

Let \(v^-:=\min\{v,0\}\).  The drift still vanishes when tested by \(v^-\),
and
\((r-s)(r^--s^-)\ge|r^--s^-|^2\).  Thus
\eqref{eq:fibre-resolvent-equation} with \(\phi=v^-\) and \(f\ge0\) forces
\(v^-=0\).  The test \(\phi=1\), justified by weighted approximation, gives
\(\int Mv=\int Mf\).

For positive bounded data, test by a smooth bounded approximation of
\(\log v\).  Convexity of \(s\log s-s+1\), the identity
\(\nabla v\cdot\nabla\log v=4|\nabla\sqrt v|^2\), and pair symmetry give
\eqref{eq:fibre-resolvent-entropy}.  Truncation of the datum and monotone
convergence remove positivity and boundedness.  Finally subtract two
resolvent equations and test by smooth convex approximations of the sign of
the difference.  The skew term vanishes, and both the configurational and
state terms are accretive.  This proves \eqref{eq:fibre-resolvent-L1}.
\end{proof}

\begin{proposition}[Non-autonomous fibre evolution]
\label{prop:nonautonomous-fibre-evolution}
Let \(I=[s,t]\), let
\(B\in L^1(I;\mathfrak{so}(2))\), and let
\[
 K\in L^\infty(I\times D\times Y^2),\qquad K(r,q,y,y')=K(r,q,y',y)\ge0.
\]
Write
\[
 (\mathcal R_{K(r)}v)(q,y)
 :=\int_YK(r,q,y,y')[v(q,y')-v(q,y)]\dd\pi(y').
\]
For every non-negative \(g\in L^1_M(D\times Y)\) of finite relative entropy,
there is a unique renormalized solution
\(V_K\in C(I;L^1_M(D\times Y))\) of
\begin{equation}
 \partial_r(MV_K)=-\nabla_q\cdot(B(r)qMV_K)
 +\frac1{2{\rm Wi}}\nabla_q\cdot(M\nabla_qV_K)
 +M\mathcal R_{K(r)}V_K,\qquad V_K(s)=g.
 \label{eq:nonautonomous-fibre-evolution}
\end{equation}
It is non-negative, preserves \(\int MV_K\), and satisfies the time-integrated
entropy inequality
\begin{align}
 &\int M(V_K(r)\log V_K(r)-V_K(r)+1)
 +\frac2{{\rm Wi}}\int_s^r\!\int M|\nabla_q\sqrt{V_K}|^2\notag\\
 &\quad+\frac12\int_s^r\!\int_{D\times Y^2}
 MK(V_K'-V_K)(\log V_K'-\log V_K)
 \le \int M(g\log g-g+1).
 \label{eq:nonautonomous-fibre-entropy}
\end{align}
For the same kernel it is
an \(L^1_M\)-contraction.  If \(K,L\) are two such kernels and the datum is
the same, then, with \(m_g=\int Mg\),
\begin{equation}
 \sup_{\sigma\in[s,r]}\|V_K(\sigma)-V_L(\sigma)\|_{L^1_M}
 \le2m_g\int_s^r\|K(\sigma)-L(\sigma)\|_{L^\infty(D\times Y^2)}\dd\sigma.
 \label{eq:nonautonomous-fibre-kernel-stability}
\end{equation}
\end{proposition}

\begin{proof}
Because \(B(r)\in\mathfrak{so}(2)\), the matrix equation
\begin{equation}
 Q'(r)=B(r)Q(r),\qquad Q(s)=I,
 \label{eq:fibre-rotation}
\end{equation}
has a unique absolutely continuous solution in \({\rm SO}(2)\).  The ball
\(D\), the radial Maxwellian \(M\), and the weighted Dirichlet form are
rotation invariant.  Define
\[
 Z_K(r,\xi,y):=V_K(r,Q(r)\xi,y),\qquad
 \widetilde K(r,\xi,y,y'):=K(r,Q(r)\xi,y,y').
\]
The Sobolev chain rule for the absolutely continuous map \(Q\) shows that
\eqref{eq:nonautonomous-fibre-evolution} is equivalent to
\begin{equation}
 \partial_r(MZ_K)=\frac1{2{\rm Wi}}\nabla_\xi\cdot(M\nabla_\xi Z_K)
 +M\mathcal R_{\widetilde K(r)}Z_K,
 \qquad Z_K(s)=g.
 \label{eq:rotated-fibre-evolution}
\end{equation}
This identity is first verified for smooth functions and then for energy
solutions by density; no bounded approximation of \(B\) is required.

Let \(g\in\mathcal H\).  Choose partitions
\(s=r_0^n<\cdots<r_{J_n}^n=t\), with mesh \(\tau_n\downarrow0\), and put
\begin{equation}
 \widetilde K_n^m(\xi,y,y')
 :=\frac1{r_m^n-r_{m-1}^n}
 \int_{r_{m-1}^n}^{r_m^n}\widetilde K(r,\xi,y,y')\dd r.
 \label{eq:fibre-time-averaged-kernel}
\end{equation}
Starting from \(Z_n^0=g\), define \(Z_n^m\) by
\begin{align}
 &\frac1{r_m^n-r_{m-1}^n}\int M(Z_n^m-Z_n^{m-1})\phi
 +\frac1{2{\rm Wi}}\int M\nabla_qZ_n^m\cdot\nabla_q\phi
 \notag\\
 &\qquad+\frac12\int_{D\times Y^2}M\widetilde K_n^m
 (Z_n^{m\prime}-Z_n^m)(\phi'-\phi)=0
 \qquad(\phi\in\mathcal V).
 \label{eq:fibre-implicit-Euler}
\end{align}
This is Lemma \ref{lem:weighted-fibre-resolvent} with
\(\alpha=(r_m^n-r_{m-1}^n)^{-1}\) and \(B=0\).  It is positivity and mass
preserving.  Summing its entropy estimate gives, for every grid index \(j\),
\begin{align}
 &\int MF(Z_n^j)
 +\frac2{{\rm Wi}}\sum_{m=1}^j(r_m^n-r_{m-1}^n)
 \int M|\nabla_q\sqrt{Z_n^m}|^2
 \notag\\
 &\quad+\frac12\sum_{m=1}^j(r_m^n-r_{m-1}^n)
 \int_{D\times Y^2}M\widetilde K_n^m
 (Z_n^{m\prime}-Z_n^m)(\log Z_n^{m\prime}-\log Z_n^m)
 \le\int MF(g).
 \label{eq:fibre-discrete-entropy-sum}
\end{align}
Testing instead by \(Z_n^m\) gives
\begin{equation}
 \sup_m\|Z_n^m\|_{\mathcal H}^2
 +\sum_m(r_m^n-r_{m-1}^n)\|Z_n^m\|_{\mathcal V}^2
 \le C\|g\|_{\mathcal H}^2.
 \label{eq:fibre-discrete-L2-bound}
\end{equation}
For the piecewise affine interpolant \(\widehat Z_n\), equation
\eqref{eq:fibre-implicit-Euler} and boundedness of the Markov form imply
\begin{equation}
 \|\partial_r\widehat Z_n\|_{L^2(I;\mathcal V')}
 \le C(1+\|K\|_\infty)\|g\|_{\mathcal H}.
 \label{eq:fibre-discrete-time-bound}
\end{equation}

Weak compactness in the Gelfand triple
\(\mathcal V\subset\mathcal H\subset\mathcal V'\), together with
\eqref{eq:fibre-discrete-time-bound}, gives a subsequence converging weakly in
\(L^2(I;\mathcal V)\) and in \(C_w(I;\mathcal H)\), with the initial trace.
No compact embedding in the continuum state variable is used.  Conditional expectations
\(\widetilde K_n\) converge to \(\widetilde K\) strongly in every finite
\(L^p(I\times D\times Y^2;M)\).  Thus strong--weak pairing identifies the
Markov term in \eqref{eq:fibre-implicit-Euler} first for bounded energy tests;
uniform \(L^\infty\) control of the kernels and density of bounded tests in
\(\mathcal V\) extend the identity to every \(\phi\in\mathcal V\).  The diffusion term and time
derivative pass by weak convergence.  This proves
\eqref{eq:rotated-fibre-evolution}.  The energy identity for the difference of
two limits gives uniqueness, so the whole sequence converges and the standard
non-autonomous form theorem \citep{lions1961equations} yields
\(Z_K\in C(I;\mathcal H)\).

For general entropy data choose non-negative \(g_j\in\mathcal H\) converging
to \(g\) in \(L^1_M\), with convergent entropy.  The discrete \(L^1\)
contraction makes the corresponding solutions Cauchy in
\(C(I;L^1_M)\).  Lower semicontinuity of the weighted Dirichlet integral and
Lemma \ref{lem:fixed-measure-Jeffreys-liminf}, now with the fixed coefficient
\(\widetilde K\), pass \eqref{eq:fibre-discrete-entropy-sum} to the limit and
prove \eqref{eq:nonautonomous-fibre-entropy}.  Rotation back by \(Q(r)\) is an
isometry of every weighted \(L^p_M\) space and preserves all conclusions.

For two solutions with the same kernel, apply smooth Kato tests to their
difference in the rotated equation and pass to the renormalized limit.
Configuration diffusion and the symmetric Markov operator are accretive;
hence the \(L^1_M\) norm is non-increasing.  For two kernels, write
\[
 \mathcal R_KV_K-\mathcal R_LV_L
 =\mathcal R_K(V_K-V_L)+(\mathcal R_K-\mathcal R_L)V_L.
\]
The first term is accretive and
\[
 \| (\mathcal R_K-\mathcal R_L)V_L\|_{L^1_M}
 \le2\|K-L\|_\infty\int MV_L
 =2m_g\|K-L\|_\infty.
\]
Time integration proves \eqref{eq:nonautonomous-fibre-kernel-stability} and
completes the proof.
\end{proof}

\begin{proposition}[Trace-free matrix fibre evolution]
\label{prop:full-drag-fibre-evolution}
Let \(d\in\{2,3\}\), \(I=[s,t]\), and let
\[
 B\in L^2(I;\R^{d\times d}),\qquad \operatorname{tr}B=0
 \quad\hbox{a.e. on }I.
\]
Let \((D,M)\) be an admissible finite-extension Maxwellian and let
\[
 K\in L^1\bigl(I;L^\infty(D\times Y^2)\bigr),\qquad
	K(r,q,y,y')=K(r,q,y',y)\ge0.
	\]
For every non-negative \(g\in L^1_M(D\times Y)\) of finite relative entropy
there is a unique renormalized solution
\(V_{B,K}\in C(I;L^1_M(D\times Y))\) of
\begin{equation}
 \begin{aligned}
 \partial_r(MV_{B,K})
 &=-\nabla_q\cdot(B(r)qMV_{B,K})
 +\frac1{2{\rm Wi}}\nabla_q\cdot(M\nabla_qV_{B,K})\\
 &\quad+M\mathcal R_{K(r)}V_{B,K},
 \qquad V_{B,K}(s)=g.
 \end{aligned}
 \label{eq:full-drag-fibre-evolution}
\end{equation}
It is non-negative and preserves \(m_g:=\int Mg\).  Writing
\[
 \tau[V](r):=\int_{D\times Y}\nabla_qU\otimes q\,MV(r,q,y)
 \dd q\dd\pi(y),
\]
and
\[
 \mathcal J_K[V](r):=\int_{D\times Y^2}
 MK(r)(V'-V)(\log V'-\log V),
\]
it satisfies, for every \(r\in I\),
\begin{align}
 &\int MF(V_{B,K}(r))
 +\frac2{{\rm Wi}}\int_s^r\!\int M|\nabla_q\sqrt{V_{B,K}}|^2
 +\frac12\int_s^r\mathcal J_K[V_{B,K}](\sigma)\dd\sigma
 \notag\\
 &\qquad\le \int MF(g)
 +\int_s^r B(\sigma):\tau[V_{B,K}](\sigma)\dd\sigma,
 \label{eq:full-drag-fibre-entropy-work}
\end{align}
and hence
\begin{align}
 &\sup_{\sigma\in I}\int MF(V_{B,K}(\sigma))
 +\frac1{{\rm Wi}}\int_I\!\int M|\nabla_q\sqrt{V_{B,K}}|^2
 +\frac12\int_I\mathcal J_K[V_{B,K}](\sigma)\dd\sigma
 \notag\\
 &\qquad\le \int MF(g)
 +C_{D,M,{\rm Wi}}m_g\int_I(1+|B|^2)\dd r,
 \label{eq:full-drag-fibre-a-priori}
\end{align}
where \(F(z)=z\log z-z+1\) and
\(J(z',z)=(z'-z)(\log z'-\log z)\).

For the same \(B,K\), the evolution is an \(L^1_M\)-contraction.  If only the
kernel changes, then
\begin{equation}
 \sup_{\sigma\in[s,r]}
 \|V_{B,K}(\sigma)-V_{B,L}(\sigma)\|_{L^1_M}
 \le2m_g\int_s^r
 \|K(\sigma)-L(\sigma)\|_{L^\infty(D\times Y^2)}\dd\sigma.
 \label{eq:full-drag-fibre-kernel-stability}
\end{equation}
	In particular, neither contraction nor kernel stability requires
	skew-symmetry of \(B\).
 The equation is understood with its natural weighted no-flux realization:
 for every \(\phi\in\mathcal V\),
 \begin{equation}
  \langle\partial_rV,\phi\rangle_{\mathcal V',\mathcal V}
  +\frac1{2{\rm Wi}}\int M\nabla_qV\cdot\nabla_q\phi
  -\int MVBq\cdot\nabla_q\phi
  +\frac12\int MK(V'-V)(\phi'-\phi)=0.
  \label{eq:full-drag-fibre-weak-form}
 \end{equation}
	\end{proposition}

\begin{proof}
We give the construction because it is the point at which corotation is
removed.  Put
\(\mathcal H=L^2_M(D\times Y)\) and
\(\mathcal V=\{v\in\mathcal H:\nabla_qv\in\mathcal H\}\).  First suppose
	\(K\in L^\infty(I\times D\times Y^2)\).  For bounded data in \(\mathcal H\),
	Galerkin approximation uses the time-dependent form
\begin{align*}
 a_r(v,\phi)
 &=\frac1{2{\rm Wi}}\int M\nabla_qv\cdot\nabla_q\phi
 -\int MvB(r)q\cdot\nabla_q\phi\\
 &\quad+\frac12\int_{D\times Y^2}MK(r)(v'-v)(\phi'-\phi).
\end{align*}
Writing \(R_D=\sup_{q\in D}|q|\), Young's inequality gives the uniform G\aa rding
estimate
\begin{equation}
 a_r(v,v)\ge\frac1{4{\rm Wi}}\|\nabla_qv\|_{\mathcal H}^2
 -R_D^2{\rm Wi}|B(r)|^2\|v\|_{\mathcal H}^2.
 \label{eq:full-drag-Garding}
\end{equation}
	The negative coefficient belongs to \(L^1(I)\).  Fix an orthonormal
	Galerkin basis contained in \(\mathcal V\).  Since the matrix coefficients
	of the finite system are measurable and integrable in time, the
	Carath\'eodory theorem gives its absolutely continuous solution directly for
	\(B\in L^2(I)\).  Testing by the Galerkin solution \(V_n\) and using
	\eqref{eq:full-drag-Garding} gives
	\begin{align}
	 &\sup_{\sigma\in I}\|V_n(\sigma)\|_{\mathcal H}^2
	 +\frac1{2{\rm Wi}}\int_I\|\nabla_qV_n\|_{\mathcal H}^2
	 \notag\\
	 &\qquad\le
	 \|g_n\|_{\mathcal H}^2
	 +2R_D^2{\rm Wi}\int_I|B|^2\|V_n\|_{\mathcal H}^2
	 \le C_B\|g_n\|_{\mathcal H}^2,
	 \label{eq:full-drag-Galerkin-energy}
	\end{align}
	where \(C_B\) depends only on \(D,{\rm Wi}\), and
	\(\int_I|B|^2\).  Moreover, for \(\|\phi\|_{\mathcal V}\le1\),
	\begin{equation}
	 |\langle\partial_rV_n,\phi\rangle|
	 \le C\bigl(\|\nabla_qV_n\|_{\mathcal H}
	 +(1+\|K(r)\|_\infty)\|V_n\|_{\mathcal H}
	 +|B|\|V_n\|_{\mathcal H}\bigr).
	 \label{eq:full-drag-Galerkin-time-bound}
	\end{equation}
	Thus \((V_n)\) is bounded in
	\(L^2(I;\mathcal V)\cap H^1(I;\mathcal V')\), and hence, after extraction,
	converges weakly to a member of
	\(L^2(I;\mathcal V)\cap C(I;\mathcal H)\).  Every term in
	\eqref{eq:full-drag-fibre-weak-form} passes to the limit: for the drift this
	follows from \(Bq\phi\in L^2(I;\mathcal H)\), and the Markov form is a bounded
	operator on \(\mathcal H\).  This proves existence for \(L^2\) data without
	first imposing a boundedness assumption on \(B\).

	Positivity is proved after the Galerkin limit, where the truncation is an
	admissible \(\mathcal V\)-test.  The Hilbert-triple chain rule with
	\(\phi=-V^-\), the monotonicity of \(z\mapsto\min\{z,0\}\), and
	\eqref{eq:full-drag-Garding} give
	\[
	 \frac12\|V^-(r)\|_{\mathcal H}^2
	 \le C\int_s^r|B|^2\|V^-\|_{\mathcal H}^2,
	\]
so positivity follows from Gronwall.  The constant test gives mass
conservation.  For contraction, subtract two equations and test by
\(\eta_\varepsilon'(V_1-V_2)\), where \(\eta_\varepsilon\) is a smooth convex
approximation of the absolute value.  The drift contains
\[
 \int M(V_1-V_2)Bq\cdot\nabla_q(V_1-V_2)
 \eta_\varepsilon''(V_1-V_2).
\]
	Choose \(\eta_\varepsilon\) so that
	\(|z\eta_\varepsilon''(z)|\le C\) and
	\(z\eta_\varepsilon''(z)\to0\).  The absolute value of the integrand is
	dominated by
	\(CR_D\,M|B||\nabla_q(V_1-V_2)|\), which belongs to \(L^1\) by
	Cauchy--Schwarz in \(I\times D\times Y\).  Dominated convergence therefore
	removes the drift as \(\varepsilon\downarrow0\).  Configuration diffusion and the symmetric
	Markov form are accretive.  This proves \(L^1_M\)-contraction without
skew-symmetry.  The
same calculation applied to
\(\mathcal R_KV_{B,K}-\mathcal R_LV_{B,L}\), followed by
\[
 \|(\mathcal R_K-\mathcal R_L)V_{B,L}\|_{L^1_M}
 \le2m_g\|K-L\|_\infty,
\]
proves \eqref{eq:full-drag-fibre-kernel-stability}.

	For a general non-negative
	\(K\in L^1(I;L^\infty(D\times Y^2))\), put \(K_j=\min\{K,j\}\).
	Since \(\|K_j-K_\ell\|_{L^1_tL^\infty}\to0\), estimate
\eqref{eq:full-drag-fibre-kernel-stability} makes
\(V_j:=V_{B,K_j}\) Cauchy in \(C(I;L^1_M)\); denote its limit by \(V\).
The reaction fluxes converge strongly in \(L^1(I\times D\times Y;M)\).
Indeed, symmetry and positivity give
\(\|\mathcal R_Lw\|_{L^1_M}\le2\|L\|_\infty\|w\|_{L^1_M}\).
With \(k(r)=\|K(r)\|_\infty\), the difference
\(\mathcal R_{K_j}V_j-\mathcal R_KV\) on
	\(\{k\le R\}\) is bounded by
	\[
	 2R\|V_j-V\|_{L^1_M}
	 +2m_g\|K_j-K\|_\infty,
	\]
	whereas its time integral on \(\{k>R\}\) is at most
	\(4m_g\int_{\{k>R\}}k(r)\dd r\).  First let \(j\to\infty\), then
	\(R\to\infty\).  This proves the weak equation, positivity, mass
	conservation, contraction, and kernel stability for the asserted
	\(L^1_tL^\infty\) kernel class.

It remains to record the entropy estimate.  There is no hidden boundary
trace in this step.  For \(\varepsilon>0\), use
\(F_\varepsilon'(z)=\log(z+\varepsilon)\) and set
\[
 \beta_\varepsilon(z)=\int_0^z\frac{s}{s+\varepsilon}\dd s
 =z-\varepsilon\log(1+z/\varepsilon),
 \qquad 0\le\beta_\varepsilon(z)\le z.
\]
The elementary inequality
\[
 \frac{|\beta_\varepsilon'(z)|^2}{\beta_\varepsilon(z)}
 \le \frac{2}{z+\varepsilon},\qquad z>0,
\]
shows that the regularized diffusion controls
\(\sqrt{\beta_\varepsilon(V)}\) in \(H^1_M\).
The convex chain rule gives the drift as
\(\int MBq\cdot\nabla\beta_\varepsilon(V)\).  Lemma
\ref{lem:closed-drift-stress-pairing}, together with the last assertion of
Lemma~\ref{lem:FENE-Hardy-tail}, identifies this term with
\(B:\tau[\beta_\varepsilon(V)]\) and bounds it uniformly in
\(\varepsilon\).  The diffusion term is
\((2{\rm Wi})^{-1}\int M|\nabla V|^2/(V+\varepsilon)\).
Letting \(\varepsilon\downarrow0\), monotone convergence for the diffusion,
the stress-tail estimate for \(\beta_\varepsilon(V)\le V\), and the standard
two-point convex chain rule for the symmetric Markov form give
\begin{equation}
 \frac{\dd}{\dd r}\int MF(V)
 +\frac2{{\rm Wi}}\int M|\nabla_q\sqrt V|^2
 +\frac12\int_{D\times Y^2}MKJ(V',V)
 \le B:\tau[V].
 \label{eq:full-drag-fibre-differential-entropy}
\end{equation}
In particular, the limiting drift contribution is
\[
 \int MBq\cdot\nabla_qV
 =-\int V\,\operatorname{div}_q(MBq)
 =B:\int M(\nabla_qU\otimes q)V.
\]
The equality is precisely \eqref{eq:closed-drift-stress-pairing}; no separate
trace is assigned to either singular configuration flux.  Notice that the
tensor order fixed after \eqref{eq:v35-large-data-stress} is needed here when
\(U\) is non-radial.
Lemma~\ref{lem:FENE-Hardy-tail}, applied at each time, yields
\begin{align*}
 |B:\tau[V]|
 &\le |B|m_g^{1/2}
 \left(C_{D,M}m_g+C_{D,M}\int M|\nabla_q\sqrt V|^2\right)^{1/2}\\
 &\le\frac1{{\rm Wi}}\int M|\nabla_q\sqrt V|^2
 +C_{D,M,{\rm Wi}}m_g(1+|B|^2).
\end{align*}
	Integration proves \eqref{eq:full-drag-fibre-entropy-work}--
	\eqref{eq:full-drag-fibre-a-priori} for bounded positive data in
	\(\mathcal H\).  For general finite-entropy data, choose interior
	mollifications of two-sided truncations, followed by a vanishing positive
	constant and a mass correction, so that
	\(g_j\to g\) in \(L^1_M\) and
	\(\int MF(g_j)\to\int MF(g)\).  Contraction makes the corresponding solutions
	Cauchy in \(C(I;L^1_M)\).  The Hardy estimate bounds
	\(\tau[V_j]\) in \(L^2(I)\); truncation of \(W\), strong \(L^1_M\)
	convergence, and \eqref{eq:FENE-stress-tail} identify its weak limit as
	\(\tau[V]\).  Since \(B\in L^2(I)\), the work integrals converge.  Lower
	semicontinuity of the entropy, Fisher information, and symmetric Jeffreys
	form then gives the two asserted inequalities.  The uniform
	\(C(I;L^1_M)\) limit supplies both strong endpoint traces and the claimed
	renormalized solution.  For \(K\in L^1_tL^\infty\), write
\(V_R=V_{B,K\wedge R}\).  Kernel stability gives
\(V_R\to V_{B,K}\) strongly in \(C(I;L^1_M)\).  Fix \(L<\infty\).  For
\(R\ge L\), discard the non-negative part weighted by
\((K\wedge R)-(K\wedge L)\), and apply Ioffe lower semicontinuity on the fixed
finite measure
\[
 (K\wedge L)M\,\dd r\dd q\dd\pi(y)\dd\pi(y').
\]
This yields the Jeffreys term weighted by \(K\wedge L\) for the limit
solution.  The stress-tail argument above gives
\(\tau[V_R]\rightharpoonup\tau[V_{B,K}]\) in \(L^2(I)\), hence convergence of
the work against \(B\in L^2(I)\).  Finally let \(L\uparrow\infty\) and use
monotone convergence.  This proves the entropy and Jeffreys bounds for the
full kernel without coupling the two truncation parameters.
\end{proof}

\begin{lemma}[Kinetic pullback by an incompressible regular Lagrangian flow]
\label{lem:kinetic-RLF-pullback}
Let \(d\in\{2,3\}\).  If \(\Omega=\T^d\), assume
\(u\in L^1(I;W^{1,1}(\T^d;\R^d))\); if \(\Omega\) is a bounded Lipschitz
domain, assume \(u\in L^1(I;W^{1,1}_0(\Omega;\R^d))\).  Suppose that \(u\)
is divergence free, let \(X\) be its measure-preserving regular
Lagrangian flow, and let
\[
 \mathsf G(u)\in\{\nabla_xu,\tfrac12(\nabla_xu-\nabla_xu^T)\},
 \qquad B(r,a)=\mathsf G(u)(r,X(r,a)).
\]
In the bounded-domain case the zero extension of \(u\) is used to construct
\(X\); it leaves \(\Omega\) invariant up to a null set.  Let
\(\mathcal K\) be measurable, symmetric in \((y,y')\), and non-negative.
Suppose \(H\ge0\), \(M\mathcal R_{\mathcal K}H\in L^1\), and
\begin{equation}
 \begin{aligned}
 &MH\in C_w(I;L^1),\qquad
 \lim_{r\downarrow s}\|H(r)-H(s)\|_{L^1_M}=0,\\
 &\rho[H]\in L^\infty(I\times\Omega),\qquad
 M\nabla_qH\in L^1(I\times\Omega\times D\times Y).
 \end{aligned}
 \label{eq:RLF-pullback-integrability}
\end{equation}
Then \(H\) solves the Eulerian kinetic equation with coefficient
\(\mathcal K\) if and only if
\[
 \widetilde H(r,a,q,y):=H(r,X(r,a),q,y)
\]
solves, for almost every \(a\), the fibre equation
\begin{align}
 \partial_r(M\widetilde H)
 &=-\nabla_q\cdot(B(r,a)qM\widetilde H)
 +\frac1{2{\rm Wi}}\nabla_q\cdot(M\nabla_q\widetilde H)
 +M\mathcal R_{\widetilde{\mathcal K}(r,a)}\widetilde H,
 \label{eq:RLF-general-fibre-equation}\\
 \widetilde{\mathcal K}(r,a,q,y,y')
 &:=\mathcal K(r,X(r,a),q,y,y').
 \notag
\end{align}
The equivalence holds in the renormalized class.  It preserves the displayed
strong initial trace and every other strong \(L^1_M\) time trace possessed by
\(H\).
\end{lemma}

\begin{proof}
In the bounded-domain case, the zero extension \(\widehat u\) belongs to
\(L^1(I;W^{1,1}(\R^d))\) and is divergence free in distributions because the
trace of \(u\) vanishes.  Both \(\mathbf1_\Omega\) and
\(\mathbf1_{\R^d\setminus\Omega}\) are stationary renormalized solutions of
\(\partial_r\zeta+\operatorname{div}(\widehat u\zeta)=0\).  Uniqueness in the
bounded renormalized class, equivalently the regular-flow representation,
therefore gives
\begin{equation}
 \mathbf1_\Omega(X(r,a))=\mathbf1_\Omega(a)
 \quad\text{for a.e. }a\text{ and every }r\in I.
 \label{eq:RLF-domain-invariance}
\end{equation}
Thus whole-space mollification of zero extensions may be used below and then
restricted to \(\Omega\).

Let \(\psi(q,y)\) be bounded, continuously differentiable in \(q\), and simple
in \(y\), and set
\[
 m_\psi(r,x):=\int_{D\times Y}MH(r,x,q,y)\psi(q,y)\dd q\dd\pi(y).
\]
The Eulerian kinetic identity implies the scalar continuity equation
\begin{equation}
 \partial_rm_\psi+\operatorname{div}_x(um_\psi)=S_\psi
 \quad\text{in }\mathcal D'(I\times\Omega),
 \label{eq:RLF-scalar-probe-equation}
\end{equation}
where
\begin{align}
 S_\psi
 &=\int MH\,\mathsf G(u)q\cdot\nabla_q\psi
 -\frac1{2{\rm Wi}}\int M\nabla_qH\cdot\nabla_q\psi
 \notag\\
 &\quad-\frac12\int_{D\times Y^2}
 M\mathcal K(H'-H)(\psi'-\psi).
 \label{eq:RLF-probe-source}
\end{align}
All integrals in \eqref{eq:RLF-probe-source} are over the unmarked kinetic
variables.  They belong to \(L^1(I\times\Omega)\): the drift is bounded by
\(C_\psi|\mathsf G(u)|\rho[H]\), the diffusion by
\eqref{eq:RLF-pullback-integrability}, and the reaction by the assumed
\(L^1\) reaction flux.

Moreover \(|m_\psi|\le\|\psi\|_\infty\rho[H]\), so
\(m_\psi\in L^\infty\).  Apply spatial mollification to
\eqref{eq:RLF-scalar-probe-equation}, using the zero-extended equation in the
bounded-domain case and periodic mollification on the torus.  This produces
the DiPerna--Lions commutator
\[
 c_\varepsilon
 =u\cdot\nabla_x(m_\psi*\varrho_\varepsilon)
 -(u\cdot\nabla_xm_\psi)*\varrho_\varepsilon,
 \qquad \|c_\varepsilon\|_{L^1(I\times\Omega)}\longrightarrow0.
\]
Writing the commutator as an integral of increments of \(u\) proves the stated
convergence from \(\nabla u\in L^1\), \(m_\psi\in L^\infty\), and
\(\operatorname{div}u=0\).  Integrating the mollified equation along \(X\),
using measure preservation for \(S_\psi\), and then passing
\(\varepsilon\downarrow0\), gives for almost every label
\begin{equation}
 m_\psi(r,X(r,a))-m_\psi(s,a)
 =\int_s^rS_\psi(\sigma,X(\sigma,a))\dd\sigma.
 \label{eq:RLF-probe-characteristic-identity}
\end{equation}
Weak continuity and the strong initial trace in
\eqref{eq:RLF-pullback-integrability} fix a common representative at the
initial time and the absolutely continuous representative of every probe at
later times.  Choose a countable
dense set of smooth \(q\)-tests and simple \(y\)-tests
from a countable generating Borel algebra.  Intersecting their full-measure
label sets and using bounded convergence and the monotone-class theorem yields
\eqref{eq:RLF-general-fibre-equation} for almost every common label.

Conversely, multiply the fibre identity by bounded label tests and integrate
in \(a\).  For a smooth Eulerian test \(\varphi\), the regular-Lagrangian chain
rule gives
\[
 \frac{\dd}{\dd r}\varphi(r,X(r,a),q,y)
 =\partial_r\varphi(r,X(r,a),q,y)
 +u(r,X(r,a))\cdot\nabla_x\varphi(r,X(r,a),q,y)
\]
for almost every trajectory.  The source bounds above justify Fubini and
passage to general admissible tests.  Changing variables
\(x=X(r,a)\), using preservation of Lebesgue measure, gives the Eulerian weak
identity.  The same argument applied to bounded renormalizations proves the
equivalence in the renormalized class.  Finally, composition with a
measure-preserving map is an \(L^1_M\)-isometry at every time.  This proves
preservation of the strong initial trace and, whenever present, any further
strong time trace.
\end{proof}

\begin{theorem}[Lagrangian state-fibre well-posedness]
\label{thm:lagrangian-state-fibre}
Let \(I=[s,t]\subset[0,T]\), let
\[
 u\in L^2(I;H^1_\sigma(\T^2)),
 \qquad \mathsf W(u)=\tfrac12(\nabla_xu-\nabla_xu^T),
\]
and let \(X(r,a)\) be the regular Lagrangian flow of \(u\), normalized by
\(X(s,a)=a\).  Suppose
\begin{equation}
 K\in L^\infty(I\times D\times Y^2),\qquad
 K(r,q,y,y')=K(r,q,y',y)\ge0.
 \label{eq:lagrangian-kernel-class}
\end{equation}
For every \(g\ge0\) with finite mass and relative entropy and with
\(\rho_g\in L^\infty(\T^2)\),
\begin{equation}
 \mathfrak m:=\int_{\T^2\times D\times Y}Mg<\infty,
 \qquad
 \int_{\T^2\times D\times Y}M(g\log g-g+1)<\infty,
 \label{eq:lagrangian-initial-class}
\end{equation}
there is a unique renormalized solution \(H_K\) of
\begin{align}
 \partial_r(MH_K)+u\cdot\nabla_x(MH_K)
 &=-\nabla_q\cdot(\mathsf W(u)qMH_K)
   +\frac1{2{\rm Wi}}\nabla_q\cdot(M\nabla_qH_K)
   +M\mathcal R_{K(r)}H_K,
 \label{eq:lagrangian-lift-equation}\\
 H_K(s)&=g,
 \notag
\end{align}
where
\[
 (\mathcal R_{K(r)}v)(q,y)
 :=\int_YK(r,q,y,y')[v(q,y')-v(q,y)]\dd\pi(y').
\]
It is non-negative, belongs to
\(C(I;L^1_M(\T^2\times D\times Y))\), preserves mass, and satisfies
\begin{align}
 &\int M(H_K(r)\log H_K(r)-H_K(r)+1)
 +\frac2{{\rm Wi}}\int_s^r\!\int M|\nabla_q\sqrt{H_K}|^2
 \notag\\
 &\quad+\frac12\int_s^r\!\int_{\T^2\times D\times Y^2}
 MK(\,H_K'-H_K\,)(\log H_K'-\log H_K)
 \le \int M(g\log g-g+1).
 \label{eq:lagrangian-lift-entropy}
\end{align}
If \(\rho_g(a)=\int_{D\times Y}Mg(a,q,y)\dd q\dd\pi(y)\), then
\begin{equation}
 \rho[H_K](r,X(r,a))=\rho_g(a)
 \quad\text{for a.e. }(r,a).
 \label{eq:lagrangian-number-density}
\end{equation}

The solution is Lagrangian in the following precise sense.  With
\begin{equation}
 B(r,a):=\mathsf W(u)(r,X(r,a)),
 \qquad
 \widetilde H_K(r,a,q,y):=H_K(r,X(r,a),q,y),
 \label{eq:lagrangian-pullback}
\end{equation}
the pullback solves, for almost every \(a\),
\begin{equation}
 \partial_r(M\widetilde H_K)
 =-\nabla_q\cdot(B(r,a)qM\widetilde H_K)
 +\frac1{2{\rm Wi}}\nabla_q\cdot(M\nabla_q\widetilde H_K)
 +M\mathcal R_{K(r)}\widetilde H_K.
 \label{eq:lagrangian-fibre-equation}
\end{equation}
Conversely, the measurable family of fibre solutions in
\eqref{eq:lagrangian-fibre-equation}, pushed forward by \(X\), is the
renormalized Eulerian solution of \eqref{eq:lagrangian-lift-equation}.

For the same kernel, initial data \(g,\widehat g\), and corresponding solutions
\(H_K,\widehat H_K\),
\begin{equation}
 \int M|H_K(r)-\widehat H_K(r)|
 \le \int M|g-\widehat g|.
 \label{eq:lagrangian-L1-contraction}
\end{equation}
For two kernels \(K,L\) satisfying \eqref{eq:lagrangian-kernel-class} and the
same datum \(g\),
\begin{equation}
 \sup_{\sigma\in[s,r]}\int M|H_K-H_L|(\sigma)
 \le2\mathfrak m\int_s^r
 \|K(\sigma)-L(\sigma)\|_{L^\infty(D\times Y^2)}\dd\sigma.
 \label{eq:lagrangian-kernel-stability}
\end{equation}
Finally, \(\int_YH_K\dd\pi\) is independent of \(K\): it is the unique
renormalized scalar solution obtained from \eqref{eq:lagrangian-lift-equation}
after deleting the reaction and averaging the initial datum in \(Y\).
\end{theorem}

\begin{proof}
\medskip\noindent\textit{1. The regular Lagrangian flow.}\par\noindent
Since the torus has finite measure,
\(u\in L^1(I;W^{1,1}(\T^2))\) and \(\operatorname{div}u=0\).  The
DiPerna--Lions flow theorem \citep{diperna1989ordinary} gives a unique regular
Lagrangian flow \(X\).  It preserves Lebesgue measure, and therefore
\begin{equation}
 \int_I\!\int_{\T^2}|B(r,a)|\dd a\dd r
 =\int_I\!\int_{\T^2}|\mathsf W(u)(r,x)|\dd x\dd r<\infty.
 \label{eq:lagrangian-B-integrability}
\end{equation}
Thus \(B(\cdot,a)\in L^1(I)\) for almost every \(a\).  Because \(B\) is
skew-symmetric and \(M\) is radial,
\begin{equation}
 \operatorname{div}_q(Bq)=0,
 \qquad Bq\cdot\nabla_qU=0,
 \qquad \operatorname{div}_q(MBq)=0.
 \label{eq:lagrangian-skew-identities}
\end{equation}

\medskip\noindent\textit{2. Construction on one fibre.}\par\noindent
For almost every \(a\), Proposition
\ref{prop:nonautonomous-fibre-evolution} applies to
\(B(\cdot,a)\) and \(g(a,\cdot,\cdot)\).  To obtain the family without a
measurable-selection assumption, perform the rotated construction once on the
product space \(\T^2_a\times D\times Y\).  Write
\(B(r,a)=b(r,a)J\), where \(J\) is the fixed quarter-turn, and set
\begin{equation}
 \theta(r,a)=\int_s^r b(\sigma,a)\dd\sigma,
 \qquad Q(r,a)=\exp(\theta(r,a)J).
 \label{eq:product-fibre-rotation}
\end{equation}
The map \((r,a)\mapsto Q(r,a)\) is jointly measurable and is absolutely
continuous in \(r\) for almost every \(a\).  With
\[
 \widehat K(r,a,\xi,y,y')=K(r,Q(r,a)\xi,y,y'),
\]
the product-space equation has no \(q\)-drift and a bounded measurable
symmetric kernel.  To make the product construction explicit, put
\[
 \mathbb H=L^2(\T^2_a;\mathcal H),\qquad
 \mathbb V=L^2(\T^2_a;\mathcal V).
\]
On a time cell \(I_j\), let \(\widehat K_j(a)\) be the time average of
\(\widehat K\) and define on \(\mathbb V\)
\[
 \mathbb B_j(v,\phi)
 :=\int_{\T^2}\mathfrak b_{|I_j|^{-1},0,\widehat K_j(a)}
 (v(a),\phi(a))\dd a.
\]
The coercivity identity \eqref{eq:fibre-resolvent-coercivity} is uniform in
\(a\), and boundedness uses only \(\|K\|_\infty\).  Lax--Milgram therefore
gives a unique product resolvent for \(L^2\) data.  Testing with
\(\mathbf1_E(a)\) times the usual truncations, for arbitrary Borel
\(E\subset\T^2\), localizes positivity, mass conservation, entropy, and
contraction to almost every label.  Implicit Euler time averages thus define
the solution directly on the product space.  Approximation of the datum by
bounded \(L^2_M\) functions, the product-space contraction, and
\eqref{eq:nonautonomous-fibre-entropy} give a jointly measurable limit in
\(C(I;L^1_M(\T^2_a\times D\times Y))\).  Disintegration and fibrewise
uniqueness identify it with Proposition
\ref{prop:nonautonomous-fibre-evolution}; rotation back gives
\eqref{eq:lagrangian-fibre-equation}.  Integrating the fibre entropy inequality
in \(a\) proves \eqref{eq:lagrangian-lift-entropy}, and fibre mass conservation
gives \eqref{eq:lagrangian-number-density}.

\medskip\noindent\textit{3. Contraction and kernel stability.}\par\noindent
The fibrewise contraction in Proposition
\ref{prop:nonautonomous-fibre-evolution}, integrated in \(a\), gives
\eqref{eq:lagrangian-L1-contraction} and uniqueness in the renormalized class.
Its kernel estimate, followed by Tonelli's theorem and
\(\int_{\T^2}\rho_g=\mathfrak m\), gives
\eqref{eq:lagrangian-kernel-stability} with the displayed constant.

\medskip\noindent\textit{4. Return to Eulerian variables.}\par\noindent
The bounded number density is used at this point and nowhere in the fibre
construction.  Fibre mass conservation and
\eqref{eq:lagrangian-B-integrability} give
\begin{equation}
 \int_I\!\int_{\T^2}|B(r,a)|\rho_g(a)\dd a\dd r
 \le \|\rho_g\|_{L^\infty}\|B\|_{L^1(I\times\T^2)}<\infty.
 \label{eq:weighted-drift-assembly}
\end{equation}
Hence the configurational drift belongs to \(L^1\) after the fibres are
assembled, so every term in the Eulerian weak formulation is defined.
The Fisher estimate and mass conservation give
\(M\nabla_q\widetilde H_K\in L^1\), so all hypotheses of Lemma
\ref{lem:kinetic-RLF-pullback} hold.  That lemma transforms the fibre weak
formulation into \eqref{eq:lagrangian-lift-equation} and pulls every
renormalized Eulerian solution back to the same fibre problem.  Fibrewise
contraction then proves Eulerian
uniqueness and shows that the strong \(L^1\) endpoint traces are preserved by
the push-forward.  Testing the fibre
equation by a function independent of \((q,y)\) shows that its total
\((q,y)\)-mass is constant, which is
\eqref{eq:lagrangian-number-density}.  Integrating only in \(Y\) cancels the
reaction.  The contraction argument with no state variable gives uniqueness
of the scalar equation, so the state average is independent of \(K\).
\end{proof}

\begin{remark}[Where bounded number density is used]
\label{rem:lagrangian-density-threshold}
Proposition \ref{prop:nonautonomous-fibre-evolution} and the rotated
product-space construction require only finite fibre mass and entropy.  The
Eulerian assembly additionally requires
\begin{equation}
 \int_I\!\int_{\T^2}|B(r,a)|\rho_g(a)\dd a\dd r<\infty.
 \label{eq:sharp-weighted-drift-threshold}
\end{equation}
This is the exact velocity-dependent threshold.  The stated condition
\(\rho_g\in L^\infty\) is a velocity-independent sufficient hypothesis and is
propagated by the incompressible flow.  Without
\eqref{eq:sharp-weighted-drift-threshold}, the rotated fibres still exist for
almost every label, but their drift need not define an Eulerian \(L^1\) flux;
the theorem therefore makes no Eulerian claim in that larger class.
\end{remark}

\begin{lemma}[Scalar projection and prescribed-velocity uniqueness]
\label{lem:full-drag-scalar-projection}
Let \(d\in\{2,3\}\), let \(\Omega=\T^d\) or a bounded Lipschitz domain, and
let \((D,M)\) be an admissible finite-extension Maxwellian.  Assume
\(u\in L^2(I;H^1_\sigma(\T^d))\) on the torus and
\(u\in L^2(I;H^1_0(\Omega;\R^d)\cap H_\sigma(\Omega))\) on a bounded domain.
Let \(X\) be its measure-preserving regular Lagrangian flow and set
\(B=\nabla u\circ X\).  In the class
\begin{equation}
 \begin{aligned}
 &Mv\in C_w(I;L^1(\Omega\times D)),\qquad
 \lim_{r\downarrow s}\|v(r)-v(s)\|_{L^1_M}=0,\\
 &\rho[v]\in L^\infty(I\times\Omega),\qquad
 M\nabla_qv\in L^1.
 \end{aligned}
 \label{eq:scalar-projection-class}
\end{equation}
the scalar full-drag Fokker--Planck equation with prescribed velocity \(u\)
has at most one non-negative renormalized solution for each initial datum.
If \(H\) is a state-resolved solution with a symmetric conservative Markov
kernel, then \(\bar H=\int_YH\dd\pi\) is that scalar solution and is therefore
independent of the kernel.
\end{lemma}

\begin{proof}
Lemma~\ref{lem:kinetic-RLF-pullback}, with the state space reduced to one
point, pulls two scalar solutions with the same datum to fibres driven by the
same trace-free matrix \(B(\cdot,a)\).  The \(L^1_M\)-contraction in
Proposition~\ref{prop:full-drag-fibre-evolution}, integrated in the label,
identifies the two solutions.  For a state-resolved solution, integration in
\(Y\) cancels the reaction because symmetry and Fubini give
\(\int_Y\mathcal R_KH\dd\pi=0\).  The remaining terms are exactly the scalar
equation, and the preceding uniqueness proves the last assertion.
\end{proof}

\begin{theorem}[Full-drag Lagrangian state-fibre lift]
\label{thm:full-drag-state-fibre}
Let \(d\in\{2,3\}\).  Let \(\Omega=\T^d\), or let \(\Omega\subset\R^d\)
be bounded Lipschitz.  Let \((D,M)\) be an admissible finite-extension
Maxwellian, and set
\[
 V_\sigma(\Omega)=
 \begin{cases}
 H^1_\sigma(\T^d),&\Omega=\T^d,\\
 H^1_0(\Omega;\R^d)\cap H_\sigma(\Omega),&\partial\Omega\ne\varnothing.
 \end{cases}
\]
Let \(u\in L^2(I;V_\sigma(\Omega))\), let \(X\) be its
measure-preserving regular Lagrangian flow, and put
\[
 B(r,a)=\nabla_xu(r,X(r,a)).
\]
Suppose \(g\ge0\) has finite mass and relative entropy and
\(\rho_g\in L^\infty(\Omega)\).  For every measurable label-dependent kernel
\begin{equation}
 \mathsf K(r,a,q,y,y')=\mathsf K(r,a,q,y',y)\ge0,
 \label{eq:full-drag-label-kernel}
\end{equation}
such that, with
\(\kappa(r,a)=\|\mathsf K(r,a)\|_{L^\infty(D\times Y^2)}\),
\begin{equation}
 \kappa(\cdot,a)\in L^1(I)\quad\hbox{for a.e. }a,
 \qquad
 \int_{I\times\Omega_a}\rho_g(a)\kappa(r,a)\dd a\dd r<\infty,
 \label{eq:full-drag-integrable-label-kernel}
\end{equation}
there is a unique jointly measurable family
\(\widetilde H_{\mathsf K}\) such that, for almost every \(a\),
\begin{align}
 \partial_r(M\widetilde H_{\mathsf K})
 &=-\nabla_q\cdot(B(r,a)qM\widetilde H_{\mathsf K})
 +\frac1{2{\rm Wi}}\nabla_q\cdot(M\nabla_q\widetilde H_{\mathsf K})
 +M\mathcal R_{\mathsf K(r,a)}\widetilde H_{\mathsf K},
 \label{eq:full-drag-label-fibre}\\
 \widetilde H_{\mathsf K}(s,a)&=g(a).
 \notag
\end{align}
It preserves \(\rho_g(a)\), belongs to
\(C(I;L^1_M(\Omega_a\times D\times Y))\), and satisfies the labelwise
contraction and kernel estimate
\begin{equation}
 \sup_{\sigma\in[s,r]}
 \int_{D\times Y}M
 |\widetilde H_{\mathsf K}-\widetilde H_{\mathsf L}|(\sigma,a)
 \le2\rho_g(a)\int_s^r
 \|\mathsf K-\mathsf L\|_{L^\infty(D\times Y^2)}(\theta,a)\dd\theta.
 \label{eq:full-drag-label-kernel-stability}
\end{equation}
Define
\begin{equation}
 H_{\mathsf K}(r,X(r,a),q,y)
 =\widetilde H_{\mathsf K}(r,a,q,y),\qquad
 K^E(r,X(r,a),q,y,y')=\mathsf K(r,a,q,y,y').
 \label{eq:full-drag-Eulerian-pushforward}
\end{equation}
Then \(H_{\mathsf K}\) is the unique renormalized solution, for the prescribed
velocity \(u\), of
\begin{equation}
 \partial_r(MH)+u\cdot\nabla_x(MH)
 =-\nabla_q\cdot((\nabla_xu)qMH)
 +\frac1{2{\rm Wi}}\nabla_q\cdot(M\nabla_qH)
 +M\mathcal R_{K^E}H.
 \label{eq:full-drag-Eulerian-lift}
\end{equation}
with \(H(s)=g\), in the class
\eqref{eq:RLF-pullback-integrability} with bounded number density.
Its state average is independent of \(K^E\) and is the unique renormalized
solution of the scalar full-drag Fokker--Planck equation with datum
\(\bar g=\int_Yg\dd\pi\).  Moreover,
\begin{align}
 &\sup_{\sigma\in I}\int_{\Omega\times D\times Y}MF(H_{\mathsf K}(\sigma))
 +\frac1{{\rm Wi}}\int_I\!\int M|\nabla_q\sqrt{H_{\mathsf K}}|^2
 +\frac12\int_I\!\int_{\Omega\times D\times Y^2}
 MK^EJ(H_{\mathsf K}',H_{\mathsf K})
 \notag\\
 &\qquad\le
 \int_{\Omega\times D\times Y}MF(g)
 +C_{D,M,{\rm Wi}}\|\rho_g\|_\infty
 \left(|I|\,|\Omega|+\|\nabla_xu\|_{L^2(I\times\Omega)}^2\right).
 \label{eq:full-drag-lift-entropy-bound}
\end{align}
In particular,
\(\tau_Y[H_{\mathsf K}]\in L^2(I\times\Omega)\), and
\[
 \tau_Y[H_{\mathsf K}]
 =\int_D\nabla_qU\otimes q\,M
 \left(\int_YH_{\mathsf K}\dd\pi\right)\dd q.
\]
\end{theorem}

\begin{proof}
The zero extension of a field in \(V_\sigma(\Omega)\) is divergence free on
\(\R^d\) in the bounded-domain case.  DiPerna--Lions theory therefore gives a
measure-preserving regular Lagrangian flow which leaves \(\Omega\) invariant
up to null sets.  Change of variables along \(X\) gives
\begin{equation}
 \int_{\Omega}\int_I|B(r,a)|^2\dd r\dd a
 =\int_I\int_\Omega|\nabla_xu(r,x)|^2\dd x\dd r<\infty.
 \label{eq:full-drag-B-Fubini}
\end{equation}
Hence \(B(\cdot,a)\in L^2(I)\) and
\(\operatorname{tr}B(\cdot,a)=0\) for almost every label.  Proposition
\ref{prop:full-drag-fibre-evolution} applies on every such fibre.

We spell out measurability, since no measurable-selection assertion is
needed.  First replace \(g\) by the jointly measurable bounded truncations
\(g_j=\min\{g,j\}\), and use one fixed countable Galerkin basis of
\(\mathcal V\).  For every finite dimension, the matrix entries obtained by
	integrating \(\min\{\mathsf K,j\}\) against two basis functions are jointly
	measurable and bounded; the drift entries are jointly measurable and belong
to \(L^2(I\times\Omega_a)\).  Truncating \(B\) gives finite Carath\'eodory ODEs
whose solutions are jointly measurable.  Estimates
\eqref{eq:full-drag-Galerkin-energy} and
\eqref{eq:full-drag-Galerkin-time-bound}, first integrated over labels on
finite truncation levels, give the weak limits in the Galerkin dimension.
The drift truncations then converge strongly in \(L^2(I\times\Omega_a)\), so
their weak formulations converge by strong--weak pairing.  Finally
\(g_j\to g\) in \(L^1_M(\Omega_a\times D\times Y)\), and the integrated
\(L^1\)-contraction makes the corresponding solutions Cauchy uniformly in
	time.  Letting the kernel truncation tend to infinity uses the \(L^1_t\)
	kernel stability from Proposition~\ref{prop:full-drag-fibre-evolution};
	condition \eqref{eq:full-drag-integrable-label-kernel} permits the same limit
	after integration in the label.  The resulting limit is jointly measurable
	and unique on almost every fibre.

For each label, the time-continuity supplied by Proposition
\ref{prop:full-drag-fibre-evolution} and the bound
\(\|\widetilde H(r,a)-\widetilde H(\sigma,a)\|_{L^1_M}\le2\rho_g(a)\)
allow dominated convergence in \(a\).  This proves
\(C(I;L^1_M(\Omega_a\times D\times Y))\), rather than only weak continuity.
Formula
\eqref{eq:full-drag-label-kernel-stability} is exactly
\eqref{eq:full-drag-fibre-kernel-stability} with
\(m_g=\rho_g(a)\).

The drift source in Lemma \ref{lem:kinetic-RLF-pullback} is integrable because
\[
 \int_I\int_\Omega|\nabla_xu|\,\rho[H_{\mathsf K}]
 \le |I|^{1/2}|\Omega|^{1/2}\|\rho_g\|_\infty
 \|\nabla_xu\|_{L^2(I\times\Omega)}.
\]
The Fisher estimate gives \(M\nabla_qH_{\mathsf K}\in L^1\).  The Eulerian
reaction flux is integrable because
\[
 \int_{I\times\Omega\times D\times Y}M
 |\mathcal R_{K^E}H_{\mathsf K}|
 \le2\int_{I\times\Omega_a}\rho_g(a)\kappa(r,a)\dd a\dd r<\infty.
\]
That lemma,
with \(\mathsf G(u)=\nabla_xu\), proves
\eqref{eq:full-drag-Eulerian-lift} and preserves the endpoint traces.
Conversely, it pulls every renormalized Eulerian solution into the unique
fibre evolution, proving uniqueness for prescribed \(u,K^E\).

Lemma~\ref{lem:full-drag-scalar-projection} shows that state integration
cancels the symmetric reaction and identifies the average with the unique
scalar prescribed-velocity solution; hence it is independent of the kernel.
Integrating \eqref{eq:full-drag-fibre-a-priori} in \(a\), using
\eqref{eq:full-drag-B-Fubini} and
\(\rho_g(a)\le\|\rho_g\|_\infty\), proves
\eqref{eq:full-drag-lift-entropy-bound}.  Finally Lemma
\ref{lem:FENE-Hardy-tail} and the transported \(L^\infty\) bound on the
number density give the asserted \(L^2\) stress estimate.  State integration
gives the displayed stress identity.
\end{proof}

\begin{proposition}[Conditional state entropy contraction]
\label{prop:full-drag-conditional-entropy}
Under the hypotheses of Theorem~\ref{thm:full-drag-state-fibre}, put
\(\bar H=\int_YH_{\mathsf K}\dd\pi\) and
\begin{equation}
 \mathcal C_Y[H\mid\bar H](r)
 :=\int_{\Omega\times D\times Y}
 MH(r)\log\frac{H(r)}{\bar H(r)}.
 \label{eq:full-drag-conditional-entropy}
\end{equation}
The integrand is set to zero on \(\{\bar H=0\}\), where
\(H=0\) for \(\pi\)-almost every state.  Then, for every \(r\in I\),
\begin{align}
 &\mathcal C_Y[H\mid\bar H](r)
 +\frac1{2{\rm Wi}}\int_s^r\!\int
 MH\left|\nabla_q\log\frac{H}{\bar H}\right|^2
 \notag\\
 &\qquad
 +\frac12\int_s^r\!\int_{\Omega\times D\times Y^2}
 MK^E(H'-H)(\log H'-\log H)
 \le \mathcal C_Y[g\mid\bar g].
 \label{eq:full-drag-conditional-entropy-decay}
\end{align}
Neither the velocity gradient nor its symmetric part occurs on the right-hand
side.
\end{proposition}

\begin{proof}
Work first with bounded strictly positive fibre approximations.  Differentiate
the conditional entropy in \eqref{eq:full-drag-conditional-entropy}, use the
equations for \(H\) and \(\bar H\), and sum in the state variable.  The
full-drag contribution vanishes algebraically:
\[
 \int_Y MH\,Bq\cdot\nabla_q\log\frac H{\bar H}\dd\pi
 =MBq\cdot\left(\nabla_q\bar H-
 \bar H\nabla_q\log\bar H\right)=0.
\]
The diffusion contribution is
\begin{align*}
 &-\frac1{2{\rm Wi}}\int M
 \left(\int_YH|\nabla_q\log H|^2\dd\pi
 -\bar H|\nabla_q\log\bar H|^2\right)\\
 &\qquad=-\frac1{2{\rm Wi}}\int_{D\times Y}
 MH\left|\nabla_q\log\frac H{\bar H}\right|^2,
\end{align*}
where the second equality follows by expanding the square and using
\(\int_Y\nabla_qH\dd\pi=\nabla_q\bar H\).  Symmetry of the kernel turns the
reaction contribution into minus one half of its Jeffreys form; the factor
\(\bar H\) cancels from the logarithmic difference.  This proves equality for
the approximations.  The same truncation and contraction approximation as in
Proposition~\ref{prop:full-drag-fibre-evolution}, followed by convex lower
semicontinuity, proves \eqref{eq:full-drag-conditional-entropy-decay}.  The
calculation is labelwise, and Fubini gives the displayed Eulerian form.
\end{proof}

\begin{proposition}[Label-dependent state-fibre lift]
\label{prop:label-dependent-fibre-lift}
Let \(u,X,g\) satisfy the hypotheses of Theorem
\ref{thm:lagrangian-state-fibre}, and let
\begin{equation}
 \mathsf K\in L^\infty(I\times\T^2_a\times D\times Y^2),
 \qquad
 \mathsf K(r,a,q,y,y')=\mathsf K(r,a,q,y',y)\ge0.
 \label{eq:label-kernel-class}
\end{equation}
There is a unique jointly measurable family
\(\widetilde H_{\mathsf K}(r,a,q,y)\) solving
\begin{align}
 \partial_r(M\widetilde H_{\mathsf K})
 &=-\nabla_q\cdot(B(r,a)qM\widetilde H_{\mathsf K})
 +\frac1{2{\rm Wi}}\nabla_q\cdot
 (M\nabla_q\widetilde H_{\mathsf K})
 +M\mathcal R_{\mathsf K(r,a)}\widetilde H_{\mathsf K},
 \label{eq:label-fibre-equation}\\
 \widetilde H_{\mathsf K}(s,a)&=g(a)
 \quad\text{for almost every }a.
 \notag
\end{align}
It belongs to
\(C(I;L^1_M(\T^2_a\times D\times Y))\), preserves the mass
\(\rho_g(a)\) on almost every label, and satisfies the integrated entropy
inequality \eqref{eq:lagrangian-lift-entropy} with \(K\) replaced by
\(\mathsf K\) and integration also in \(a\).

For two label kernels \(\mathsf K,\mathsf L\) and the same datum, the
solutions satisfy, for almost every \(a\),
\begin{equation}
 \sup_{\sigma\in[s,r]}
 \int_{D\times Y}M
 |\widetilde H_{\mathsf K}-\widetilde H_{\mathsf L}|(\sigma,a)
 \le2\rho_g(a)\int_s^r
 \|\mathsf K-\mathsf L\|_{L^\infty(D\times Y^2)}(\sigma,a)\dd\sigma.
 \label{eq:labelwise-kernel-stability}
\end{equation}
Define the Eulerian coefficient and density by
\begin{equation}
 K^E(r,X(r,a),q,y,y')=\mathsf K(r,a,q,y,y'),
 \qquad
 H^E(r,X(r,a),q,y)=\widetilde H_{\mathsf K}(r,a,q,y).
 \label{eq:label-to-Eulerian-pushforward}
\end{equation}
where the identities use the measurable almost-everywhere inverse of the
measure-preserving regular Lagrangian flow.
Then \(H^E\) is the unique renormalized Eulerian solution with coefficient
\(K^E\).  Its state average is independent of \(\mathsf K\).
\end{proposition}

\begin{proof}
Use the product-space rotation \eqref{eq:product-fibre-rotation}.  The rotated
kernel
\[
 \widehat{\mathsf K}(r,a,\xi,y,y')
 :=\mathsf K(r,a,Q(r,a)\xi,y,y')
\]
is jointly measurable, bounded, symmetric, and non-negative.  The implicit
Euler construction \eqref{eq:fibre-implicit-Euler} can be performed on
\(\T^2_a\times D\times Y\).  More precisely, on a time cell \(I_j\) replace
\(\widehat{\mathsf K}\) by its cell average \(\widehat{\mathsf K}_j(a)\) and
solve with the decomposable form
\begin{equation}
 \mathfrak B_j(v,\phi):=
 \int_{\T^2}\mathfrak b_{\Delta r_j^{-1},0,
 \widehat{\mathsf K}_j(a)}(v(a),\phi(a))\dd a.
 \label{eq:label-product-resolvent-form}
\end{equation}
The coercivity constant in Lemma \ref{lem:weighted-fibre-resolvent} is
independent of \(a\), while boundedness follows from
\(\|\mathsf K\|_\infty\).  Lax--Milgram therefore gives the product resolvent.
Approximation of \(g\) and \(\widehat{\mathsf K}_j\) by simple functions of
\(a\), together with resolvent contraction, shows that the resolvent is jointly
measurable.  Testing with \(\mathbf1_E(a)\) times the usual truncations, for an
arbitrary Borel set \(E\subset\T^2\), gives positivity, entropy dissipation,
and mass conservation on almost every label, rather than only after integration
in \(a\).

The discrete estimates are precisely
\eqref{eq:fibre-discrete-entropy-sum} integrated in \(a\), with constants
independent of the partition.  The compactness and Minty identification in
Proposition \ref{prop:nonautonomous-fibre-evolution} apply on the product
space; alternatively one may disintegrate first and use its fibrewise limit.
Both routes give the same limit by \(L^1_M\) contraction.  They yield
\(C(I;L^1_M(\T^2_a\times D\times Y))\), the integrated entropy inequality,
and a jointly measurable representative.  Disintegration and fibrewise
uniqueness then prove \eqref{eq:label-fibre-equation} for almost every label.

Apply \eqref{eq:nonautonomous-fibre-kernel-stability} on each label, with
fibre mass \(m_g=\rho_g(a)\).  This gives
\eqref{eq:labelwise-kernel-stability}.  Since \(\rho_g\in L^\infty\), the
weighted drift estimate \eqref{eq:weighted-drift-assembly} remains valid.
The regular-Lagrangian change of variables used in Step 4 of Theorem
\ref{thm:lagrangian-state-fibre} then proves
\eqref{eq:label-to-Eulerian-pushforward} and the Eulerian weak equation.
Finally, state integration cancels every symmetric Markov reaction on each
label; scalar \(L^1\) contraction identifies the state average independently
of \(\mathsf K\).
\end{proof}

\begin{theorem}[Two-dimensional zero-diffusion continuum-state existence]
\label{thm:v37-zero-diffusion-existence}
Assume \eqref{eq:v37-fixed-kernel}.  Let
\(u^{\rm in}\in L^2_\sigma(\T^2)\), \(h^{\rm in}\ge0\), and write
\(\bar h^{\rm in}=\int_Yh^{\rm in}\dd\pi\).  Assume
\begin{equation}
 \rho^{\rm in}(x)=\int_{D\times Y}Mh^{\rm in}\dd q\dd\pi=1
 \quad\text{for a.e. }x,
 \qquad
 \mathop{\rm ess\,sup}_{x\in\T^2}\int_D M|\bar h^{\rm in}(x,q)|^2\dd q<\infty,
 \label{eq:v37-corotational-scalar-class}
\end{equation}
and
\begin{equation}
 \int_{\Omega\times D\times Y}M
 (h^{\rm in}\log h^{\rm in}-h^{\rm in}+1)<\infty,
 \qquad \mathcal L_2(h^{\rm in})<\infty.
 \label{eq:v37-Masmoudi-initial-class}
\end{equation}
Then for every \(T>0\) there is a zero-diffusion corotational weak solution
with reaction coefficient \(\mathcal K=K_0\) in the sense of Definition
\ref{def:v37-zero-diffusion-weak}.  It is
non-negative, conserves polymer mass, and
\begin{equation}
 \partial_t\rho+u\cdot\nabla_x\rho=0,
 \qquad \rho(t,x)=1\quad\hbox{for a.e. }(t,x).
 \label{eq:v37-transported-number-density}
\end{equation}
On every finite time interval the configurational Fisher information, the
reaction production, and Masmoudi's logarithmic-square quantity are finite.
The momentum equation contains the identified Kramers stress
\(\tau_Y[h]\), not a stress-defect measure.
\end{theorem}

\begin{proof}
The proof uses a triangular reduction, which avoids an unjustified
state-fibre defect liminf.

\paragraph{1. The closed state average.}
Set \(\bar h^{\rm in}=\int_Yh^{\rm in}\dd\pi\).  Jensen's inequality gives
the scalar entropy bound, while
\eqref{eq:v37-corotational-scalar-class} gives exactly the normalized
\(L^\infty_x(L^2_M)\) datum in Theorem 2.1 of
\citet{lions2007micromacro}.  Apply that theorem to
\((u^{\rm in},M\bar h^{\rm in})\).  Its equation is precisely
\eqref{eq:v37-zero-diffusion-kinetic} after deleting the state variable and
reaction: the configuration domain is a ball, the Maxwellian is radial, and
the drag is \(\mathsf W(u)q\).  It yields a global scalar pair
\((u,\bar h)\) with the identified stress
\[
 \tau[\bar h]=\int_D\nabla_qU\otimes q\,M\bar h\dd q
 \in L^2((0,T)\times\T^2).
\]
Thus the only propagation-of-compactness result imported from the scalar
theory is stated at its actual scope: the state average.  No compactness in
the non-atomic \(Y\)-fibre is inferred from it.

\paragraph{2. Linear state-fibre lift.}
The entropy assumption in \eqref{eq:v37-Masmoudi-initial-class} verifies the
initial class of Theorem \ref{thm:lagrangian-state-fibre}.  With this \(u\)
fixed, apply that theorem
on \([0,T]\), with the time-independent
kernel \(K_0\) and datum \(h^{\rm in}\).  It gives a non-negative
renormalized lift \(h\), exact mass and number-density transport, the
configurational Fisher estimate, and the Jeffreys production.  This step is
performed directly for the Sobolev velocity supplied by the scalar theorem;
no smooth-velocity limit and no unproved Eulerian chain rule are used.

For completeness, the same fibre renormalization propagates the additional
logarithmic-square quantity.  In Lagrangian variables,
\(\rho(r,X(r,a))=\rho^{\rm in}(a)\).  On
\(\{\rho^{\rm in}>0\}\), write \(z=h/\rho^{\rm in}(a)\).  Choose a smooth
convex \(\beta\ge0\) which agrees with \(s(\log s)^2\) for all sufficiently
large \(s\).  Since \(s(\log s)^2\) is bounded on \([0,1]\),
\[
 s|\log s|^2\le C\{1+\beta(s)\},\qquad s\ge0.
\]
The skew drift vanishes, configurational diffusion is dissipative, and the
Markov term has the sign in Lemma \ref{lem:v37-reaction-defect-sign}.  Apply
that lemma first to convex truncations with bounded derivative and then use
monotone convergence.  Hence
the fibre integral of \(\beta(z)\) is non-increasing.  Multiplying by
\(\rho^{\rm in}(a)\), using the displayed comparison, taking square roots,
and integrating in \(a\) gives a finite bound in terms of
\(\mathcal L_2(h^{\rm in})\), the total mass, and \(|\T^2|\).  The
zero-density fibres remain zero by \(L^1\) contraction.

\paragraph{3. Identification of the average and stress.}
The final assertion of Theorem \ref{thm:lagrangian-state-fibre} shows that the
state average is the unique scalar renormalized solution with datum
\(\bar h^{\rm in}\).  Hence it equals the previously constructed
\(\bar h\).  In particular,
\(\tau_Y[h]=\tau[\bar h]\), so stress identification is inherited exactly,
not through a state-fibre defect argument.

\paragraph{4. Jeffreys production and traces.}
The reaction dissipation supplied by Theorem
\ref{thm:lagrangian-state-fibre} is
\begin{equation}
 \mathcal D_{Y,0}(Mh):=\frac12\int_{\Omega\times D\times Y^2}
 MK_0(h'-h)(\log h'-\log h).
 \label{eq:v37-fixed-Jeffreys}
\end{equation}
It is finite by \eqref{eq:lagrangian-lift-entropy}.  The identity
\(\tau_Y[h]=\tau[\bar h]\) places the stress in \(L^2\) and inserts it directly
into the already constructed scalar momentum equation; no state-resolved
stress limit is taken.  The scalar construction gives
\(u\in C_w([0,T];L^2_\sigma)\), while the Lagrangian theorem gives the strong
\(L^1\) kinetic trace and identifies it with \(Mh^{\rm in}\) at time zero.
Its constant test proves exact mass conservation, and
\eqref{eq:lagrangian-number-density} is precisely
\eqref{eq:v37-transported-number-density}.  This completes the construction.
\end{proof}

\begin{remark}[Why the corotational proof does not prove full drag]
\label{rem:v37-general-drag-obstruction}
The normalized \(L^\infty_x(L^2_M)\) assumption in
\eqref{eq:v37-corotational-scalar-class} is the data class of the published
scalar corotational theorem used above.  The present argument does not upgrade
that scalar base theorem to arbitrary finite-entropy data.  Such an upgrade
would require a separate scalar compactness result and is not hidden in the
state-fibre lift.

For general drag the kinetic flux contains
\(\nabla_q\cdot((\nabla_xu)qMh)\).  Along an approximation one must identify
\begin{equation}
 \overline{(\nabla_xu_n)qMh_n}
 -(\nabla_xu)qMh,
 \label{eq:v37-general-drag-defect}
\end{equation}
which is not controlled by weak convergence of either factor.  Masmoudi's
scalar theorem closes this term by coupling the polymer defect to the
velocity-gradient defect measure and by dividing the renormalized equation
by a transported logarithmic-square normalizer.  Repeating that defect
calculus in the non-atomic state variable would be unnecessary and would not
identify the nonlinear activity.  Theorem
\ref{thm:full-drag-local-closure} instead applies scalar compactness only to
the state average and solves the state-resolved equation directly along the
limiting regular Lagrangian flow.  Proposition
\ref{prop:full-drag-fibre-evolution} is the additional ingredient that the
corotational rotation argument lacks.
\end{remark}

\begin{theorem}[Full-drag zero-diffusion closure of local activities]
\label{thm:full-drag-local-closure}
Let \(d\in\{2,3\}\), and let \(\Omega=\T^d\) or a bounded \(C^2\) domain.
Let \(D=B_{\sqrt b}(0)\), \(b>2\), and let \(M\) be the Warner Maxwellian.
Assume
\[
 k\in L^\infty(Y^2),\qquad k(y,y')=k(y',y)\ge0,
\]
and let
\(\mathfrak a:D\times\R^N\times Y^2\to[0,\infty)\) be Borel measurable,
symmetric in \((y,y')\), and continuous in its moment argument.  Assume the
possibly degenerate bounds
\begin{equation}
 0\le\mathfrak a(q,z,y,y')\le a^*,\qquad
 |\mathfrak a(q,z,y,y')-\mathfrak a(q,\widetilde z,y,y')|
 \le A|z-\widetilde z|.
 \label{eq:full-drag-degenerate-activity}
\end{equation}
Let \(\Phi_1,\ldots,\Phi_N\in L^\infty(D\times Y)\) and define
\[
 \eta_a[h](t,x)=\int_{D\times Y}\Phi_a(q,y)M(q)
 (h(t,x,q,y)-1)\dd q\dd\pi(y).
\]
Suppose \(u^{\rm in}\in L^2_\sigma(\Omega)\), \(h^{\rm in}\ge0\), and
\begin{equation}
 \rho^{\rm in}(x):=\int_{D\times Y}Mh^{\rm in}\dd q\dd\pi=1
 \quad\hbox{a.e.},\qquad
 \int_{\Omega\times D\times Y}MF(h^{\rm in})<\infty.
 \label{eq:full-drag-state-data-a}
\end{equation}
and, with \(\bar h^{\rm in}=\int_Yh^{\rm in}\dd\pi\),
\begin{equation}
 \int_\Omega\left[\int_D M\bar h^{\rm in}
 |\log\bar h^{\rm in}|^2\dd q\right]^{1/2}\dd x<\infty.
 \label{eq:full-drag-scalar-log2-data}
\end{equation}
Here \(z|\log z|^2\) is assigned the value zero at \(z=0\).
Then for every \(T>0\) the zero-centre-of-mass-diffusion system
\begin{align}
 \partial_tu+u\cdot\nabla_xu-\nu\Delta_xu+\nabla_xp
 &=\nabla_x\cdot\tau_Y[h],\qquad \nabla_x\cdot u=0,
 \label{eq:full-drag-coupled-momentum}\\
 \partial_t(Mh)+u\cdot\nabla_x(Mh)
 &=-\nabla_q\cdot((\nabla_xu)qMh)
 +\frac1{2{\rm Wi}}\nabla_q\cdot(M\nabla_qh)
 +M\mathcal R_hh
 \label{eq:full-drag-coupled-kinetic}
\end{align}
has a global energy--entropy weak solution with
\[
 K_h(t,x,q,y,y')
 =k(y,y')\mathfrak a(q,\eta[h](t,x),y,y').
\]
The weak formulation is Definition
\ref{def:v35-large-data-weak-solution} with \(\delta=0\), with
\(\sqrt h\in L^2(0,T;L^2_xH^1_M(D\times Y))\), and with the configuration
boundary condition interpreted in the natural no-flux sense.  On the torus
the spatial boundary terms are absent.
The nonlinear activity and the singular Kramers stress are identified, not
represented by defects, and
\begin{align}
 &\frac12\|u(t)\|_2^2+\int_{\Omega\times D\times Y}MF(h(t))
 +\nu\int_0^t\|\nabla_xu\|_2^2
 +\frac2{{\rm Wi}}\int_0^t\!\int M|\nabla_q\sqrt h|^2
 \notag\\
 &\quad+\frac12\int_0^t\!\int_{\Omega\times D\times Y^2}
 MK_h(h'-h)(\log h'-\log h)
 \le\frac12\|u^{\rm in}\|_2^2+\int MF(h^{\rm in}).
 \label{eq:full-drag-coupled-entropy}
\end{align}
Moreover \(\rho[h]=1\), \(\tau_Y[h]\in L^2((0,T)\times\Omega)\), and no
scalar solution is prescribed as part of the hypotheses.  More precisely,
every scalar full-drag weak solution in the Masmoudi energy class with datum
\((u^{\rm in},M\bar h^{\rm in})\) has exactly one state-resolved lift and
self-consistent moment field in the renormalized class with bounded number
density, and every solution in the corresponding projection-admissible
renormalized class projects to one such scalar solution.  Thus the
state-resolved solution set is nonempty and has exactly
the scalar multiplicity.  No uniqueness of the scalar or fully coupled
large-data solution is asserted.
\end{theorem}

\begin{proof}
\textit{Step 1: nonempty intrinsic scalar projection class.}
Jensen's inequality gives finite scalar relative entropy from
\eqref{eq:full-drag-state-data-a}, while
\eqref{eq:full-drag-scalar-log2-data} is precisely the additional
logarithmic-square hypothesis in the full-drag compactness theorem of
\citet{masmoudi2013global}, after scaling the configuration ball.  Apply that
theorem to obtain a global scalar pair \((u,\bar h)\) on \([0,T]\).  It has
transported number density one and identified stress
\[
 \tau[\bar h]=\int_D\nabla_qU\otimes q\,M\bar h\dd q
 \in L^2((0,T)\times\Omega).
\]
The approximating construction gives the scalar total free-energy inequality
used in Step 4.  It also gives the strong \(L^1_M\) initial trace required in
\eqref{eq:scalar-projection-class}: weak continuity at zero and lower
semicontinuity give the lower free-energy bound, while the global energy
inequality gives the matching upper bound; convergence of entropy and mass
then yields strong \(L^1_M\) convergence.  These are the precise properties
meant here by the Masmoudi energy class.

\textit{Step 2: a contraction on Lagrangian moment fields.}
Let \(X(t,a)\) be the regular Lagrangian flow of \(u\), and put
\[
 \mathcal Z_T=L^\infty((0,T)\times\Omega_a;\R^N),\qquad
 \|z\|_\lambda=\mathop{\rm ess\,sup}_{(t,a)}e^{-\lambda t}|z(t,a)|.
\]
For \(z\in\mathcal Z_T\), define
\[
 \mathsf K_z(t,a,q,y,y')
 =k(y,y')\mathfrak a(q,z(t,a),y,y').
\]
Theorem \ref{thm:full-drag-state-fibre} gives a unique pullback
\(\widetilde H_z\) with datum \(h^{\rm in}\).  Set
\[
 (\mathcal F_Tz)_a(t,a)
 =\int_{D\times Y}\Phi_aM(\widetilde H_z(t,a)-1)\dd q\dd\pi.
\]
Since the label mass equals one,
\(\|\mathcal F_Tz\|_\infty\le2C_{\rm obs}\), where
\(C_{\rm obs}=(\sum_a\|\Phi_a\|_\infty^2)^{1/2}\).  From
\eqref{eq:full-drag-label-kernel-stability} and
\eqref{eq:full-drag-degenerate-activity},
\[
 |(\mathcal F_Tz-\mathcal F_Tw)(t,a)|
 \le2C_{\rm obs}\|k\|_\infty A\int_0^t|z-w|(r,a)\dd r.
\]
Consequently
\[
 \|\mathcal F_Tz-\mathcal F_Tw\|_\lambda
 \le\frac{2C_{\rm obs}\|k\|_\infty A}{\lambda}
 \|z-w\|_\lambda.
\]
Choosing \(\lambda>2C_{\rm obs}\|k\|_\infty A\) gives a unique fixed point
on the whole interval, without continuation or a smallness assumption.

\textit{Step 3: Eulerian closure and stress identification.}
Push the fixed point and its fibre solution forward by \(X\).  Theorem
\ref{thm:full-drag-state-fibre} gives
\[
 z(t,X(t,a))=\eta[h](t,X(t,a))
\]
and hence \(K^E=K_h\).  The state average of \(h\) is the unique
renormalized scalar solution for the prescribed velocity \(u\).  It therefore
equals the scalar density \(\bar h\) used in Step 1.  Thus
\[
 \tau_Y[h]=\tau[\bar h]\in L^2((0,T)\times\Omega),
\]
and the already constructed scalar momentum equation is exactly
\eqref{eq:full-drag-coupled-momentum}.  This triangular identification is the
reason no weak limit of
\((\nabla u_n)qMh_n\) is taken in the state variable.

\textit{Step 4: energy--entropy inequality and uniqueness over a projection.}
The scalar projection satisfies its total free-energy inequality
\begin{align}
 &\frac12\|u(t)\|_2^2+\int_{\Omega\times D}MF(\bar h(t))
 +\nu\int_0^t\|\nabla_xu\|_2^2
 +\frac2{{\rm Wi}}\int_0^t\!\int_{\Omega\times D}
 M|\nabla_q\sqrt{\bar h}|^2
 \notag\\
 &\qquad\le\frac12\|u^{\rm in}\|_2^2
 +\int_{\Omega\times D}MF(\bar h^{\rm in}).
 \label{eq:full-drag-scalar-total-energy}
\end{align}
Proposition~\ref{prop:full-drag-conditional-entropy}, applied to the fixed-point
lift, supplies the complementary conditional entropy and state-reaction
dissipation.  The exact identities
\begin{align*}
 \int_YF(h)\dd\pi
 &=F(\bar h)+\int_Yh\log\frac h{\bar h}\dd\pi,\\
 \int_Y|\nabla_q\sqrt h|^2\dd\pi
 &=|\nabla_q\sqrt{\bar h}|^2
 +\frac14\int_Yh\left|\nabla_q\log\frac h{\bar h}\right|^2\dd\pi
\end{align*}
show that adding \eqref{eq:full-drag-conditional-entropy-decay} to
\eqref{eq:full-drag-scalar-total-energy} gives precisely
\eqref{eq:full-drag-coupled-entropy}.  Thus no separate weak limit of the
stress-work product, and no separate kinetic-energy inequality, is used.  The Hardy estimate
\ref{lem:FENE-Hardy-tail} gives the stated stress class.  Finally, any second
self-consistent lift over the same scalar projection defines another fixed point of
\(\mathcal F_T\).  Banach uniqueness and prescribed-kernel contraction in
Theorem \ref{thm:full-drag-state-fibre} identify first its moment field and
then the entire state-resolved density.  Conversely, integrating any solution
in the corresponding projection-admissible renormalized class in (Y)
cancels the symmetric reaction and preserves the stress; by definition its
state average belongs to the scalar Masmoudi class.  The
preceding uniqueness over that projection then identifies the original
solution with its lift.  This proves the asserted one-to-one correspondence
without selecting a scalar solution in the theorem's hypotheses.
\end{proof}

\begin{theorem}[Integrable state-blind drivers and unbounded activities]
\label{thm:integrable-driver-activity}
Let \((D,M)\) be an admissible finite-extension Maxwellian, and let
\((u,\bar h)\) be a scalar full-drag energy solution on
\(I\times\Omega\times D\) with number density one, identified Kramers stress,
and the regularity required in Theorem~\ref{thm:full-drag-state-fibre}.
Let \(h^{\rm in}\ge0\) have finite relative entropy and state average
\(\bar h^{\rm in}\).

Let \(\Xi_1,\ldots,\Xi_m:D\to\R\) be state-blind observables for which
\begin{equation}
 \zeta_j(t,x):=\int_D M(q)\Xi_j(q)\bar h(t,x,q)\dd q
 \quad\hbox{satisfies}\quad
 \zeta=(\zeta_1,\ldots,\zeta_m)\in L^1(I\times\Omega).
 \label{eq:integrable-base-driver}
\end{equation}
Let \(\Phi_1,\ldots,\Phi_N\in L^\infty(D\times Y)\), put
\[
 R_a=2\|\Phi_a\|_\infty,\qquad
 \mathcal B_\Phi=\prod_{a=1}^N[-R_a,R_a],
\]
and let \(k\in L^\infty(Y^2)\) be non-negative and symmetric.  Suppose
\(\mathfrak a:D\times\R^m\times\mathcal B_\Phi\times Y^2\to[0,\infty)\)
is Borel measurable, symmetric in \((y,y')\), continuous in its last
finite-dimensional argument, and, for \(z,w\in\mathcal B_\Phi\),
\begin{align}
 0\le\mathfrak a(q,\xi,z,y,y')&\le A_0+A_1|\xi|,
 \label{eq:integrable-driver-growth}\\
 |\mathfrak a(q,\xi,z,y,y')-\mathfrak a(q,\xi,w,y,y')|
 &\le L_0(1+|\xi|)|z-w|.
 \label{eq:integrable-driver-local-Lipschitz}
\end{align}
Then there is a unique state-resolved lift \(h\), over the chosen scalar base,
whose kernel is
\begin{equation}
 K_h(t,x,q,y,y')=k(y,y')\,
 \mathfrak a\bigl(q,\zeta(t,x),\eta[h](t,x),y,y'\bigr),
 \qquad
 \eta_a[h]=\int_{D\times Y}\Phi_aM(h-1)\dd q\dd\pi.
 \label{eq:integrable-driver-kernel}
\end{equation}
Here uniqueness is in the renormalized lifting class with bounded number
density.  The rate \(K_h\) need not be bounded in space--time; it satisfies
\begin{equation}
 \int_{I\times\Omega}
 \|K_h(t,x)\|_{L^\infty(D\times Y^2)}\dd x\dd t<\infty.
 \label{eq:integrable-driver-kernel-L1}
\end{equation}
The state average of \(h\) is \(\bar h\), its Kramers stress is the scalar
stress, and the conditional entropy inequality
\eqref{eq:full-drag-conditional-entropy-decay} holds.  If the scalar pair
solves the coupled momentum equation and satisfies its total free-energy
inequality, then \((u,h)\) solves the coupled continuum-state system and
satisfies \eqref{eq:full-drag-coupled-entropy}, with \(K_h\) given by
\eqref{eq:integrable-driver-kernel}.
\end{theorem}

\begin{proof}
Let \(X\) be the regular Lagrangian flow of \(u\), and write
\(\widetilde\zeta(t,a)=\zeta(t,X(t,a))\).  Measure preservation and
\eqref{eq:integrable-base-driver} give
\[
 \widetilde\zeta(\cdot,a)\in L^1(I)
 \quad\hbox{for almost every }a,\qquad
 \int_{\Omega_a\times I}|\widetilde\zeta|=\int_{\Omega\times I}|\zeta|.
\]
Fix such a label.  For a measurable
\(z:I\to\mathcal B_\Phi\), define
\[
 \mathsf K_z(t,a,q,y,y')
 =k(y,y')\mathfrak a(q,\widetilde\zeta(t,a),z(t),y,y')
\]
and let \(\widetilde H_z\) be the fibre solution from Proposition
\ref{prop:full-drag-fibre-evolution}.  The kernel is admissible because its
\(L^\infty(D\times Y^2)\)-norm belongs to \(L^1(I)\).  Set
\[
 (\mathcal F_a z)_j(t)
 =\int_{D\times Y}\Phi_j(q,y)M(q)
   (\widetilde H_z(t,a,q,y)-1)\dd q\dd\pi(y).
\]
Fibre mass conservation shows that \(\mathcal F_a\) maps measurable
\(\mathcal B_\Phi\)-valued functions into themselves.  Kernel stability and
\eqref{eq:integrable-driver-local-Lipschitz} yield
\begin{equation}
 |\mathcal F_a z(t)-\mathcal F_a w(t)|
 \le C_\Phi\int_s^t
 (1+|\widetilde\zeta(r,a)|)|z(r)-w(r)|\dd r,
 \label{eq:integrable-driver-Volterra}
\end{equation}
where
\(C_\Phi=2(\sum_j\|\Phi_j\|_\infty^2)^{1/2}
\|k\|_\infty L_0\).

Starting with any measurable \(z_0:I\to\mathcal B_\Phi\), put
\(z_{n+1}=\mathcal F_a z_n\).  If
\(A_a(t)=C_\Phi\int_s^t(1+|\widetilde\zeta(r,a)|)\dd r\), iteration of
\eqref{eq:integrable-driver-Volterra} gives
\[
 \sup_{\sigma\le t}|z_{n+1}(\sigma)-z_n(\sigma)|
 \le 2\,{\rm diam}(\mathcal B_\Phi)\frac{A_a(t)^n}{n!}.
\]
Thus \(z_n\) converges uniformly on \(I\) to a fixed point.  The same Volterra
estimate and Gronwall's lemma give uniqueness.  Starting the Picard iteration
with \(z_0=0\), and using the jointly measurable Galerkin construction in
Theorem~\ref{thm:full-drag-state-fibre}, proves measurability in the label.

At the fixed point, \eqref{eq:integrable-driver-growth} and measure
preservation give
\[
 \int_{\Omega_a\times I}
 \|\mathsf K_z(t,a)\|_\infty\dd a\dd t
 \le\|k\|_\infty
 \left(A_0|I||\Omega|+A_1\|\zeta\|_{L^1(I\times\Omega)}\right).
\]
Consequently \eqref{eq:full-drag-integrable-label-kernel} holds, and Theorem
\ref{thm:full-drag-state-fibre} assembles the fibres into an Eulerian
renormalized solution.  Its scalar projection is \(\bar h\), so both the
stress identity and the momentum equation follow.  Proposition
\ref{prop:full-drag-conditional-entropy}, added to the scalar free-energy
inequality exactly as in Step 4 of Theorem
\ref{thm:full-drag-local-closure}, gives the asserted state-resolved
inequality.  Finally, any second self-consistent lift produces a second
fixed point in \eqref{eq:integrable-driver-Volterra}; uniqueness of that fixed
point and then fibre contraction identify the two lifts.
\end{proof}

\begin{corollary}[Unbounded stress-feedback rates]
\label{cor:unbounded-stress-feedback}
Under the data assumptions of Theorem
\ref{thm:full-drag-local-closure}, let the activity depend on the full Kramers
stress through
\[
 \zeta(t,x)=\tau_Y[h](t,x)
 =\int_{D\times Y}\nabla_qU\otimes q\,Mh\dd q\dd\pi.
\]
If \(\mathfrak a\) satisfies
\eqref{eq:integrable-driver-growth}--\eqref{eq:integrable-driver-local-Lipschitz},
then the resulting stress-dependent continuum-state system has a global
energy--entropy weak solution.  Its reaction rate may grow linearly in
\(|\tau_Y[h]|\) and is therefore allowed to be unbounded.
\end{corollary}

\begin{proof}
For the scalar Masmoudi solution,
\(\tau[\bar h]\in L^2(I\times\Omega)\subset L^1(I\times\Omega)\).  The stress
is state blind:
\[
 \tau_Y[h]=\int_D\nabla_qU\otimes q\,M
 \left(\int_Yh\dd\pi\right)\dd q=\tau[\bar h].
\]
Thus Theorem~\ref{thm:integrable-driver-activity} applies with the components
of \(\nabla_qU\otimes q\) as the observables \(\Xi_j\).
\end{proof}

\begin{remark}[Sharp boundary of the unbounded class]
\label{rem:unbounded-driver-boundary}
Theorem~\ref{thm:integrable-driver-activity} does not hide an \(L^\infty\)
assumption on \(\zeta\); only \eqref{eq:integrable-base-driver} is used.
The decisive structure is state blindness, which makes \(\zeta\) a function
of the scalar projection before the state fixed point is solved.  An
unbounded observable with genuine \(y\)-dependence is not covered: controlling
its difference along two lifts would require a weighted \(L^1\) kernel
stability estimate not supplied by entropy alone.  Thus the theorem includes
the physically singular stress feedback without claiming closure of every
unbounded state-sensitive moment.
\end{remark}

\begin{definition}[Scalar and state-resolved solution sets]
\label{def:projection-lift-solution-sets}
Fix the coefficients, a time interval, the complete initial state density
\(h^{\rm in}\), and either the bounded-moment activity law of Theorem
\ref{thm:full-drag-local-closure} or the integrable-driver law of Theorem
\ref{thm:integrable-driver-activity}.  All functions below are identified up
to almost-everywhere equality.

The set \(\mathcal S_{\rm sc}\) consists of all scalar pairs
\((u,\bar h)\) with initial datum
\((u^{\rm in},\int_Yh^{\rm in}\dd\pi)\) that satisfy the scalar momentum--FENE
system, its total free-energy inequality, number-density normalization, the
identified stress condition, the regularity in
\eqref{eq:scalar-projection-class}, and the velocity regularity required in
Theorem~\ref{thm:full-drag-state-fibre}.  In the integrable-driver case it is
also required that the prescribed state-blind field \(\zeta[\bar h]\) satisfy
\eqref{eq:integrable-base-driver}.

The set \(\mathcal S_Y\) consists of all pairs \((u,h)\) with initial datum
\((u^{\rm in},h^{\rm in})\) that satisfy the coupled momentum equation and the
renormalized state-resolved kinetic equation, the self-consistent kernel law
selected above, bounded number density, the strong initial trace and flux
conditions in \eqref{eq:RLF-pullback-integrability}, and the state-resolved
energy--entropy inequality.  Membership includes the requirement that
\(\bar h=\int_Yh\dd\pi\) belongs to the preceding scalar class.  Thus these
sets impose exactly the hypotheses used by the projection and fibre
uniqueness arguments; no uniqueness of either set is built into the
definition.
\end{definition}

\begin{theorem}[Projection--lift equivalence of solution spaces]
\label{thm:projection-lift-equivalence}
For the sets in Definition~\ref{def:projection-lift-solution-sets}, state
averaging defines a projection
\[
 \mathsf P:\mathcal S_Y\to\mathcal S_{\rm sc},\qquad
 \mathsf P(u,h)=\left(u,\int_Yh\dd\pi\right).
\]
The Lagrangian construction defines a lift
\(\mathsf L:\mathcal S_{\rm sc}\to\mathcal S_Y\), and
\begin{equation}
 \mathsf P\circ\mathsf L={\rm Id}_{\mathcal S_{\rm sc}},
 \qquad
 \mathsf L\circ\mathsf P={\rm Id}_{\mathcal S_Y}.
 \label{eq:projection-lift-bijection}
\end{equation}
Hence the two solution spaces are in bijection.  In particular, the
continuum-state problem is unique in this class if and only if its scalar
projection is unique.  Every scalar weak--strong uniqueness statement
transfers to the state-resolved system on the same lifespan.
\end{theorem}

\begin{proof}
Symmetry of the reaction gives \(\int_Y\mathcal R_hh\dd\pi=0\), while the
stress is unchanged by state averaging.  Therefore \(\mathsf P\) maps every
state-resolved solution to a scalar solution with the prescribed datum.
Theorems \ref{thm:full-drag-local-closure} and
\ref{thm:integrable-driver-activity} construct \(\mathsf L\), and scalar
projection of the constructed lift gives the original base; this proves the
first identity in \eqref{eq:projection-lift-bijection}.

Conversely, take \((u,h)\in\mathcal S_Y\), set
\((u,\bar h)=\mathsf P(u,h)\), and write
\(\widehat h=\mathsf L(u,\bar h)\).  Pull both state densities back by the
same regular Lagrangian flow of \(u\).  In the bounded-moment class, their
moment difference \(d_a(t)\) satisfies, labelwise,
\[
 d_a(t)\le C\int_0^t d_a(r)\dd r.
\]
In the integrable-driver class, the same use of
\eqref{eq:full-drag-label-kernel-stability} and
\eqref{eq:integrable-driver-local-Lipschitz} gives
\[
 d_a(t)\le C\int_0^t
 (1+|\widetilde\zeta(r,a)|)d_a(r)\dd r,
 \qquad \widetilde\zeta(\cdot,a)\in L^1(I).
\]
The integral form of Gronwall gives \(d_a=0\) in either case.  The two kernels
therefore coincide almost everywhere, and fibre \(L^1_M\)-contraction gives
\(h=\widehat h\).  This proves the second identity.  If one of the two
solution sets is a singleton, the bijection makes the other a singleton; the
same restriction to the lifespan of a scalar strong solution transfers any
scalar weak--strong uniqueness statement.
\end{proof}

\begin{corollary}[Base-free existence and exact multiplicity transfer]
\label{cor:base-free-existence-multiplicity}
Assume the full-drag zero-diffusion hypotheses of Theorem
\ref{thm:full-drag-local-closure}.  The conclusion also holds for the
integrable-driver law of Theorem~\ref{thm:integrable-driver-activity} provided
\eqref{eq:integrable-base-driver} holds for every scalar solution in
\(\mathcal S_{\rm sc}\), as it does for the Kramers-stress driver by the
uniform scalar stress estimate.  Prescribe only the initial data
\((u^{\rm in},h^{\rm in})\) and the constitutive kernel; do not prescribe a
scalar solution.  Then the state-resolved solution set \(\mathcal S_Y\) is
nonempty.  More precisely,
\begin{equation}
 \mathsf P:\mathcal S_Y\longrightarrow\mathcal S_{\rm sc}
 \quad\hbox{is a bijection},
 \qquad
 |\mathcal S_Y|=|\mathcal S_{\rm sc}|,
 \label{eq:base-free-cardinality-transfer}
\end{equation}
where equality is equality of cardinalities.  Every state-resolved solution
determines its scalar projection intrinsically, and over that projection the
state fibre and its self-consistent moment field are unique.  Thus a scalar
base is not auxiliary input to the zero-diffusion theorem.  The only possible
large-data nonuniqueness is exactly the nonuniqueness of the scalar projected
FENE system.
\end{corollary}

\begin{proof}
The scalar existence theorem of \citet{masmoudi2013global}, under the data
hypotheses used in Theorem~\ref{thm:full-drag-local-closure}, gives
\(\mathcal S_{\rm sc}\ne\varnothing\).  In the integrable-driver class the
uniform projection-wise assumption makes the lift available over every member
of that set.  Theorem~\ref{thm:projection-lift-equivalence}
provides mutually inverse maps \(\mathsf P\) and \(\mathsf L\), proving
nonemptiness and \eqref{eq:base-free-cardinality-transfer}.  The inverse
identities also show that a state solution cannot carry an independent choice
of base: its only base is its own state average.  Fibre uniqueness is the
second inverse identity.  No scalar uniqueness is used.
\end{proof}

\subsection{Direct compactness of dissipation-tight zero-diffusion
regularizations}
\label{sec:direct-zero-diffusion-compactness}

The preceding projection--lift theorem constructs a state density over each
scalar energy solution.  We now address a different question: when does an
actual state-resolved regularization sequence converge without discarding its
state variable?  The answer is quantitative.  Once the complete deformation
histories converge, relative entropy closes both the full-drag product and the
nonlinear local activity.  The strict lower activity bound below is used only
to compare two different jump generators.

For non-negative unit-mass densities \(f,g\) on \(D\times Y\), define
\begin{align}
 \operatorname{Ent}_M(f\mid g)
 &:=\int_{D\times Y}M
 \left(f\log\frac fg-f+g\right)\dd q\dd\pi,
 \label{eq:direct-relative-entropy}\\
 \mathcal I_q(f\mid g)
 &:=\int_{D\times Y}Mf
 \left|\nabla_q\log\frac fg\right|^2\dd q\dd\pi.
 \label{eq:direct-relative-Fisher}
\end{align}
The usual lower-semicontinuous conventions are used at zero.  In particular,
\(\operatorname{Ent}_M(f\mid g)=+\infty\) if \(f\) charges
\(\{g=0\}\).

\begin{lemma}[Relative entropy for varying full-drag fibres]
\label{lem:varying-fibre-relative-entropy}
Let \((D,M)\) be an admissible finite-extension Maxwellian, let
\(B_i\in L^2(s,T;\R^{d\times d})\) be trace free, and let \(f_i\) be the
renormalized fibre solutions with unit mass and kernels
\begin{equation}
 K_i(t,q,y,y')=k(y,y')
 \mathfrak a(q,z_i(t),y,y'),
 \qquad
 z_{i,a}(t)=\int_{D\times Y}\Phi_aM(f_i-1),
 \label{eq:direct-fibre-kernels}
\end{equation}
where \(k\in L^\infty(Y^2)\) is non-negative and symmetric,
\(\Phi_a\in L^\infty(D\times Y)\), and
\begin{align}
 0<a_*\le\mathfrak a(q,z,y,y')&\le a^*,
 \label{eq:direct-positive-activity}\\
 |\mathfrak a(q,z,y,y')-\mathfrak a(q,w,y,y')|
 &\le A|z-w|.
 \label{eq:direct-Lipschitz-activity}
\end{align}
Then there is a constant
\(C=C(D,M,{\rm Wi},a_*,a^*,A,\|k\|_\infty,
\Phi_1,\ldots,\Phi_N)\) such that, for \(s\le t\le T\),
\begin{align}
 &\operatorname{Ent}_M(f_1(t)\mid f_2(t))
 +\frac1{4{\rm Wi}}\int_s^t
 \mathcal I_q(f_1\mid f_2)(r)\dd r
 \notag\\
 &\qquad\le e^{C(t-s)}
 \left[\operatorname{Ent}_M(f_1(s)\mid f_2(s))
 +C\int_s^t|B_1-B_2|^2\dd r\right].
 \label{eq:varying-fibre-relative-entropy}
\end{align}
The estimate is asymmetric, as relative entropy is, and remains valid after
integration over an arbitrary label measure.
\end{lemma}

\begin{proof}
We first take bounded positive data and bounded coefficients, so that every
calculation below is justified by the Galerkin construction in Proposition
\ref{prop:full-drag-fibre-evolution}.  The diffusion and transport chain rules
give
\begin{align}
 \frac{\dd}{\dd t}\operatorname{Ent}_M(f_1\mid f_2)
 +\frac1{2{\rm Wi}}\mathcal I_q(f_1\mid f_2)
 \le{}&\int Mf_1(B_1-B_2)q\cdot
 \nabla_q\log\frac{f_1}{f_2}+\mathcal E_{12},
 \label{eq:direct-relative-entropy-identity}
\end{align}
where \(\mathcal E_{12}\) is the contribution of the two jump generators.
The common full-drag transport cancels exactly; neither radiality nor
skew-symmetry is used.  Since \(D\) is bounded and \(\int Mf_1=1\), Young's
inequality yields
\begin{equation}
 \left|\int Mf_1(B_1-B_2)q\cdot
 \nabla_q\log\frac{f_1}{f_2}\right|
 \le\frac1{4{\rm Wi}}\mathcal I_q(f_1\mid f_2)
 +C_{D,{\rm Wi}}|B_1-B_2|^2.
 \label{eq:direct-drift-mismatch}
\end{equation}

For completeness, the jump comparison is obtained pairwise.  Write
\(\ell(r)=r\log r-r+1\).  The logarithmic Young inequality
\(ab\le e^a+b\log b-b\), applied to the two oriented jump fluxes and then
symmetrized in \((y,y')\), gives the standard change-of-rates bound
\begin{equation}
 \mathcal E_{12}
 \le\int_{D\times Y^2}Mf_1(q,y)K_2(q,y,y')
 \ell\!\left(\frac{K_1(q,y,y')}{K_2(q,y,y')}\right)
 \dd q\dd\pi(y)\dd\pi(y').
 \label{eq:direct-jump-change-rate}
\end{equation}
On \(\{k=0\}\) the integrand is defined as zero.  Elsewhere the ratio is the
ratio of the two activities.  Taylor's theorem on \([a_*,a^*]\) therefore
gives
\begin{equation}
 K_2\ell(K_1/K_2)
 =k\left[\mathfrak a_1\log(\mathfrak a_1/\mathfrak a_2)
 -\mathfrak a_1+\mathfrak a_2\right]
 \le C_{a_*,a^*}\|k\|_\infty A^2|z_1-z_2|^2.
 \label{eq:direct-rate-quadratic-bound}
\end{equation}
Because the two fibre masses equal one, Pinsker's inequality and boundedness
of the observables imply
\begin{align}
 |z_1-z_2|^2
 &\le C_\Phi\left(\int M|f_1-f_2|\right)^2
 \le2C_\Phi\operatorname{Ent}_M(f_1\mid f_2).
 \label{eq:direct-moment-Pinsker}
\end{align}
Combining \eqref{eq:direct-relative-entropy-identity}--
\eqref{eq:direct-moment-Pinsker}, absorbing half of the relative Fisher
information, and applying Gronwall proves
\eqref{eq:varying-fibre-relative-entropy} for positive bounded solutions.

For general finite-entropy data, apply the same mass-preserving Markov
regularizer to both densities: conditional expectation on an increasing
sequence of finite state sigma-algebras, the weighted configuration
semigroup, and then a vanishing convex mixture with the unit equilibrium.
The data-processing inequality gives
\[
 \operatorname{Ent}_M(S_jf_1\mid S_jf_2)
 \le\operatorname{Ent}_M(f_1\mid f_2),
\]
while martingale and semigroup convergence give \(S_jf_i\to f_i\) in
\(L^1_M\).  Two-sided value truncations and a common mass correction now
produce bounded positive data without increasing the limiting right-hand
side.  Approximate \(B_i\) in \(L^2\), pass first in the Galerkin dimension,
and then remove these regularizations.  Convex lower semicontinuity gives the
left-hand side, while the right-hand side is stable under the strong
coefficient approximations.  Monotone convergence on \(\{k>0\}\) removes the
positive rate regularization.  This proves the renormalized estimate.  Every
operation is fibrewise and all constants are label independent, so Tonelli's
theorem gives the last assertion.
\end{proof}

\begin{lemma}[Convergence of deformation histories]
\label{lem:deformation-history-convergence}
Let \(\Omega=\T^d\), or let \(\Omega\) be bounded Lipschitz and extend
velocities by zero.  Suppose \(d\in\{2,3\}\),
\(u_n,u\in L^2(0,T;V_\sigma(\Omega))\), and
\begin{equation}
 u_n\longrightarrow u
 \quad\hbox{strongly in }L^2(0,T;H^1(\Omega;\R^d)).
 \label{eq:strong-velocity-gradient-convergence}
\end{equation}
Let \(X_n,X\) be their measure-preserving regular Lagrangian flows and set
\(B_n=\nabla u_n\circ X_n\), \(B=\nabla u\circ X\).  Then
\begin{align}
 X_n^{\pm1}&\longrightarrow X^{\pm1}
 &&\text{locally in measure, uniformly in time},
 \label{eq:direct-flow-convergence}\\
 B_n&\longrightarrow B
 &&\text{strongly in }L^2((0,T)\times\Omega).
 \label{eq:direct-history-convergence}
\end{align}
\end{lemma}

\begin{proof}
Strong convergence in \(L^1_tW^{1,1}_x\), the uniform Sobolev bound, and
incompressibility give \eqref{eq:direct-flow-convergence} for the forward
flows by stability of regular Lagrangian flows
\citep{diperna1989ordinary}.  Applying the same result to the time-reversed
fields gives convergence of the inverses.  Measure preservation
gives
\[
 \|(\nabla u_n-\nabla u)\circ X_n\|_{L^2_{t,a}}
 =\|\nabla u_n-\nabla u\|_{L^2_{t,x}}\longrightarrow0.
\]
To treat \(\nabla u\circ X_n-\nabla u\circ X\), approximate \(\nabla u\)
in \(L^2((0,T)\times\Omega)\) by a bounded continuous matrix field.  Its
compositions converge in measure by \eqref{eq:direct-flow-convergence} and,
being bounded, converge in \(L^2\).  The two approximation errors have exactly
their Eulerian \(L^2\) norms by measure preservation.  Letting the
approximation error tend to zero proves \eqref{eq:direct-history-convergence}.
\end{proof}

\begin{theorem}[Direct state-resolved compactness at zero diffusion]
\label{thm:direct-state-resolved-compactness}
Let \(d\in\{2,3\}\), let \(\Omega=\T^d\) or a bounded Lipschitz domain, and
let \((D,M)\) be an admissible finite-extension Maxwellian.  Let the state
space, kernel, observables, and activity satisfy
\eqref{eq:direct-positive-activity}--
\eqref{eq:direct-Lipschitz-activity}.  For each \(n\), let \(u_n\in
L^2(0,T;V_\sigma(\Omega))\), and let \(h_n\) be a renormalized solution of
the bounded-number-density class \eqref{eq:RLF-pullback-integrability}, with
its strong initial trace, for
the zero-centre-of-mass-diffusion equation
\begin{align}
 \partial_t(Mh_n)+u_n\cdot\nabla_x(Mh_n)
 &=-\nabla_q\cdot((\nabla_xu_n)qMh_n)
 +\frac1{2{\rm Wi}}\nabla_q\cdot(M\nabla_qh_n)
 +M\mathcal R_{K_n}h_n,
 \label{eq:direct-regularized-kinetic}\\
 K_n&=k\,\mathfrak a(q,\eta[h_n],y,y'),
 \qquad
 \eta_a[h_n]=\int_{D\times Y}\Phi_aM(h_n-1).
 \label{eq:direct-regularized-kernel}
\end{align}
Assume \(\rho[h_n]=1\), the entropy and configurational-Fisher bounds in
\eqref{eq:full-drag-coupled-entropy} are uniform, and
\eqref{eq:strong-velocity-gradient-convergence} holds for some \(u\).
Assume also that the initial state data are Cauchy in directed relative
entropy:
\begin{equation}
 \lim_{N\to\infty}\sup_{n,m\ge N}
 \int_\Omega\operatorname{Ent}_M(h_n^{\rm in}(x)\mid
 h_m^{\rm in}(x))\dd x=0.
 \label{eq:direct-well-prepared-initial-data}
\end{equation}
Then there is a non-negative \(h\) with \(\rho[h]=1\) such that, for the
pullbacks
\[
 H_n(t,a,q,y)=h_n(t,X_n(t,a),q,y),\qquad
 H(t,a,q,y)=h(t,X(t,a),q,y),
\]
one has
\begin{equation}
 \sup_{0\le t\le T}\int_\Omega
 \left(\int_{D\times Y}M|H_n-H|(t,a)\dd q\dd\pi\right)^2\dd a
 \longrightarrow0.
 \label{eq:direct-fibre-total-variation-convergence}
\end{equation}
In Eulerian variables,
\begin{align}
 h_n&\longrightarrow h
 &&\text{strongly in }L^1((0,T)\times\Omega\times D\times Y;M),
 \label{eq:direct-Eulerian-L1-convergence}\\
 \eta[h_n]&\longrightarrow\eta[h]
 &&\text{strongly in }L^2((0,T)\times\Omega;\R^N),
 \label{eq:direct-local-moment-convergence}\\
 K_n&\longrightarrow K_h
 &&\text{in measure and in every finite }L^p,
 \quad K_h=k\mathfrak a(q,\eta[h],y,y').
 \label{eq:direct-kernel-convergence}
\end{align}
The limit \(h\) satisfies the renormalized zero-diffusion kinetic equation
with velocity \(u\) and kernel \(K_h\).  Moreover,
\begin{align}
 (\nabla u_n)qMh_n&\longrightarrow(\nabla u)qMh
 &&\text{in distributions},
 \label{eq:direct-drag-product-convergence}\\
 \tau_Y[h_n]&\longrightarrow\tau_Y[h]
 &&\text{strongly in }L^1((0,T)\times\Omega),
 \label{eq:direct-stress-convergence}
\end{align}
and the entropy, Fisher information, and Jeffreys production satisfy the
corresponding lower-semicontinuity inequalities.  Thus every term in the
state-resolved equation is identified directly along \((u_n,h_n)\); the
limiting density is not defined by lifting a separately selected scalar
limit.

If, in addition, \((u_n,h_n)\) satisfies the momentum equation with residual
\(r_n\to0\) in distributions, \(u_n(0)\to u^{\rm in}\) in \(L^2\), and the
uniform total energy inequality, then \((u,h)\) is a coupled global
energy--entropy weak solution and obeys \eqref{eq:full-drag-coupled-entropy}.
\end{theorem}

\begin{proof}
\textit{Step 1: a Cauchy estimate on moving fibres.}
Lemma~\ref{lem:deformation-history-convergence} makes
\(B_n=\nabla u_n\circ X_n\) Cauchy in \(L^2((0,T)\times\Omega_a)\).  Pulling
\eqref{eq:direct-regularized-kinetic} back by \(X_n\) gives, for almost every
label, the fibre equation in Lemma
\ref{lem:varying-fibre-relative-entropy}, with
\[
 z_{n,a}(t)=\eta[h_n](t,X_n(t,a))
 =\int_{D\times Y}\Phi_aM(H_n-1).
\]
Integrating \eqref{eq:varying-fibre-relative-entropy} in the label yields
\begin{align}
 \sup_{0\le t\le T}\int_\Omega
 \operatorname{Ent}_M(H_n(t,a)\mid H_m(t,a))\dd a
 \le C_T\left[
 \int_\Omega\operatorname{Ent}_M(h_n^{\rm in}\mid h_m^{\rm in})
 +\|B_n-B_m\|_{L^2_{t,a}}^2\right].
 \label{eq:direct-integrated-relative-entropy-Cauchy}
\end{align}
The right-hand side tends to zero.  Fibrewise Pinsker and unit mass therefore
show that \((H_n)\) is Cauchy in
\(C([0,T];L^1_M(\Omega_a\times D\times Y))\), and, more precisely, in the
quadratic total-variation seminorm on the left of
\eqref{eq:direct-fibre-total-variation-convergence}.  Its limit \(H\) is
non-negative and has unit mass on almost every label.  This proves
\eqref{eq:direct-fibre-total-variation-convergence} and defines \(h\) by
push-forward under \(X\), without using a scalar projection.

\textit{Step 2: return to Eulerian variables.}
The difference between \(h_n\) and \(h\) is split into the pullback difference
\(H_n-H\) at the same label and the composition difference generated by
\(X_n-X\).  The first tends to zero by Step 1 and measure preservation.  For
the second, approximate \(H\) in \(L^1_M\) by a bounded function continuous
in \((t,a)\), use \eqref{eq:direct-flow-convergence}, and then remove the
approximation.  This proves \eqref{eq:direct-Eulerian-L1-convergence}.

The same argument applied to the fibrewise total variation gives
\begin{equation}
 \int_0^T\!\int_\Omega
 \left(\int_{D\times Y}M|h_n-h|\right)^2\dd x\dd t\longrightarrow0.
 \label{eq:direct-Eulerian-TV-L2}
\end{equation}
Indeed, the integrand is bounded by four because both fibre masses are one,
so convergence in measure of the composition error upgrades its \(L^1\)
convergence to \(L^2\).  Boundedness of \(\Phi_a\) now proves
\eqref{eq:direct-local-moment-convergence}; the Lipschitz activity law proves
\eqref{eq:direct-kernel-convergence}.

\textit{Step 3: identification of full drag and reaction.}
For a smooth kinetic test \(\varphi\), the part of the drag difference
containing \(\nabla u_n-\nabla u\) is bounded by
\[
 C_\varphi\|\nabla u_n-\nabla u\|_{L^1_{t,x}},
\]
because \(\rho[h_n]=1\) and \(D\) is bounded.  The remaining part is at most
\[
 C_\varphi\|\nabla u\|_{L^2_{t,x}}
 \left\|\int M|h_n-h|\right\|_{L^2_{t,x}},
\]
which tends to zero by \eqref{eq:direct-Eulerian-TV-L2}.  This proves
\eqref{eq:direct-drag-product-convergence}.  Similarly,
\begin{align*}
 \int M|\mathcal R_{K_n}h_n-\mathcal R_{K_h}h|
 &\le2a^*\|k\|_\infty\int M|h_n-h|\\
 &\quad+C\int_\Omega|\eta[h_n]-\eta[h]|\rho[h_n],
\end{align*}
so the reaction flux converges in \(L^1\).  The transport and time terms pass
by \eqref{eq:direct-Eulerian-L1-convergence} and strong convergence of
\(u_n\).  Uniform Fisher information gives weak compactness of
\(\nabla_q\sqrt{h_n}\); the standard square-root argument identifies the
limit with \(\nabla_q\sqrt h\).  Hence \(h\) satisfies the asserted
weak kinetic equation.  Applying the same passage to smooth renormalizations
with bounded first derivative identifies the transport and reaction terms by
the strong convergences above; convex lower semicontinuity retains the
non-negative configurational-diffusion defect.  Removing the renormalization
cutoff gives the asserted renormalized equation.

\textit{Step 4: singular stress and dissipation.}
For fixed \(A\), strong \(L^1_M\) convergence identifies the stress truncated
to \(\{|\nabla U\otimes q|\le A\}\).  The uniform Fisher bound and
\eqref{eq:FENE-stress-tail} make the complementary tails uniformly
\(O(A^{-1})\) in \(L^1_{t,x}\).  Letting first \(n\to\infty\) and then
\(A\to\infty\) proves \eqref{eq:direct-stress-convergence}.  Entropy and
Fisher information are convex and lower semicontinuous.  For the reaction
term, use the finite measure
\[
 k(y,y')M(q)\dd t\dd x\dd q\dd\pi(y)\dd\pi(y')
\]
and apply Lemma~\ref{lem:fixed-measure-Jeffreys-liminf} to the activity
coefficients, which converge in measure and remain in \([a_*,a^*]\).  This
proves the Jeffreys liminf and all kinetic conclusions.

\textit{Step 5: the coupled equation.}
Strong \(L^2_tH^1_x\) convergence passes the convection and viscous terms in
the momentum equation, while \eqref{eq:direct-stress-convergence} passes its
polymeric forcing and \(r_n\to0\) removes the regularization residual.  Weak
time continuity follows from the limiting equation.  Finally take the liminf
in the uniform total energy inequality, using the three lower-semicontinuity
statements from Step 4.  This gives \eqref{eq:full-drag-coupled-entropy} and
completes the proof.
\end{proof}

\begin{corollary}[Viscous-defect criterion]
\label{cor:viscous-defect-criterion}
In Theorem~\ref{thm:direct-state-resolved-compactness}, condition
\eqref{eq:strong-velocity-gradient-convergence} may be replaced by
\begin{align}
 u_n&\longrightarrow u &&\text{strongly in }L^2(0,T;L^2(\Omega)),
 \notag\\
 \nabla u_n&\rightharpoonup\nabla u
 &&\text{weakly in }L^2((0,T)\times\Omega),
 \label{eq:direct-weak-gradient}\\
 \int_0^T\!\int_\Omega|\nabla u_n|^2
 &\longrightarrow\int_0^T\!\int_\Omega|\nabla u|^2.
 \label{eq:no-viscous-defect}
\end{align}
Thus, for a fixed viscosity, direct state-resolved compactness holds exactly
when the weak velocity limit carries no viscous dissipation defect.  The
criterion does not assert that every large-data approximation is defect free.
\end{corollary}

\begin{proof}
The Hilbert-space Radon--Riesz property turns
\eqref{eq:direct-weak-gradient}--\eqref{eq:no-viscous-defect} into strong
\(L^2\) convergence of \(\nabla u_n\).  Together with the first line this is
\eqref{eq:strong-velocity-gradient-convergence}, so Theorem
\ref{thm:direct-state-resolved-compactness} applies.  Conversely, strong
gradient convergence implies \eqref{eq:no-viscous-defect}.
\end{proof}

\begin{remark}[Scope of the direct theorem]
\label{rem:direct-compactness-scope}
Theorem~\ref{thm:direct-state-resolved-compactness} treats Galerkin, cutoff,
or coefficient regularizations of the equation with zero centre-of-mass
diffusion, provided each kinetic equation has the exact conservative form
\eqref{eq:direct-regularized-kinetic}.  It does not claim that a
family with positive diffusivities \(\delta_n\downarrow0\) satisfies
\eqref{eq:no-viscous-defect}; proving that property is a separate singular
limit problem.  Nor does the theorem cover degenerate activities: the
construction theorem allows \(\mathfrak a=0\), but the direct comparison of
two varying jump laws uses \(a_*>0\) in
\eqref{eq:direct-rate-quadratic-bound}.  Finally, the unit-density example in
Proposition~\ref{prop:v38-local-moment-oscillation-obstruction} shows that the
relative-entropy preparation of the initial state fibre cannot be weakened to
mere entropy boundedness.  These are the three explicit boundaries of the
sequential result.
\end{remark}

\begin{remark}[Local activity obstruction]
\label{rem:v37-activity-obstruction}
If \(K_n=k\mathfrak a(\eta[h_n](t,x))\), the reaction part of the defect equation
contains
\begin{equation}
 k\,[\mathfrak a(\eta[h_n])-\mathfrak a(\eta[h])]
 (h_n'-h_n),
 \label{eq:v37-activity-coefficient-defect}
\end{equation}
before any strong local-moment compactness is available.  Weak compactness
gives neither a sign nor a weak--strong product for this term.  The scalar
state-average compactness does not repair the problem because a general
observable depends on the unresolved \(Y\)-distribution.  A separate
fixed-point compactness theorem uniform in the frozen coefficient, or a
monotone activity structure with its own limit theorem, is required for such a
sequential passage.  The direct construction below takes a different route.
Pointwise density-dependent and unbounded state-sensitive moment activities
are still further outside the available compactness class.  State-blind
integrable drivers, including stress, are instead covered constructively by
Theorem~\ref{thm:integrable-driver-activity}.  Section
\ref{sec:v38-global-moment-activity} separates construction from sequential
stability: a labelwise Bielecki--Volterra map directly constructs a general
bounded local coefficient, while a finite-dimensional Volterra map treats
globally averaged
coefficients.  Neither argument asserts strong Eulerian compactness of arbitrary
approximating sequences.  Theorem
\ref{thm:direct-state-resolved-compactness} supplies such compactness only
after the missing information is quantified by the no-viscous-defect
condition \eqref{eq:no-viscous-defect}; Proposition
\ref{prop:v38-local-moment-oscillation-obstruction} shows that no analogous
statement follows from the natural weak bounds alone.
\end{remark}

\section{Conclusion}

The continuum-state system is governed by the scalar FENE dynamics more
rigidly than a specieswise construction suggests.  At zero centre-of-mass
diffusion, state averaging and Lagrangian lifting are inverse maps on the
energy-solution classes.  The continuum of reversible states therefore
preserves the scalar existence and multiplicity structure and introduces no
additional source of large-data nonuniqueness.  The lift remains valid for
full drag, degenerate bounded state-sensitive activities, and state-blind
drivers with integrable growth, including linear dependence on the singular
Kramers stress.

For approximation sequences, the same structure is recovered directly when
the missing compactness has zero price.  Relative-entropy preparation and
equality of viscous dissipation force convergence of the deformation
histories and then strong convergence of the complete state densities.  Full
drag, nonlinear activity, singular stress, and Jeffreys production pass to the
limit without post-limit reconstruction.  The stationary state oscillation
in the appendix shows that entropy bounds and zero viscous defect alone do not
imply this preparation.

Positive centre-of-mass diffusion provides a third closure mechanism: spatial
Fisher compactness identifies finite local moments, while a weighted Hardy
estimate and lower semicontinuity identify the FENE stress and non-atomic
reaction production.  The infinite-rank example confirms that this theorem
is not a finite-species reformulation.  What remains open is exactly what the
theorems leave open: scalar three-dimensional large-data uniqueness, direct
compactness without preparation, and the singular limit from positive to zero
centre-of-mass diffusion.

\appendix

\section{Corotational refinements and compactness obstructions}
\label{sec:v38-global-moment-activity}

The Lagrangian lift closes the full class of bounded finite local-moment
activities described below, despite the compactness obstruction in Remark
\ref{rem:v37-activity-obstruction}.  The distinction is logical: the local
theorem below constructs the coefficient directly on particle labels, whereas
weak Eulerian compactness alone does not identify it along arbitrary
approximating sequences.  We then treat spatially averaged moments and record
their sequential stability.

Let \(N_g<\infty\).  Throughout this section the activity
\(\mathfrak a\) is defined for \(z\in\R^{N_g}\) and satisfies
\eqref{eq:continuum-activity-bounds}--
\eqref{eq:continuum-activity-lipschitz} in that dimension.  Fix arbitrary
observables \(\Phi_a\in L^\infty(D\times Y)\), and define the centred global
moments
\begin{equation}
 \widehat\eta_a[h](t):=\frac1{|\T^2|}
 \int_{\T^2\times D\times Y}\Phi_a(q,y)M(q)(h(t,x,q,y)-1)
 \dd\pi(y)\dd q\dd x,\qquad 1\le a\le N_g.
 \label{eq:v38-global-moments}
\end{equation}
For \(z\in\R^{N_g}\), define
\begin{align}
 K_z(q,y,y')&:=k(y,y')\mathfrak a(q,z,y,y'),
 \label{eq:v38-global-kernel}\\
 (\mathcal R_zh)(q,y)&:=\int_YK_z(q,y,y')
 [h(q,y')-h(q,y)]\dd\pi(y'),
 \label{eq:v38-global-reaction}\\
 (\mathcal J_z\Phi)(q,y)&:=\int_YK_z(q,y,y')
 [\Phi(q,y')-\Phi(q,y)]\dd\pi(y'),
 \label{eq:v38-jump-adjoint}\\
 \mathcal A_z\Phi&:=\frac1{2{\rm Wi}}M^{-1}
 \nabla_q\cdot(M\nabla_q\Phi)+\mathcal J_z\Phi.
 \label{eq:v38-frozen-adjoint-generator}
\end{align}
The effective kernel is \(K_h(t,q,y,y')=K_{\widehat\eta[h](t)}(q,y,y')\).
We assume \(k\) is bounded, non-negative, and symmetric, and that
\(\mathfrak a\) satisfies
\eqref{eq:continuum-activity-bounds}--
\eqref{eq:continuum-activity-lipschitz}.  In particular,
\begin{align}
 0&\le K_z\le \|k\|_\infty a^*,
 \label{eq:v38-kernel-bound}\\
 \|K_z-K_{\widetilde z}\|_{L^\infty(D\times Y^2)}
 &\le \|k\|_\infty A|z-\widetilde z|.
 \label{eq:v38-kernel-lipschitz}
\end{align}
Thus the activity is a bounded Lipschitz function of a finite-dimensional
global variable.

We now state the generator assumption needed to extract the global-moment
evolution law from the weak equation.  It is not used in the local fixed point
of Theorem \ref{thm:local-moment-Lagrangian-closure}.  An observable \(\Phi\)
belongs to the bounded
adjoint-generator closure if it is the graph-norm limit of admissible kinetic
tests \(\Phi^{(m)}\) from Definition \ref{def:v37-zero-diffusion-weak},
independent of \((t,x)\), such that, for every \(R<\infty\),
\begin{align}
 &\Phi^{(m)}\longrightarrow\Phi,\qquad
 q\otimes\nabla_q\Phi^{(m)}\longrightarrow
 q\otimes\nabla_q\Phi
 &&\text{in }L^\infty(D\times Y),
 \label{eq:v38-generator-core-a}\\
 &\sup_{|z|\le R}\|\mathcal A_z\Phi^{(m)}-
 \mathcal A_z\Phi\|_{L^\infty(D\times Y)}\longrightarrow0.
 \label{eq:v38-generator-core-b}
\end{align}
This definition incorporates the natural no-flux realization of
\(M^{-1}\nabla_q\cdot(M\nabla_q\cdot)\); no boundary value of the singular
FENE drift is taken separately.  For each \(a\), assume
\begin{align}
 &\Phi_a\text{ belongs to this closure},\qquad
 \|\Phi_a\|_\infty+\|q\otimes\nabla_q\Phi_a\|_\infty\le C_\Phi,
 \label{eq:v38-generator-assumption-a}\\
 &\sup_{z\in\R^{N_g}}\|\mathcal A_z\Phi_a\|_\infty\le C_G,\qquad
 \|\mathcal A_z\Phi_a-\mathcal A_{\widetilde z}\Phi_a\|_\infty
 \le C_G|z-\widetilde z|.
 \label{eq:v38-generator-assumption-b}
\end{align}
The last Lipschitz bound follows directly from
\eqref{eq:v38-jump-adjoint} when the graph-core approximation is uniform and
\(\mathfrak a\) satisfies \eqref{eq:continuum-activity-lipschitz}.  We state
it explicitly because it is the exact generator closure used below.  Radial
\(\psi_a\) satisfy
\begin{equation}
 (Wq)\cdot\nabla_q\Phi_a=0
 \quad\text{for every skew-symmetric matrix }W.
 \label{eq:v38-corotational-observable}
\end{equation}

\begin{proposition}[A non-empty computable generator class]
\label{prop:computable-generator-class}
Let \(c_a\in L^\infty(Y,\pi)\), and let each \(\psi_a\) be either the
constant one or a radial function in \(C_c^\infty(D)\).  Then
\(\Phi_a(q,y)=c_a(y)\psi_a(q)\) satisfies
\eqref{eq:v38-generator-assumption-a}--
\eqref{eq:v38-generator-assumption-b} and
\eqref{eq:v38-corotational-observable}.  More precisely, with
\[
 \mathcal L_M\psi:=M^{-1}\nabla_q\cdot(M\nabla_q\psi)
 =\Delta_q\psi-\nabla_qU\cdot\nabla_q\psi,
\]
one may take
\begin{align}
 C_{\Phi,a}&=\|c_a\|_\infty
 \bigl(\|\psi_a\|_\infty+
 \|q\otimes\nabla_q\psi_a\|_\infty\bigr),
 \label{eq:explicit-generator-Cphi}\\
 C_{G,a}^{(0)}&=\frac{\|c_a\|_\infty}{2{\rm Wi}}
 \|\mathcal L_M\psi_a\|_\infty
 +2a^*\|k\|_\infty\|c_a\|_\infty\|\psi_a\|_\infty,
 \label{eq:explicit-generator-CG0}\\
 C_{G,a}^{(1)}&=2A\|k\|_\infty
 \|c_a\|_\infty\|\psi_a\|_\infty.
 \label{eq:explicit-generator-CG1}
\end{align}
The common constants in \eqref{eq:v38-generator-assumption-a}--
\eqref{eq:v38-generator-assumption-b} are obtained by taking the maximum of
these three quantities over the finite index set.
Thus the self-consistent theorem contains non-trivial continuum-state feedback
even with \(\psi_a\equiv1\), and it also contains bounded probes of chain
extension supported arbitrarily close to, but a positive distance from, the
FENE boundary.  For example, if \(a_0(q,y,y')\) is bounded, symmetric, and
\(a_0\ge a_0^*>\kappa>0\), then
\begin{equation}
 \mathfrak a(q,z,y,y')=a_0(q,y,y')+
 \kappa\tanh(\lambda\cdot z)
 \label{eq:explicit-nonlinear-activity-example}
\end{equation}
is a genuinely nonlinear member of the admissible activity class, with
\(A\le\kappa|\lambda|\).
\end{proposition}

\begin{proof}
For compactly supported \(\psi_a\), the singular coefficient
\(\nabla_qU\) is bounded on its support, so
\(\mathcal L_M\psi_a\in L^\infty(D)\).  The observable itself is an
admissible kinetic test; hence the constant sequence
\(\Phi_a^{(m)}=\Phi_a\) verifies
\eqref{eq:v38-generator-core-a}--\eqref{eq:v38-generator-core-b}.  For
\(\psi_a\equiv1\), both configurational derivatives vanish and the same
conclusion follows from the constant test in the natural no-flux realization.

Direct substitution gives
\begin{equation}
 \mathcal A_z\Phi_a
 =\frac{c_a\mathcal L_M\psi_a}{2{\rm Wi}}
 +\psi_a(q)\int_YK_z(q,y,y')
 [c_a(y')-c_a(y)]\dd\pi(y').
 \label{eq:explicit-generator-action}
\end{equation}
The kernel bound \eqref{eq:v38-kernel-bound} proves
\eqref{eq:explicit-generator-CG0}; the kernel Lipschitz estimate
\eqref{eq:v38-kernel-lipschitz} proves
\[
 \|\mathcal A_z\Phi_a-\mathcal A_{\widetilde z}\Phi_a\|_\infty
 \le C_{G,a}^{(1)}|z-\widetilde z|.
\]
Equation \eqref{eq:explicit-generator-Cphi} is immediate.  Finally, radiality
gives \((Wq)\cdot\nabla_q\psi_a=0\), proving
\eqref{eq:v38-corotational-observable}.
\end{proof}

For any bounded observables define the centred local moments
\begin{equation}
 \eta_a^{\rm loc}[h](r,x):=
 \int_{D\times Y}\Phi_a(q,y)M(q)(h(r,x,q,y)-1)
 \dd q\dd\pi(y),
 \qquad1\le a\le N_g.
 \label{eq:local-zero-diffusion-moments}
\end{equation}

\begin{theorem}[Direct closure of local finite-moment activity]
\label{thm:local-moment-Lagrangian-closure}
Assume \eqref{eq:continuum-kernel},
\eqref{eq:continuum-activity-bounds}--
\eqref{eq:continuum-activity-lipschitz}, and suppose only that
\(\Phi_a\in L^\infty(D\times Y)\), \(1\le a\le N_g\).  No adjoint-generator
closure or radiality is imposed.  Let the initial data satisfy
\eqref{eq:v37-corotational-scalar-class}--
\eqref{eq:v37-Masmoudi-initial-class}.  Then, for every \(T>0\), there is a
zero-diffusion corotational weak solution, in the sense of Definition
\ref{def:v37-zero-diffusion-weak} with \(\mathcal K=K_h\), such that
\begin{equation}
 K_h(r,x,q,y,y')
 =K_{\eta^{\rm loc}[h](r,x)}(q,y,y')
 \quad\text{for almost every }(r,x,q,y,y').
 \label{eq:self-consistent-local-kernel}
\end{equation}
It has the entropy, configurational Fisher, logarithmic-square, trace, and
identified-stress conclusions of Theorem
\ref{thm:v37-zero-diffusion-existence}, with \(K_0\) replaced by \(K_h\) in
the Jeffreys production.  For every chosen scalar
corotational base solution, the state-resolved lift and its local moment field
are unique, up to null sets, in the renormalized class with bounded number
density.
\end{theorem}

\begin{proof}
Choose the scalar pair \((u,\bar h)\) as in Step 1 of Theorem
\ref{thm:v37-zero-diffusion-existence}, and let \(X(r,\alpha)\) be the regular
Lagrangian flow of \(u\), normalized at zero.  We construct the state-resolved
lift over this fixed base on the whole interval \([0,T]\).

Put
\begin{equation}
 \mathcal Z_T:=L^\infty((0,T)\times\T^2_\alpha;\R^{N_g}),
 \qquad
 \|z\|_\lambda:=
 \mathop{\rm ess\,sup}_{(r,\alpha)}e^{-\lambda r}|z(r,\alpha)|,
 \quad \lambda>0.
 \label{eq:local-Bielecki-space}
\end{equation}
On a finite interval this norm is equivalent to the usual \(L^\infty\) norm,
so \(\mathcal Z_T\) is complete.  For \(z\in\mathcal Z_T\), put
\begin{equation}
 \mathsf K_z(r,\alpha,q,y,y'):=K_{z(r,\alpha)}(q,y,y').
 \label{eq:local-prescribed-label-kernel}
\end{equation}
Changing an \(L^\infty\) representative of \(z\) changes \(\mathsf K_z\) only
on a null set.  The activity bounds therefore make \(\mathsf K_z\) a
well-defined bounded measurable label kernel, so
Proposition \ref{prop:label-dependent-fibre-lift} gives a unique lift
\(\widetilde H_z\) with datum \(h^{\rm in}\).  Define
\begin{equation}
 (\mathcal F_Tz)_a(r,\alpha):=
 \int_{D\times Y}\Phi_aM(\widetilde H_z(r,\alpha)-1)
 \dd q\dd\pi,
 \qquad1\le a\le N_g.
 \label{eq:local-label-moment-map}
\end{equation}

Joint measurability follows from the product-space construction in Proposition
\ref{prop:label-dependent-fibre-lift}.  Since label mass is conserved and
\(M\dd q\dd\pi\) is a probability measure,
\begin{equation}
 |\mathcal F_Tz(r,\alpha)|
 \le C_{\rm obs}(\rho^{\rm in}(\alpha)+1),
 \qquad
 C_{\rm obs}:=\left(\sum_{a=1}^{N_g}\|\Phi_a\|_\infty^2\right)^{1/2}.
 \label{eq:local-moment-map-bound}
\end{equation}
Thus \(\mathcal F_T\) maps \(\mathcal Z_T\) to itself.  Notice that this step
uses no derivative of \(\Phi_a\).

For \(z,w\in\mathcal Z_T\),
\eqref{eq:v38-kernel-lipschitz} gives
\begin{equation}
 \|\mathsf K_z-\mathsf K_w\|_{L^\infty(D\times Y^2)}
 (\theta,\alpha)
 \le \|k\|_\infty A|z-w|(\theta,\alpha).
 \label{eq:local-label-kernel-Lipschitz}
\end{equation}
Boundedness of the observables and
\eqref{eq:labelwise-kernel-stability} then give, for almost every
\((r,\alpha)\),
\begin{align}
 |(\mathcal F_Tz-\mathcal F_Tw)(r,\alpha)|
 &\le C_{\rm obs}\int_{D\times Y}M
 |\widetilde H_z-\widetilde H_w|(r,\alpha)\dd q\dd\pi\\
 &\le C_{\rm loc}\int_0^r|z-w|(\theta,\alpha)\dd\theta,
 \label{eq:local-pointwise-Volterra}\\
 C_{\rm loc}&:=2\|\rho^{\rm in}\|_\infty\|k\|_\infty A C_{\rm obs}.
 \notag
\end{align}
Multiplication by \(e^{-\lambda r}\) and insertion of
\(e^{\lambda\theta}e^{-\lambda\theta}\) in the last integral give
\[
 e^{-\lambda r}\int_0^r|z-w|(\theta,\alpha)\dd\theta
 \le \|z-w\|_\lambda\int_0^r e^{-\lambda(r-\theta)}\dd\theta
 \le \lambda^{-1}\|z-w\|_\lambda.
\]
Consequently,
\begin{equation}
 \|\mathcal F_Tz-\mathcal F_Tw\|_\lambda
 \le \frac{C_{\rm loc}}{\lambda}\|z-w\|_\lambda.
 \label{eq:local-Bielecki-contraction}
\end{equation}
If \(C_{\rm loc}>0\), choose \(\lambda=2C_{\rm loc}\); if
\(C_{\rm loc}=0\), the map is constant.  Banach's theorem gives a unique fixed
point \(z\in\mathcal Z_T\) on the full time interval, without restarting.

Push the fixed point to Eulerian variables by
\(z^E(r,X(r,\alpha))=z(r,\alpha)\).  Proposition
\ref{prop:label-dependent-fibre-lift} gives an Eulerian lift \(h\) with
coefficient \(K_{z^E}\).  The fixed-point identity and
\eqref{eq:label-to-Eulerian-pushforward} give
\[
 z^E(r,x)=\eta^{\rm loc}[h](r,x),
\]
which is exactly \eqref{eq:self-consistent-local-kernel}.  State averaging
recovers \(\bar h\), so \(\tau_Y[h]=\tau[\bar h]\) and the scalar momentum
equation is unchanged.  The logarithmic-square estimate follows fibrewise as
in Step 2 of Theorem \ref{thm:v37-zero-diffusion-existence}.

Finally, let \(h_*\) be any self-consistent renormalized lift over the same
scalar base and pull it back by \(X\).  Its bounded local moment field
\(z_*\) belongs to \(\mathcal Z_T\).  Prescribed-kernel uniqueness in
Proposition \ref{prop:label-dependent-fibre-lift} identifies its pullback with
\(\widetilde H_{z_*}\), so self-consistency says exactly
\(z_*=\mathcal F_Tz_*\).  Uniqueness of the Bielecki fixed point gives
\(z_*=z\), and prescribed-kernel uniqueness then gives \(h_*=h\).
\end{proof}

\begin{remark}[Orientation-sensitive local feedback]
No radiality is hidden in Theorem
\ref{thm:local-moment-Lagrangian-closure}.  Since \(D\) is bounded, it applies,
for example, to the non-radial observables
\[
 \Phi(q,y)=c(y)q_i,
 \qquad
 \Phi_{ij}(q,y)=c(y)\left(q_iq_j-\frac{|q|^2}{2}\delta_{ij}\right),
\]
and hence to bounded Lipschitz activities driven by local orientation and
traceless conformation moments.  These observables need not belong to an
invariant finite-dimensional adjoint-generator space.
\end{remark}

\begin{proposition}[Sequential stability of global moments]
\label{prop:v38-global-moment-time-compactness}
Let \((u_n,h_n)\) satisfy the zero-diffusion weak equation with the
self-consistent kernel \(K_{\widehat\eta[h_n]}\).  Assume, uniformly in
\(n\),
\begin{equation}
 \|\rho[h_n]\|_{L^\infty((0,T)\times\T^2)}\le R_\rho,\qquad
 \|\nabla_xu_n\|_{L^2((0,T)\times\T^2)}\le R_u,\qquad
 \sup_t\int Mh_n\le R_m.
 \label{eq:v38-moment-uniform-bounds}
\end{equation}
If \eqref{eq:v38-generator-assumption-a}--
\eqref{eq:v38-generator-assumption-b} hold, then
\(\widehat\eta[h_n]\) is equibounded and equicontinuous on \([0,T]\), with
\begin{equation}
 |\widehat\eta[h_n](t)-\widehat\eta[h_n](s)|
 \le C\bigl(|t-s|+R_\rho R_u|t-s|^{1/2}\bigr).
 \label{eq:v38-moment-modulus}
\end{equation}
If \eqref{eq:v38-corotational-observable} holds for every \(a\), the
square-root term is absent and the moments are equi-Lipschitz.  If, in
addition,
\begin{equation}
 Mh_n\rightharpoonup Mh
 \quad\text{weakly in }L^1((0,T)\times\T^2\times D\times Y),
 \label{eq:v38-density-weak}
\end{equation}
then, after extraction,
\begin{align}
 \widehat\eta[h_n]&\longrightarrow\widehat\eta[h]
 &&\text{in }C([0,T];\R^{N_g}),
 \label{eq:v38-global-moment-strong}\\
 K_{\widehat\eta[h_n]}&\longrightarrow K_{\widehat\eta[h]}
 &&\text{in }C([0,T];L^\infty(D\times Y^2)).
 \label{eq:v38-global-kernel-strong}
\end{align}
Here \(\widehat\eta[h]\) denotes its continuous representative.
\end{proposition}

\begin{proof}
Use \(\Phi_a^{(m)}\), independent of \((t,x)\), in the weak kinetic
equation on \([s,t]\), and then let \(m\to\infty\) by
\eqref{eq:v38-generator-core-a}--
\eqref{eq:v38-generator-core-b}.  Spatial transport integrates to zero on
\(\T^2\), while diffusion and reaction combine into the frozen adjoint
generator.  Consequently
\begin{align}
 \widehat\eta_a[h_n](t)-\widehat\eta_a[h_n](s)
 =\frac1{|\T^2|}\int_s^t\!\int Mh_n
 \bigl[\mathcal A_{\widehat\eta[h_n]}\Phi_a
 +\mathsf W(u_n)q\cdot\nabla_q\Phi_a\bigr].
 \label{eq:v38-global-moment-weak-evolution}
\end{align}
The generator term is bounded by \(C_GR_m|t-s|/|\T^2|\).  For the drag
term, integrate first in \((q,y)\), use the number-density bound, and apply
Cauchy--Schwarz in \((t,x)\):
\begin{equation}
 \left|\int_s^t\!\int Mh_n\mathsf W(u_n)q\cdot\nabla_q\Phi_a\right|
 \le C_\Phi R_\rho |\T^2|^{1/2}R_u|t-s|^{1/2}.
 \label{eq:v38-global-moment-drag-bound}
\end{equation}
This proves \eqref{eq:v38-moment-modulus};
\eqref{eq:v38-corotational-observable} removes the last integral entirely.
The moment bound follows from boundedness of \(\Phi_a\), mass, and the fixed
centering term.

Arzela--Ascoli gives uniform convergence along a subsequence to a continuous
vector \(z\).  On the other hand, for every \(\theta\in C_c^\infty(0,T)\),
\eqref{eq:v38-density-weak} gives
\begin{equation}
 \int_0^T\theta(t)\widehat\eta_a[h_n](t)\dd t
 \longrightarrow\int_0^T\theta(t)\widehat\eta_a[h](t)\dd t.
 \label{eq:v38-global-moment-distributional-identification}
\end{equation}
Hence \(z=\widehat\eta[h]\) almost everywhere, which identifies the
continuous representative and proves \eqref{eq:v38-global-moment-strong}.
Estimate \eqref{eq:v38-kernel-lipschitz} then gives
\eqref{eq:v38-global-kernel-strong}.
\end{proof}

\begin{theorem}[Self-consistent global-moment activity existence]
\label{thm:v38-self-consistent-global-moment-existence}
Assume \eqref{eq:continuum-kernel} and
\eqref{eq:continuum-activity-bounds}--
\eqref{eq:continuum-activity-lipschitz}, and suppose only that
\(\Phi_a\in L^\infty(D\times Y)\), \(1\le a\le N_g\).  Let the initial data
satisfy \eqref{eq:v37-corotational-scalar-class}--
\eqref{eq:v37-Masmoudi-initial-class}.  Then, for every
\(T>0\), there is a zero-diffusion corotational weak solution, in the sense of
Definition \ref{def:v37-zero-diffusion-weak} with \(\mathcal K=K_h\), with
\[
 K_h(t,q,y,y')=K_{\widehat\eta[h](t)}(q,y,y').
\]
It has all conclusions and bounds of Theorem
\ref{thm:v37-zero-diffusion-existence}; in particular its Kramers stress is
identified and its Jeffreys production, with \(K_0\) replaced by \(K_h\), is
finite.  Its global moments are
bounded and its activity is self-consistent for almost every time.  If the
generator hypotheses \eqref{eq:v38-generator-assumption-a}--
\eqref{eq:v38-generator-assumption-b} also hold, the moment path has an
absolutely continuous representative.  For each
chosen scalar corotational base solution, the state-resolved lift and its
moment path are unique in the renormalized class of Theorem
\ref{thm:lagrangian-state-fibre}.  No strong compactness of the state-resolved
density is asserted.
\end{theorem}

\begin{proof}
Apply the scalar construction used in Step 1 of Theorem
\ref{thm:v37-zero-diffusion-existence} to obtain \((u,\bar h)\).  Reaction is
absent from this scalar problem.  Set
\[
 \mathfrak m:=\int_{\T^2\times D\times Y}Mh^{\rm in}\dd x\dd q\dd\pi.
\]
On
\begin{equation}
 \mathcal Y_T:=L^\infty(0,T;\R^{N_g}),
 \qquad
 \|z\|_\lambda:=\mathop{\rm ess\,sup}_{0<r<T}e^{-\lambda r}|z(r)|,
 \label{eq:global-Bielecki-space}
\end{equation}
use the same Bielecki norm as in the local theorem.  For \(z\in\mathcal Y_T\),
changing its representative on a null set does not change the prescribed
kernel or its lift.  Thus Theorem \ref{thm:lagrangian-state-fibre} gives the
unique lift \(H_z\) with
kernel \(K_{z(r)}\).  Define \(\mathcal G_Tz=\widehat\eta[H_z]\).  Boundedness
of the observables and mass conservation show that \(\mathcal G_T\) maps
\(\mathcal Y_T\) to itself; explicitly,
\[
 |\mathcal G_Tz(r)|\le
 \frac{\mathfrak m+|\T^2|}{|\T^2|}
 \left(\sum_{a=1}^{N_g}\|\Phi_a\|_\infty^2\right)^{1/2}.
\]
Moreover, boundedness of the observables gives first
\[
 |(\mathcal G_Tz-\mathcal G_Tw)(r)|
 \le \frac{C_{\rm obs}}{|\T^2|}
 \int_{\T^2\times D\times Y}M|H_z-H_w|(r),
 \qquad
 C_{\rm obs}:=\left(\sum_{a=1}^{N_g}\|\Phi_a\|_\infty^2\right)^{1/2}.
\]
The kernel estimate \eqref{eq:v38-kernel-lipschitz} and the stability bound
\eqref{eq:lagrangian-kernel-stability} therefore yield
\begin{align}
 |(\mathcal G_Tz-\mathcal G_Tw)(r)|
 &\le C_{\rm glob}\int_0^r|z-w|(\theta)\dd\theta,
 \label{eq:global-pointwise-Volterra}\\
 C_{\rm glob}&:=\frac{2\mathfrak m\|k\|_\infty A C_{\rm obs}}{|\T^2|}.
 \notag
\end{align}
The exponential calculation used for the local map gives, here explicitly,
\[
 e^{-\lambda r}\int_0^r|z-w|(\theta)\dd\theta
 \le \|z-w\|_\lambda
 \int_0^r e^{-\lambda(r-\theta)}\dd\theta
 \le \lambda^{-1}\|z-w\|_\lambda.
\]
Therefore
\begin{equation}
 \|\mathcal G_Tz-\mathcal G_Tw\|_\lambda
 \le \frac{C_{\rm glob}}{\lambda}\|z-w\|_\lambda.
 \label{eq:global-Bielecki-contraction}
\end{equation}
Choosing \(\lambda=2C_{\rm glob}\) when \(C_{\rm glob}>0\), and observing that
\(\mathcal G_T\) is constant when \(C_{\rm glob}=0\), gives a unique fixed
point on \([0,T]\).  Thus \(h=H_z\) has the required self-consistent kernel,
without restarting or passing to a nonlinear approximation limit.

Theorem \ref{thm:lagrangian-state-fibre} supplies entropy, Fisher, Jeffreys,
trace, mass, and number-density bounds on the whole interval.  State averaging
gives \(\int_Yh\dd\pi=\bar h\), hence \(\tau_Y[h]=\tau[\bar h]\), and the
identified stress is inherited from the scalar solution.  The convex-tail
argument in Step 2 of Theorem \ref{thm:v37-zero-diffusion-existence} gives the
logarithmic-square bound.  Under the optional generator hypotheses, testing as
in \eqref{eq:v38-global-moment-weak-evolution} gives, for almost every \(r\),
\begin{equation}
 |\dot z_a(r)|
 \le \frac{C_G\mathfrak m}{|\T^2|}
 +\frac{C_\Phi\|\rho^{\rm in}\|_\infty}{|\T^2|^{1/2}}
 \|\mathsf W(u)(r)\|_{L^2_x}.
 \label{eq:global-fixed-moment-AC-bound}
\end{equation}
The right-hand side belongs to \(L^1(0,T)\), because
\(u\in L^2(0,T;H^1(\T^2))\).  The distributional evolution identity therefore
provides an absolutely continuous representative of \(z\).

Finally, every self-consistent lift over the same scalar base defines an element
of \(\mathcal Y_T\) fixed by \(\mathcal G_T\), by prescribed-kernel uniqueness.
The Bielecki fixed point is unique, and prescribed-kernel uniqueness then
identifies the lift itself.
\end{proof}

\begin{remark}[Exact boundary of the direct fixed point]
The construction is conditional only on the scalar corotational solution
provided by Theorem \ref{thm:v37-zero-diffusion-existence}; it does not assert
uniqueness of that scalar base.  For each chosen base, both Bielecki fixed
points and their state-fibre lifts are unique.  This corotational subsection
uses globally bounded, globally Lipschitz activities and bounded observables;
adjoint-generator closure
is needed only for the additional time-regularity and autonomous-ODE results.
The argument neither takes a positive centre-of-mass diffusivity to zero nor
invokes cutoff-uniform compactness of nonlinear fixed points.  Full-drag local
bounded observables are covered by Theorem
\ref{thm:full-drag-local-closure}, and state-blind unbounded drivers by Theorem
\ref{thm:integrable-driver-activity}; what remains outside this corotational
global-moment result is autonomous closure for orientation-sensitive global
observables under full drag.
\end{remark}

\begin{proposition}[Autonomous ODE subclass and its boundary]
\label{prop:v38-autonomous-closure-boundary}
Put \(\Phi_0=1\) and suppose that
\(\Phi_0,\ldots,\Phi_{N_g}\) are linearly independent; write
\(V=\operatorname{span}\{\Phi_0,\ldots,\Phi_{N_g}\}\).
Suppose every \(\Phi_a\) is radial in \(q\) and, for every \(z\),
\begin{equation}
 \mathcal A_z\Phi_a=\sum_{b=0}^{N_g}C_{ab}(z)\Phi_b,
 \qquad 0\le a\le N_g,
 \label{eq:v38-invariant-generator-space}
\end{equation}
with bounded locally Lipschitz coefficients.  Then, for
\(0\le a\le N_g\), the moments
\(m_a=|\T^2|^{-1}\int Mh\Phi_a\) obey the closed system
\begin{equation}
 \dot m_a=\sum_{b=0}^{N_g}C_{ab}(m_1-\langle\Phi_1\rangle_M,\ldots,
 m_{N_g}-\langle\Phi_{N_g}\rangle_M)m_b,\qquad \dot m_0=0,
 \label{eq:v38-closed-moment-ODE}
\end{equation}
where \(\langle\Phi_a\rangle_M=\int_{D\times Y}M\Phi_a\dd q\dd\pi\).
Conversely, at every strictly positive realizable state, invariance
\eqref{eq:v38-invariant-generator-space} is necessary for an autonomous law
depending only on these moments.  Independence from arbitrary corotational
flows further forces \(Jq\cdot\nabla_q\Phi_a=0\) in two dimensions, hence
radiality on each circle.
\end{proposition}

\begin{proof}
Radiality annihilates the drag term in
\eqref{eq:v38-global-moment-weak-evolution}; substituting
\eqref{eq:v38-invariant-generator-space} gives
\eqref{eq:v38-closed-moment-ODE}.  Since \(\mathcal A_z1=0\), its zeroth row
vanishes.  Standard finite-dimensional continuation applies because the
moments remain bounded by mass and the bounded observables.

For necessity, fix a realizable \(z\) and a bounded density \(h_0\ge c>0\),
and put \(F=\mathcal A_z\Phi_a\).  If \(F\notin V\), its orthogonal component
\(F-P_VF\) is non-zero.  We claim that there is an
\(r\in L^\infty(M\dd q\dd\pi)\cap V^\perp\) with \(\langle F,r\rangle_M\ne0\).
Indeed, truncate \(F-P_VF\) in value, obtaining \(s_n\in L^\infty\) with
\(s_n\to F-P_VF\) in \(L^2(M)\), and set \(r_n=s_n-P_Vs_n\).  The generators
are tested only against bounded observables, so the basis of \(V\) is bounded;
consequently \(r_n\in L^\infty\cap V^\perp\), and
\[
 \langle F,r_n\rangle_M
 =\langle F-P_VF,s_n\rangle_M
 \longrightarrow \|F-P_VF\|_{L^2(M)}^2>0.
\]
Choose one such \(r=r_n\).  If
\(0<\varepsilon<c/\|r\|_{L^\infty}\), the spatially homogeneous densities
\(h_0\pm\varepsilon r\) are non-negative and have identical \(V\)-moments,
whereas their \(a\)-th generator derivatives differ by
\(2\varepsilon\langle F,r\rangle_M\ne0\).  Thus no common autonomous moment
law exists.

For the flow obstruction, let \(J\) be rotation through \(\pi/2\) and put
\(g=Jq\cdot\nabla_q\Phi_a\).  If \(g\ne0\), choose a smooth mean-zero
\(\omega\) on \(\T^2\), solve \(-\Delta\psi=\omega\), and set
\(u=\nabla^\perp\psi\), so \(\mathsf W(u)q\cdot\nabla_q\Phi_a\) is a non-zero
constant multiple of \(\omega g\).  The positive perturbations
\(h_0\pm\varepsilon\omega g\) have the same global moments because
\(\int\omega=0\), whereas their drag contributions differ by a non-zero
multiple of \(\varepsilon\int_{\T^2}\omega^2\int M g^2\).  Hence flow-independent
closure forces \(g=0\).  This is precisely radiality in \(q\) on the ball.
\end{proof}

\begin{proposition}[Local moments retain arbitrary spatial oscillations]
\label{prop:v38-local-moment-oscillation-obstruction}
The bounds in Proposition \ref{prop:v38-global-moment-time-compactness} do
not imply strong compactness of a local moment in
\(L^1((0,T)\times\T^2)\), even when the number density is identically one
and the velocity has no viscous defect.  Indeed, normalize
\(\int_DM\dd q=1\), take
\[
 Y=\{-1,1\}\times\T,\qquad
 \pi=\tfrac12(\delta_{-1}+\delta_1)\otimes\mathcal L^1_{\T},
\]
where \(\mathcal L^1_{\T}\) is normalized Haar measure,
write \(y=(s,\theta)\), and set
\[
 c(y)=s,\qquad
 k((s,\theta),(s',\theta'))=\mathbf1_{\{s=s'\}}.
\]
Thus the Markov process mixes within two non-trivial compact communicating
classes but never changes the sign \(s\).  Fix \(0<\varepsilon<1\), and set
\begin{equation}
 u_n=0,\qquad
 h_n(t,x,q,s,\theta)=1+\varepsilon s\sin(nx_1).
 \label{eq:v38-stationary-spatial-oscillation}
\end{equation}
Then \((u_n,h_n)\) is a stationary zero-diffusion corotational weak solution
for every self-consistent kernel
\[
 K_n(y,y')=k(y,y')a(\eta[h_n]),\qquad
 a(z)=2+\min\{|z|,1\},
\]
because \(h_n(y')=h_n(y)\) whenever \(k(y,y')>0\).  Configurational diffusion
also vanishes, and the state-averaged Kramers stress is constant and isotropic.
All displayed mass, entropy, unit-number-density, configurational-Fisher, and
logarithmic-square bounds are uniform, every time derivative is zero, and
\(\nabla u_n=0\).  Nevertheless the bounded state-sensitive local moment
generated by \(\Phi(q,s,\theta)=s\) is
\begin{equation}
 \eta[h_n](x)=\varepsilon\sin(nx_1),
 \label{eq:v38-local-state-oscillation}
\end{equation}
which converges weakly to zero but has no strongly convergent subsequence in
\(L^1(\T^2)\).

Consequently the bounded, uniformly positive Lipschitz activity is not
identified by weak convergence:
\(a(\eta[h_n])\) converges weakly to
\(2+\varepsilon |\T^2|^{-1}\int_{\T^2}|\sin x_1|\dd x\), not to \(a(0)=2\).
By contrast, the global moment of \eqref{eq:v38-local-state-oscillation} is
identically zero.  The sequence is not Cauchy under
\eqref{eq:direct-well-prepared-initial-data}; hence this condition cannot be
removed from Theorem~\ref{thm:direct-state-resolved-compactness} and replaced
by the natural weak estimates, even when
\eqref{eq:no-viscous-defect} holds identically.
\end{proposition}

\begin{proof}
The state symmetry gives
\[
 \rho[h_n]=\int_{D\times Y}Mh_n=1,\qquad
 \eta[h_n]=\int_{D\times Y}sM(h_n-1)=\varepsilon\sin(nx_1).
\]
The function is independent of \((t,q,\theta)\).  Its time derivative and
configurational diffusion vanish, and \(h_n'-h_n=0\) on the support of \(k\),
so every reaction flux vanishes despite the non-zero within-class kernel.
Since \(u_n=0\), both transport terms vanish as well.  Radial symmetry of
\(D,M,U\) gives
\[
 \int_D\nabla_qU\otimes q\,M\dd q=c_UI,
\]
so state averaging gives \(\tau_Y[h_n]=c_UI\), which is invisible to every
divergence-free momentum test.  Positivity
\(1-\varepsilon\le h_n\le1+\varepsilon\) gives the uniform entropy and
logarithmic-square bounds, while \(\nabla_qh_n=0\) gives zero Fisher
information.

If a subsequence of \(\sin(nx_1)\) converged strongly in \(L^1\), its weak
limit would be zero, while its \(L^1\)-norm is independent of \(n\), a
contradiction.  Boundedness of \(s\) shows that strong \(L^1_M\) convergence
of \(h_n\) would imply strong convergence of \(\eta[h_n]\), so the state
densities have no strongly convergent subsequence.  Pinsker's inequality then
shows that they cannot be Cauchy in the directed relative entropy of
\eqref{eq:direct-well-prepared-initial-data}.  Finally, for every continuous
periodic test \(g\), partition
the \(x_1\)-circle into \(n\) cells and rescale each cell; uniform continuity
of \(g\) shows
\[
 \int_{\T^2}g(x)|\sin(nx_1)|\dd x
 \longrightarrow \frac1{|\T^2|}\int_{\T^2}|\sin x_1|\dd x
 \int_{\T^2}g(x)\dd x.
\]
This proves the asserted weak activity limit.
\end{proof}

\section*{Statements and Declarations}

\paragraph{Funding.}
This work was supported by NSFC Grant 12501602, Hunan Provincial Education
Department Grant 24C0055, Hunan Provincial Science and Technology Department
Grant 2025JJ60052, and Xiangtan University Start-up Fund Grant KZ0810769.

\paragraph{Competing interests.}
The author declares that there are no competing interests.

\begingroup
\small
\raggedright
\setlength{\bibsep}{0pt plus 0.2ex}
\setlength{\bibhang}{0pt}
\bibliographystyle{unsrtnat}
\bibliography{paper_III_references_v56}
\endgroup

\end{document}